\documentclass[11pt]{article}
\usepackage{amsmath}
\usepackage{amsthm}
\input{epsf.tex}
\newtheorem{theorem}{Theorem}[section]
\newtheorem{proposition}[theorem]{Proposition}
\newtheorem{lemma}[theorem]{Lemma}
\newtheorem{corollary}[theorem]{Corollary}
\newtheorem{define}[theorem]{Definition}

\def\Empty{}

\catcode`\@=11
\def\section{\@startsection {section}{1}{\z@}{-3.5ex plus -1ex minus 
-.2ex}{2.3ex plus .2ex}{\large\bf}}

\def\fnum@figure{{\small Figure \thefigure}}
\def\fakefigure{\def\@captype{figure}}

\long\def\@makecaption#1#2{
    \vskip 10pt 
    \def\FCap{#2} \def\NoCap{\ignorespaces}
    \ifx \FCap\NoCap
       \setbox\@tempboxa\hbox{#1}  % This is to avoid the damn colon.
      \else
       \setbox\@tempboxa\hbox{#1: \small \it #2}
    \fi
    \ifdim \wd\@tempboxa >\hsize   % IF longer than one line:
        \unhbox\@tempboxa\par      %   THEN set as ordinary paragraph.
      \else                        %   ELSE  center.
        \hbox to\hsize{\hfil\box\@tempboxa\hfil}  
    \fi}

\def\@oddhead{\hbox{}\rightmark \hfil \rm\thepage}% Right heading.
\def\sectionmark#1{\markright {\sc{\ifnum \c@secnumdepth >\z@
      \S\thesection.\hskip 1em\relax \fi #1}}}

\catcode`\@=12

\def\oplabel#1{
  \def\OpArg{#1} \ifx \OpArg\Empty {} \else
  	\label{#1}
  \fi}

\newlength{\saveu}

\input{serg.tex}

\def\centeredepsfbox#1{\centerline{\epsfbox{#1}}}

\begin{document}

\title{QUASIGEODESIC PSEUDO-ANOSOV FLOWS in HYPERBOLIC 3-MANIFOLDS
and CONNECTIONS with  LARGE SCALE GEOMETRY
{\footnote{MSC-class: Primary: 57M50, 37D20, 37D50, 37C85, Secondary: 20F67, 34D45, 57M60, 57R30}}}
\author{S\'{e}rgio R. Fenley
\thanks{Research partially supported by a grant of the Simons 
foundation.}}
\maketitle

%\centerline{{\large{PRELIMINARY VERSION}}}
%\vskip .1in

{\small{
\noindent
{\bf {Abstract}} $-$ 
In this article we obtain a simple topological and dynamical systems
condition which is necessary and sufficient for an arbitrary
 pseudo-Anosov flow in a closed,
 hyperbolic three manifold to be quasigeodesic.
Quasigeodesic means that orbits are efficient in measuring
length up to a bounded multiplicative distortion when lifted
to the universal cover. 
We prove that such flows are quasigeodesic if and only if there is
an upper bound, depending only on the flow,
on the number of orbits which are 
freely homotopic to an arbitrary closed orbit of the flow.
The main ingredient is a proof that,
under the boundedness condition,
the fundamental group of the manifold acts as a uniform
convergence group on a flow ideal boundary of the
universal cover. 
We also construct a flow ideal compactification of the universal cover,
 and prove that it is equivariantly
homeomorphic to the Gromov compatification. This implies the
quasigeodesic behavior of the flow.
The flow ideal boundary and flow ideal compactification
are constructed using only the structure
of the flow. 
%This main ingredient implies that good flow properties
%are transferred to good properties in the Gromov
%ideal compactification of the universal cover and this implies
%the quasigeodesic behavior of the flow.
}}

\section{Introduction}

The goal of this article is to relate dynamical systems behavior 
with the geometry of the underlying manifold, and in particular with
the large scale geometry of the universal cover.
We analyse pseudo-Anosov flows in three manifolds.
This is an extremely common class of flows which is known to be closely related
with the topology of the underlying manifold. We study these flows
in hyperbolic  three manifolds 
%and consider geometric properties of the
%flow in the universal cover. 
and we obtain a very simple characterization
of good geometric behavior of the flow in the universal cover.
We also prove that in the case of good geometric behavior
 the flow generates a flow ideal
compactification of the universal cover which is equivariantly 
homeomorphic to the Gromov compactification.
It follows  that in theses cases the flow encodes
the asymptotic or large scale geometric structure
of the universal cover.

The field of hyperbolic flows was started by Anosov \cite{An} who
studied geodesic flows in the unit tangent bundle of manifolds with
negative sectional curvature.  
In fact Anosov studied much more general flows, which have since then
been called Anosov flows.
Anosov obtained deep and far reaching results concerning the 
dynamical behavior of these flows and with connections
and applications to ergodic theory, foliation theory and other areas \cite{An}.
In dimension three
these flows were generalized  by Thurston who defined pseudo-Anosov 
homeomorphisms of surfaces \cite{Th4} and used their suspensions
to obtain deep results about three manifolds that fiber over the
circle \cite{Th1,Th2,Th3}. These suspension flows 
are the most basic examples of pseudo-Anosov flows.

Pseudo-Anosov flows are the most useful flows to study the
topology of three manifolds \cite{GK1,Mo3,Mo4,Mo5,Cal1,Cal2,
Cal3,Ba2,Fe6,Fe7}. The goal of this
article is to establish a strong relationship between these flows 
and the geometry of the manifold, and more specifically with the large scale geometry
of the universal cover. This is extremely important in the case of
hyperbolic manifolds
\cite{Th1,Th2,Gr,Gh-Ha}. At first it might seem that nothing can be said
in general
about large scale geometric properties of flows lines. This is because
flows and flow lines are very flexible and floppy and aparently not
very geodesic. We will give a necessary and sufficient topological
and dynamical systems condition for all the flow lines in the universal cover
to have good geometric behavior.
%be quasigeodesic.

Zeghib \cite{Ze} proved that a flow in a closed hyperbolic $3$-manifold
cannot be geodesic, that is, not all flow lines can be geodesics in 
the hyperbolic metric. The next best property is that flow lines
are quasigeodesics, which we now define. A {\em quasi-isometric embedding}
is a map between metric spaces which is bi-Lipschitz in the large.
Equivalently, up to an additive constant, the map is at most a bounded
multiplicative distortion in the metric. A {\em quasigeodesic} is a quasi-isometric
embedding of the real line (or a segment) into a metric space.
The work of Thurston \cite{Th1,Th2,Th3}, Gromov \cite{Gr} and
many, many others have thoroughly established the fundamental importance
of quasigeodesics in hyperbolic manifolds. In this article we analyse the
interaction of the quasigeodesic property with flows in $3$-manifolds. 

Given a flow with rectifiable orbits  in a manifold  we say it is quasigeodesic
if every flow line of the lifted flow to the universal cover is
a quasigeodesic. The metric in the domain of a flow line is the path
metric along the orbit. For the remainder of this article we will only
consider flows in closed $3$-manifolds and their lifts to covering spaces.
The first example of a quasigeodesic flow in a hyperbolic manifold was 
that of suspensions of pseudo-Anosov homeomorphisms of closed
surfaces \cite{Th3,Bl-Ca}. In a seminal work,
Cannon and Thurston \cite{Ca-Th} showed in $1984$ that
the quasigeodesic property holds for these
flows and used this property to prove the sensational result that lifts of
fibers to the universal cover extend continuously to the sphere at infinity and
produce group invariant Peano or sphere filling curves. 
After 
the Cannon-Thurston result, Zeghib \cite{Ze} gave a very elementary proof that 
all suspensions on closed manifolds are quasigeodesic $-$ because of the
minimal separation between lifts of fibers. 
Given the Cannon-Thurston result,
the natural question arises: when is a pseudo-Anosov flow in a hyperbolic
$3$-manifold a quasigeodesic flow? 
Over the last 25 years the quasigeodesic property has been proved
in several special circumstances.
In this article we give a 
{\underline {complete}}
and very simple characterization of the quasigeodesic property.

Around the same time as the Cannon Thurston result, \
Goodman \cite{Go} and Fried \cite{Fr} produced constructions
of new Anosov and pseudo-Anosov flows via Dehn surgery on closed orbits of
Anosov or pseudo-Anosov
flows, or  Dehn surgery near closed orbits. In general most Dehn surgeries yield new Anosov or
pseudo-Anosov flows in the surgered manifold, and in the majority of  cases all surgeries
except for the longitudinal one yield such flows \cite{Fr}. This vastly increased the
class of pseudo-Anosov flows and the known classes of manifolds supporting
pseudo-Anosov flows.
In this article we consider Anosov flows as a subclass of pseudo-Anosov flows.
Anosov flows are the pseudo-Anosov flows without singular orbits.
In addition many classes of Reebless foliations \cite{No,Ga1,Ga2,Ga3} in atoroidal
manifolds admit transverse or almost transverse pseudo-Anosov flows (see below):
this has been proved for $\rrrr$-covered foliations \cite{Fe6,Cal1},
foliations with one sided branching \cite{Cal2}, and finite depth 
foliations \cite{Mo5}. It is quite possible that every Reebless foliation
in an atoroidal manifold admits an almost
 transverse pseudo-Anosov flow \cite{Th5,Th6,Cal3}.
The conclusion is that pseudo-Anosov flows are extremely common amongst
$3$-manifolds.

We now describe further classes of pseudo-Anosov flows in hyperbolic manifolds
which were previously
shown to be quasigeodesic. The next result after Cannon and Thurston
was obtained by Mosher \cite{Mo2}
who constructed an infinite class of quasigeodesic pseudo-Anosov flows transverse to depth
one foliations \cite{Ga1,Ga2}. He used the round handles of Asimov, 
attached them to $I$-bundles over surfaces with boundary,  and did 
further glueing to produce flows in closed manifolds \cite{Mo2}. 
The closed surfaces in the construction were quasi-Fuchsian
\cite{Th1,Th2}.
This means that lifts to the universal cover are quasi-isometrically embedded, with
the path metric in the domain. The proof that such surfaces are quasi-Fuchsian
used some very deep results of Thurston \cite{Th1,Th2} and Bonahon \cite{Bon}.
Shortly after that Mosher \cite{Mo3,Mo4} formalized the concept of 
a pseudo-Anosov flow in a $3$-manifold.
Later the author and Mosher \cite{Fe-Mo} proved that a pseudo-Anosov flow
almost transverse to a finite depth foliation in a closed hyperbolic $3$-manifold is
quasigeodesic. The proof depended in an essential way on the geometric properties of the 
leaves of the foliation and a hierarchy of the manifold associated with the finite
depth foliation \cite{Ga1,Ga2,Ga3}.
  Almost transverse means that after a possible blow up of some singular
orbits of the flow, the flow becomes transverse to the foliation
\cite{Mo1,Mo3}.

Later Thurston \cite{Th5} proved that a regulating pseudo-Anosov flow transverse
to a foliation coming from a slithering is quasigeodesic. 
Slithering is essentially equivalent to the following:
any pair of leaves in the universal cover are a bounded distance from each 
other. The bound depends on the pair of leaves, but not on the
individual points in the leaves. Regulating means that 
in the universal cover an arbitrary orbit intersects every leaf of the lifted foliation.
The proof of the quasigeodesic property in this situation is quite simple, and
very similar to the straightforward proof of Zeghib that suspensions are
quasigeodesic. 

In all of these results the flow is transverse, or almost transverse, to a foliation
which has excellent geometric properties. In particular except for the case of
slitherings, all foliations above
have compact leaves. 
Such a compact leaf is quasi-Fuchsian, with excellent geometric properties \cite{Th1,Th2}.
This helped tremendously in the proof of quasigeodesic behavior for the flow.
Notice however that any compact leaf is an incompressible
surface. Unfortunately most $3$-manifolds do not have incompressible
surfaces \cite{Ha-Th} so this method to prove quasigeodesic behavior for the 
flows is somewhat restricted.

To deal with more general pseudo-Anosov flows a completely different method 
is required. In this article we use an alternate method which does not require the existence
of a transverse foliation, or any foliation at all. In addition the method does
not assume that the manifold $M$ is hyperbolic or even atoroidal.
Very roughly the method is as follows. 
Suppose that there is a bound on the size of sets of freely homotopic closed orbits.
We use the lifted flow to $\mi$
to produce a flow ideal boundary to the universal cover $\mi$. Then we show that the
fundamental group of the manifold acts as a uniform convergence group on the
flow ideal boundary. By a result of Bowditch
\cite{Bow1} this implies that the fundamental group is Gromov hyperbolic
and the flow ideal boundary is $\pi_1(M)$ equivariantly homeomorphic
to the Gromov ideal boundary $\si$. 
We then show that a natural
flow ideal compactification
and the Gromov compactification of $\mi$ are $\pi_1(M)$ equivariantly homeomorphic.
The flow ideal compactification is constructed using only the stable and unstable
foliations of the flow. This shows that in the bounded case the
flow encodes the asymptotic or large scale geometric structure of $\mi$.
One crucial implication is that properties that hold in the flow compactification
get transferred to the Gromov compactification. We then prove three easy properties
in the flow compactification which, when transferred to the Gromov compactification,
imply that the flow is quasigeodesic \cite{Fe-Mo}. This is the basic idea, but
there are some substantial complications as described below.

This method was previously employed by the author in \cite{Fe9} in the
particular case that the flow does not have {\em perfect fits}. This is a technical
condition. Roughly, a perfect fit is a pair of a stable leaf and an unstable
leaf in $\mi$ which do not intersect, but which are essentially ``asymptotic".
Pseudo-Anosov flows without perfect fits are much simpler to analyse than the
general case and they behave to an enormous extent like suspension pseudo-Anosov
flows with respect to the issues in question in this article. For example the leaf spaces
of the stable and unstable foliations (in the universal cover) are Hausdorff.
In addition all the flow lines in the universal cover essentially go in the same 
direction.
These properties tremendously simplify the analysis as will be very clear in this
article. 
%In particular many features are not present in the situation of
%perfect fits and some substantial parts of the strategy of the simple
%situation do not work in the general case.
The goal of this article is to analyse the quasigeodesic property for
general pseudo-Anosov flows, particularly in atoroidal manifolds.

But we have to be careful. As it turns out not every every pseudo-Anosov flow
in a hyperbolic manifold is quasigeodesic. Twenty years ago the author produced
a large class of Anosov flows in closed, hyperbolic $3$-manifolds 
which are not quasigeodesic \cite{Fe2}.
In these flows every closed orbit of the flow is freely homotopic
to infinitely many other closed orbits. Lifting coherently to the universal
cover they all have the same ideal points in the sphere at infinity. In addition
they cannot be very close to each other because of the pseudo-Anosov property, so they
cannot accumulate in $\mi$. Since quasigeodesics in such manifolds are a bounded
distance from a geodesic \cite{Th1,Th2,Gr}, this implies that the flow cannot be uniformly
quasigeodesic. Uniform means the same bounds work for all orbits. But a pseudo-Anosov
flow in an atoroidal manifod is transitive \cite{Mo3} so quasigeodesic implies
uniformly quasigeodesic. Alternatively a result of Calegari  showed
that if a flow of any type in a hyperbolic manifold is quasigeodesic, then it
is uniformly quasigeodesic \cite{Cal4}. This shows that these flows are not quasigeodesic.

Since not all pseudo-Anosov flows in hyperbolic manifolds are quasigeodesic,
one must be careful to determine which ones are quasigeodesic, or what properties
are equivalent or imply the quasigeodesic behavior.
We were able
to obtain an extremely simple condition which is equivalent to quasigeodesic
behavior in all situations. First we need a clarification and a   definition.

The clarification needed here is the following. For $3$-manifolds supporting
a pseudo-Anosov flow, Perelman's fantastic results \cite{Pe1,Pe2,Pe3}
imply that if the manifold $M$ is atoroidal then it is in fact 
hyperbolic and consequently the fundamental group $\pi_1(M)$
is Gromov hyperbolic. Hence all three properties are equivalent. In the method
presented in this article we will not assume Perelman's results. We will only
assume a certain dynamical systems property of the flow which implies
that $M$ is atoroidal and through the results of this article, this
property implies that $\pi_1(M)$ is Gromov hyperbolic.

\begin{define}{}{}
The free homotopy class of a closed orbit is the set of of orbits 
which are freely homotopic to it. 
We say that a pseudo-Anosov flow in  a closed $3$-manifold is
bounded if:

i) No closed orbit is non trivially freely homotopic to itself and
there is an upper bound to the cardinality of free homotopy classes.

ii) The flow is not topologically conjugate to a suspension Anosov flow.
\end{define}

Topologically conjugate means that there is a homeomorphism sending
orbits to orbits. Suspension Anosov flows have free homotopy classes
which are all singletons. There are many reasons why suspension
Anosov flows are special and they
need to be treated  separately.

A trivial free homotopy from a closed orbit $\beta$ to itself is one
that can be deformed rel boundary to another homotopy with image
contained in $\beta$. A simple example of a non trivial free homotopy
occurs in a geodesic flow $\Phi$ of a closed hyperbolic surface $S$.
Let $\alpha_1$ be a closed orbit of $\Phi$ corresponding
to a closed geodesic $\alpha$ in $S$. Turn the unit tangent vectors
along $\alpha$ continuously by a total turn of $2 \pi$. This 
is a non trivial free homotopy from $\alpha_1$ to itself.

The main result of this article is the following:

\vskip .1in
\noindent
{\bf {Main theorem}} $-$ Let $\Phi$ be a pseudo-Anosov flow in $M^3$ with
Gromov hyperbolic fundamental group. Then $\Phi$ is a quasigeodesic
flow if and only if $\Phi$ is bounded.
\vskip .1in

This theorem answers a question that was  open for almost thirty years \cite{Ga4,Ga5}.

The Main theorem  gives a surprisingly simple and compact characterization of
quasigeodesic behavior. The characterization involves only topology
and the dynamical properties of the flow and it is checked directly
in the manifold as opposed to an analysis in the universal cover.
There are many situations where one can actually check whether
orbits are freely homotopic to other orbits \cite{Mo2,Fe5,Fe8,Fe9}.
The main theorem also has applications to other problems:
\ i) The quasigeodesic property can be used to compute Thurston
norms of surfaces \cite{Mo3,Mo4,Cal4}; \ ii) The quasigeodesic
property can be used to prove the continuous extension property
for foliations as follows. Suppose that $\fol$ is a foliation
in $M^3$ closed, hyperbolic and that there is a quasigeodesic
pseudo-Anosov flow almost transverse to $\fol$. Then $\fol$ has
the continuous extension property \cite{Fe8}. This means that in
the universal cover the leaves of the lifted foliation $\fn$
extend continuously to the sphere at infinity.  We will expand
on this in the final section entitled Concluding remarks.

\vskip .2in
\noindent
{\underline {Strategy of proof of the Main theorem}}

One direction in the proof is fairly simple and it was already alluded to
previously. If $\Phi$ is quasigeodesic, then $\Phi$ is uniformly 
quasigeodesic \cite{Mo3,Cal4}. If a closed orbit is non trivially
freely homotopic to itself, this produces a $\pi_1$-injective
map of ${\bf Z}^2$ into $M$ contradicting that $M$ is atoroidal.
The previous explanation about the examples of
 non quasigeodesic flows in hyperbolic manifolds, shows that 
if $\Phi$ is quasigeodesic then
a free homotopy class cannot be infinite. In fact since $\Phi$
is uniformly quasigeodesic, the same arguments show that 
a free homotopy class has bounded cardinality and this proves
one direction of the Main theorem.

The other direction of the Main theorem is very complex and long.
It will roughly go as follows: bounded free homotopy classes
imply bounded length of chains of perfect fits and this property
implies (after
a lot of work)  that $\Phi$ is quasigeodesic. 
In fact we prove:

\vskip .1in
\noindent
{\bf {Theorem A}} $-$ Let $\Phi$ be a bounded pseudo-Anosov flow
in $M^3$ closed. Then $\pi_1(M)$ is Gromov hyperbolic and $\Phi$
is quasigeodesic. 
\vskip .05in

Notice that Gromov hyperbolicity of $\pi_1(M)$ is not part of the
hypothesis of Theorem A or even that $M$ is atoroidal. So in particular
this provides an alternative proof of Gromov hyperbolicity 
in the setting of flows.

At this point we give a further explanation of  what a perfect
fit is. Let $\ls, \lu$ be the stable and unstable foliations of the
flow $\Phi$ and let $\wls, \wlu$ be the respective lifts to the
universal cover $\mi$.

We say that a stable leaf $L$ of $\wls$ makes 
a {\em perfect fit} with an unstable leaf $U$ of $\wlu$ if $L$ and $U$
do not intersect, but almost intersect in the following sense.
If $L'$ is a stable leaf very near $L$ and in the component
of $\mi - L$ which contains $U$, then $L'$ intersects $U$.
In the same way for $U'$ near $U$ in the ``side of $L$" will
intersect $L$. See formal definition and fig. \ref{cv1}, a in the
Background section. 

Why perfect fits? The first remark is that free homotopies 
of closed orbits always generate perfect fits. First we introduce
the notion of a lozenge. A {\em lozenge} in $\mi$ is made up of 4 leaves,
2 of which $L_1, L_2$ are stable ($\wls$) and 2 of which
$U_1, U_2$ are unstable ($\wlu$).  The leaves $L_1, U_1$ make a perfect
fit as do $L_2, U_2$. In addition $L_1, U_2$ intersect each other
as do $L_2, U_1$. These four leaves form an ``ideal quadrilateral"
in $\mi$ with 2 finite corners $-$ the orbits $U_1 \cap L_2$,
$U_2 \cap L_1$ and two ``ideal" corners corresponding to the 
perfects $L_1, U_1$ and $L_2, U_2$. Again see formal definition
and fig. \ref{cv1}, b in the Background section.

If two closed orbits $\alpha, \beta$ of $\Phi$ are freely homotopic,
then coherent lifts $\widetilde \alpha, \widetilde \beta$ are 
connected by a finite chain of lozenges with initial corner
$\widetilde \alpha$, final corner $\widetilde \beta$ and consecutive
lozenges having a corner orbit in common \cite{Fe4,Fe6}. Hence free homotopies 
generate many perfect fits. We first prove 
that the converse is also true:

\vskip .1in
\noindent
{\bf {Theorem B}} $-$ Let $\Phi$ be an arbitrary pseudo-Anosov flow.
Suppose that $L, U$ make a perfect fit. A study of  the 
asymptotic behavior of orbits in $L, U$ produces free homotopies 
between closed orbits of $\Phi$.
\vskip .05in

The proof is done using a limiting argument going forward in 
the stable leaf or backwards along the unstable leaf and using
the shadow lemma for pseudo-Anosov flows \cite{Han,Man}. This result says that 
the topological structure of $\wls, \wlu$ in the universal
cover implies certain   topological properties of closed orbits in the manifold.
This result does not assume that $M$ is atoroidal or that $\Phi$
is bounded.

One crucial part of the strategy to prove the Main theorem is
to extend Theorem B to chains of perfect fits.
A {\em chain of perfect fits} is a collection of distinct leaves
of $\wls, \wlu$ so that consecutive leaves make a perfect fit.
The length of the chain is the number of leaves in it.
We next prove the following:

\vskip .1in
\noindent
{\bf {Theorem C}} $-$ Suppose that a pseudo-Anosov flow 
$\Phi$ does not have a closed orbit which is non trivially freely homotopic
to itself. Suppose in addition that $\Phi$ has a chain of perfect fits of length $k$. This produces
a free homotopy class of size at least $k$.
\vskip .05in

This is obtained by a shadowing procedure where a perfect fit
may produce more than one free homotopy, that is, a free homotopy
class with more than two orbits. This happens because in the
limiting procedure a sequence of perfect fits may  converge to a finite collection
of lozenges and not only to a single lozenge. This happens because of the 
possible non Hausdorffness in the leaf spaces of $\wls, \wlu$ and it is one
 of the many complications that occur when
there are perfect fits.

We stress that theorems B and C do not assume anything about
$M$ or the flow $\Phi$.
The proof of Theorem $C$ depends mostly on the fact that if
a free homotopy lifts to a single lozenge (as opposed to 
a finite chain  with more than one lozenge), 
then the homotopy has bounded thickness. This means that  the 
homotopy moves every point a bounded amount. The bound depends
only on $M$ and the flow $\Phi$. We already alluded to the
fact that free homotopies cannot move points an arbitrary
small distance $-$ because of the pseudo-Anosov property.
Theorem C is related to the fact that ``indivisible" free homotopies do not
move points too much. This gives another substantial 
interaction between topology, dynamics on the one hand and
the metric and geometry on the other hand. 

After the result of Theorem C we can rephrase Theorem A as 
follows:

\vskip .1in
\noindent
{\bf {Theorem D}} $-$ Let $\Phi$ be a pseudo-Anosov flow with
an upper bound on the size of chains of perfect fits. Assume
that $\Phi$ is not topologically conjugate to a suspension
Anosov flow. Then $\pi_1(M)$ is Gromov hyperbolic and
$\Phi$ is a quasigeodesic flow.
\vskip .05in

Suspension Anosov flows do not have any perfect fits.
Geodesic flows on the other hand have free homotopy classes
with two elements only, but they lift to infinite chains
of perfect fits (or lozenges) $-$ the corners of the lozenges are equivalent
by certain covering translations.

Theorem D is stated in the format  involving information
in the universal cover and the topological structure of
$\wls, \wlu$. To prove Theorem D we construct and analyse the
flow ideal boundary $\rr$ of $\mi$. This flow ideal boundary
is obtained as a quotient of the boundary of the  orbit space as follows.

The boundary of the orbit space for general pseudo-Anosov
flows was constructed in \cite{Fe9}. Let $\oo$ be the orbit
space of $\wwp$, that is, the quotient space $\mi/\wwp$. 
A basic result is that $\oo$ is always homeomorphic to the
plane $\rrrr^2$ \cite{Fe-Mo}. Since the stable/unstable foliations 
$\wls, \wlu$ are flow invariant they induce one dimensional,
possibly singular, 
foliations $\oos, \oou$ in the orbit space $\oo$. 
Using only these foliations one produces the ideal boundary $\partial \oo$ of 
$\oo$ and there is a natural topology in $\cd = \oo \cup \partial \oo$
turning it into a closed disk \cite{Fe9}. Hence $\partial \oo$ is a circle
and in addition $\pi_1(M)$ naturally acts on $\oo, \partial \oo$
and $\cd$. One fundamental fact is that if a leaf $u$ of $\oou$ makes
a perfect fit with a leaf $s$ of $\oos$, then the corresponding
ideal points of $s, u$ are the same point in 
the ideal boundary $\partial \oo$.
These constructions have no restriction on $M$ or $\Phi$.

The flow ideal boundary $\rr$ is obtained from $\partial \oo$ 
by identifying $p, q$ in $\partial \oo$ if they are ideal points of a leaf
of either the stable or unstable foliation $\oos$ or $\oou$.
We first prove the following:

\vskip .1in
\noindent
{\bf {Theorem E}} $-$ Let $\Phi$ be a bounded pseudo-Anosov flow.
Then the flow ideal boundary $\rr$ is homeomorphic to a
two dimensional sphere. 
\vskip .05in

The bounded hypothesis is crucial. In the non quasigeodesic 
examples in hyperbolic $3$-manifolds mentioned before,
 every closed orbit is freely homotopic
to infinitely many other closed orbits. This lifts to infinite
chains of perfect fits. Since ideal points of rays associated
to perfect fits are the same point in $\partial \oo$, this
produces an infinite to one identification 
of points from $\partial \oo$ to $\rr$. In these examples $\rr$ is
the union of a circle and two points $x, y$. The two points
$x, y$ are not separated from any point in the circle. 
Hence $\rr$ is not metrisable and cannot be homeomorphic
to the Gromov ideal boundary of a Gromov hyperbolic group \cite{Gr}.

The most important ingredient in the proof of  Theorem D
and hence the Main theorem is the following.

\vskip .1in
\noindent
{\bf {Theorem F}} $-$ Suppose that $\Phi$ is a bounded pseudo-Anosov flow.
Then $\pi_1(M)$ acts as a uniform convergence group on the flow
ideal boundary $\rr$.
\vskip .1in

To prove Theorem F we first show that $\pi_1(M)$ acts
as a convergence group on $\rr$. This means that if $(g_n)$
is a sequence of distinct elements of $\pi_1(M)$, there is
a subsequence $(g_{n_k})$ with a source $z$ and sink $w$ in $\rr$.
This means that if $C$ is a compact set in $\rr - \{ z \}$ \ then
the sequence
$(g_{n_k}(C))$ converges uniformly to $\{ w \}$ in the Gromov-Hausdorff
topology of closed sets of $\rr$.
The biggest  difficulty in proving this is that the existence of perfect
fits means that many points in $\partial \oo$ are identified
when projected to  $\rr$. For example there may be leaves in $\wls$ (or $\oos$)
which are non separated from each other. In particular this produces
at least two perfect fits, see Theorem \ref{theb}.   The non Hausdorff behavior implies that
 in the limiting
arguments, collections of leaves of $\oos$ may converge to more than
one limit leaf and new identifications of points of $\partial \oo$
emerge. This ends up being tricky to deal with and the proof
is complex. The bounded
hypothesis is used many times in the proof.

The second part of the proof of Theorem F is to prove that the
action is uniform. With the previous properties, it suffices to
show that an arbitrary point $z$ in $\rr$ is a conical limit point
for the action of $\pi_1(M)$ on 
$\mathcal R$  \cite{Bow1,Bow2}. This means that there is a sequence
$(g_n)$ in $\pi_1(M)$ with source $z$, \ sink $w$, and  with $w$
 {\underline {distinct}} from $z$. It is very easy
to produce sequences $(g_n)$ where $z$ is the source
by dynamically ``zooming in" to $z$. 
The big difficulty is to prove that the sink of such a sequence
is distinct from $z$.
In the metric setup, where we know that $\pi_1(M)$ is Gromov
hyperbolic and $c$ is a point in the sphere at infinity
$\si$, one uses a {\underline {geodesic}}
ray $r$ with ideal point $c$ in $\si$ and pulls back points
along this ray to a compact set in $\mi$. 
The collection of pull backs generate a sequence $(g_n)$ in $\pi_1(M)$
which shows that $c$ is a conical limit point.
The problem in
the flow setting is that we do not know what geodesics
are, or more specifically, how geodesics interact with
the flow in $\mi$. In fact the main goal of this article is to show
that flow lines are almost like geodesics. 
Continuuing the analogy with the metric situation, if we were
to approach the point $z$ using a ``horocycle like" path, then
the sink $w$ for the sequence $(g_n)$ associated with the publlbacks
would also be 
equal to $z$. This is what we want to disallow. The difficulty
for us is that since there are perfect fits, many more points
in $\partial \oo$ are identified in $\rr$. Hence we have to be
extremely careful to ensure that the sink is distinct from
the source.
Theorem F is the hardest result proved in this article
and it has the longest proof.

After Theorem F is proved, we use a very important
result of Bowditch \cite{Bow1}
that shows the following:
if $\pi_1(M)$ acts as a uniform convergence group on $\rr$ (homeomorphic to
a sphere) then $\pi_1(M)$ is Gromov hyperbolic and $\rr$ is
$\pi_1(M)$ equivariantly homeomorphic to the Gromov ideal
boundary $\si$. Another point of view is that  Theorem F and this consequence
should be interpreted as a weak hyperbolization 
theorem in the setting  of flows: dynamical systems produce geometric information.

In addition the flow creates a flow ideal compactification of $\mi$ with
excellent properties.

\vskip .1in
\noindent
{\bf {Theorem G}} $-$
Let $\Phi$ be a bounded pseudo-Anosov flow. There is a natural and well 
defined topology in $\mi \cup \rr$ depending only on the foliations
$\wls, \wlu$ and satisfying the following properties. The space $\mi \cup \rr$ 
is compact and hence is a compactification of $\mi$.
The fundamental group acts on $\mi \cup \rr$ extending
the actions on $\mi$ and $\rr$.
This compactification is $\pi_1(M)$ equivariantly homeomorphic
to the Gromov compactification $\mi \cup \si$.
\vskip .1in

Theorem G implies that the terminology ``flow ideal boundary" for $\rr$
indeed makes sense as $\mi \cup \rr$ is compact and is equivalent to the Gromov
compactification.
This theorem means that under the bounded hypothesis the flow encodes
the large scale geometric structure of $\mi$.

Once Theorem G is proved  then properties in the flow compactification
get transferred to the Gromov compactification.
We show that in $\mi \cup \rr$ flow lines of $\wwp$ have
well defined forward and backward limit points in $\wwp$. For
each flow line we show that the forward and backward ideal points
are distinct and the forward ideal point map is continuous in $\mi$
(same for the backward ideal point map). This gets transferred to
the Gromov compactification. Finally these three properties
in $\mi \cup \si$ imply that $\Phi$ is a quasigeodesic
flow by a previous result of the author and Mosher \cite{Fe-Mo}.
This finishes the proof of the Main theorem.
One way to interpret these results is that Theorem A is the first  important
corollary of Theorem F.

The flow ideal boundary and compactification have many
excellent properties. In order to keep  this article
from being overly long we omit the proof or even the statement of many
of these properties. For example using only the flow one
can prove that $\pi_1(M)$ acts as a convergence group
on $\mi \cup \rr$ (for $\Phi$ bounded). We do not prove
this, but instead it can be easily derived from Theorem G and
the fact that this
is true in the Gromov compactification.

\vskip .1in
The results of this article imply the existence of many natural
group invariant Peano curves or Cannon-Thurston maps:

\vskip .1in
\noindent
{\bf {Theorem H}} $-$ Let $\Phi$ be a bounded pseudo-Anosov flow.
By theorem A the group $\pi_1(M)$ is Gromov hyperbolic. Then
any section $\tau :  \oo \rightarrow \mi$ extends to a 
continuous map $\overline \tau: \oo \cup \partial \oo
\rightarrow \mi \cup \si$. The ideal
map $\overline \tau |_{\partial \oo} : \partial \oo \rightarrow \si$
is a $\pi_1(M)$ equivariant Peano curve, in this situation also
called a Cannon-Thurston map. 
The map $\overline \tau |_{\partial \oo}$
is unique for the flow $\Phi$.
\vskip .1in

Cannon-Thurston maps are maps between ideal boundaries. 
These were introduced by Cannon and Thurston \cite{Ca-Th} in
the setting of surface groups ($\pi_1(S)$) in hyperbolic
$3$-manifolds ($M$). They produced a group equivariant 
Peano curve from $\partial \pi_1(S)$ to $\partial \pi_1(M)$
(which is equal to $\si$). These maps were greatly generalized
to other settings in Kleinian groups by Mj \cite{Mj}. Theorem
H produces a Cannon-Thurston map in the setting of pseudo-Anosov
flows. Notice that Frankel \cite{Fra2} produced Cannon-Thurston
maps for general (not necessarily pseudo-Anosov) quasigeodesic
flows in hyperbolic $3$-manifolds.

An immediate corollary of theorem H is an alternate proof of a result
of Frankel \cite{Fra2} in the case of pseudo-Anosov flows:

\vskip .1in
\noindent
{\bf {Corollary I}} $-$ Let $\Phi$ be a quasigeodesic pseudo-Anosov flow
in $M^3$ with $\pi_1(M)$ Gromov hyperbolic. Then any section
$\tau : \oo \rightarrow \mi$ induces a $\pi_1(M)$ equivariant 
Peano curve $\overline \tau |_{\partial \oo}: \partial \oo \rightarrow \si$.

\vskip .1in
Calegari \cite{Cal4} started the study of general quasigeodesic flows
in closed hyperbolic $3$-manifolds. He obtained important results, for example
they are always uniformly hyperbolic, the orbit space in the universal
cover is $\rrrr^2$ and they induce $\pi_1(M)$ actions on a circle.
Frankel showed that the orbit space can be naturally compactified
to a disk and  that quasigeodesic flows always produce group invariant
Peano curves \cite{Fra1,Fra2}. He also proved that in almost every
case, a quasigeodesic flow has closed orbits \cite{Fra1}.

We thank Lee Mosher who  informed us of Bowditch's theorem
on uniform convergence groups actions. This was the origin of this article.
The work of this article was greatly inspired by William Thurston who,
many years ago, taught us hyperbolic geometry, foliations and 
introduced us to pseudo-Anosov flows.
We also thank Steven Frankel for comments on a  preliminary version of this article.

\section{Background: Pseudo-Anosov flows}
\label{prelim}

Pseudo-Anosov flows are flows which are locally
like suspension flows of
pseudo-Anosov surface homeomorphisms. These flows behave much like
Anosov flows, but they may have finitely many singular orbits 
which are periodic and have 
a prescribed behavior. 

The manifold $M$ has a Riemannian metric.

\begin{define}{}{(pseudo-Anosov flow)}
Let $\Phi$ be a flow on a closed
$3$-manifold $M$. 
We say that $\Phi$ is a {\em pseudo-Anosov flow} if the following 
conditions are
satisfied:

- For each $x \in M$, the flow line $t \to \Phi(x,t)$ is $C^1$,
it is not a single point,
and the tangent vector bundle $D_t \Phi$ is $C^0$.

-There are two (possibly) singular transverse foliations
$\ls, \lu$ which are two dimensional, with leaves
saturated by the flow and so that $\ls, \lu$ 
intersect exactly along the flow lines of $\Phi$.

- There is a finite number (possibly zero) 
of periodic orbits $\{ \gamma_i \}$,
called {\em singular orbits}. A stable/unstable leaf
containing a singularity is homeomorphic to $P \times I/f$
where $P$ is a $p$-prong in the plane and
$f$ is a homeomorphism from $P \times \{ 1 \}$ to
$P \times \{ 0 \}$. In addition $p$ is at least $3$.

- In a stable leaf all orbits are forward asymptotic,
in an unstable leaf all orbits are backwards asymptotic.
\end{define}

Basic references for pseudo-Anosov flows are \cite{Mo3,Mo4},
and \cite{An} for Anosov flows.
A fundamental remark is that the ambient manifold supporting
a pseudo-Anosov flow is necessarily irreducible $-$ this
is because the universal cover is homeomorphic to
$\rrrr^3$ \cite{Fe-Mo}.
We stress that one prongs are not allowed.

\vskip .05in
\noindent
\underline {Notation/definition:} \ 
We denote by $\pi: \mi \rightarrow M$ the universal
covering of $M$, and by $\pi_1(M)$ the fundamental group
of $M$, considered as the group of deck transformations
on $\mi$.
The singular
foliations lifted to $\mi$ are
denoted by $\wls, \wlu$.
If $x \in M$ let $\ls(x)$ denote the leaf of $\ls$ containing
$x$.  Similarly one defines $\lu(x)$
and in the
universal cover $\wls(x), \wlu(x)$.
If $\alpha$ is an orbit of $\Phi$, similarly define
$\ls(\alpha)$, 
$\lu(\alpha)$, etc...
Let also $\wwp$ be the lifted flow to $\mi$.

\vskip .1in
We review the results about the topology of
$\wls, \wlu$ that we will need.
We refer to \cite{Fe4,Fe5} for detailed definitions, explanations and 
proofs.
Proposition 4.2 of \cite{Fe-Mo} shows that the
orbit space of $\wwp$ in
$\mi$ is homeomorphic to the plane $\rrrr^2$.
The orbit space is denoted by $\oo$ which is the quotient
space $\mi/\wwp$. 
There is an induced action of $\pi_1(M)$ on $\oo$. Let

$$\Theta: \mi \rightarrow \oo \cong \rrrr^2$$

\noindent
be the projection map. 
It is naturally $\pi_1(M)$ equivariant.
%As the foliations $\fns, \fnu$
%are invariant under $\wwp$, they 
%induce singular, transverse $1$-dim foliations
%$\fnso, \fnuo$ in $\oo$.
%The singular points  of $\fnuo$ are the same as those
%of $\fnso$.
If $L$ is a 
leaf of $\wls$ or $\wlu$,
then $\Theta(L) \subset \oo$ is a tree which is either homeomorphic
to $\rrrr$ if $L$ is regular,
or is a union of $p$ rays all with the same starting point
if $L$ has a singular $p$-prong orbit.
%In particular every orbit in $L$ disconnects $L$.
The foliations $\wls, \wlu$ induce $\pi_1(M)$ invariang
singular $1$-dimensional foliations
$\oos, \oou$ in $\oo$. Its leaves are the $\Theta(L)$'s as
above.
%If $L$ is a leaf of $\wls$ or $\wlu$, then 
%a {\em sector} is a component of $\mi - L$.
Similarly for $\oos, \oou$. 
If $B$ is any subset of $\oo$, we denote by $B \times \rrrr$
the set $\Theta^{-1}(B)$.
The same notation $B \times \rrrr$ will be used for
any subset $B$ of $\mi$: it will just be the union
of all flow lines through points of $B$.
If $x$ is a point of $\oo$, then $\oos(x)$ (resp. $\oou(x)$)
is the leaf of $\oos$ (resp. $\oou$) containing $x$.
We stress that for pseudo-Anosov flows there are at least
$3$ prongs in any singular orbit ($p \geq 3$). For example
the fact theat the orbit space in $\mi$ is a $2$-manifold
would not be true if one allowed $1$-prongs.

\begin{define}
Let $L$ be a leaf of $\wls$ or $\wlu$. A slice leaf of $L$ is 
$l \times \rrrr$ where $l$ is a properly embedded
copy of the real line in $\Theta(L)$. For instance if $L$
is regular then $L$ is its only slice leaf. If a slice leaf
is the boundary of a component of $\mi - L$ then it is called
a line leaf of $L$.
If $a$ is a ray in $\Theta(L)$ then $A = a \times \rrrr$
is called a half leaf of $L$.
%The closure is denoted by $\overline A = A \cup \gamma$ and 
%its boundary is $\partial A = \gamma$.
If $\zeta$ is an open segment in $\Theta(L)$ 
it defines a flow band $L_1$ of $L$
by $L_1 = \zeta \times \rrrr$.
%We use the same terminology of slices and line leaves for 
%the foliations $\oos, \oou$ of $\oo$.
\end{define}

\noindent
{\bf {Important convention}} $-$ In general a slice leaf is just a slice
leaf of some $L$ in $\wls$ or $\wlu$ and so on.
We also use the terms slice leaves, line leaves,
perfect fits, lozenges and rectangles for the projections of these
objects in $\mi$ to the orbit space $\oo$.

\vskip .05in
If $F \in \wls$ and $G \in \wlu$ 
then $F$ and $G$ intersect in at most one
orbit. 
Also suppose that a leaf $F \in \wls$ intersects two leaves
$G, H \in \wlu$ and so does $L \in \wls$.
Then $F, L, G, H$ form a {\em rectangle} in $\mi$
and there is  no singularity  of $\wwp$ 
in the interior of the rectangle see \cite{Fe5} pages 637-638.
There will be two generalizations of rectangles: 1) perfect fits $=$ 
in the orbit space this is a  properly embedded rectangle
with one corner removed 
and 2) lozenges $=$ rectangle with two opposite corners removed.

We abuse convention and call a leaf $L$ of $\wls$ or $\wlu$
{\em periodic} if there is a non trivial covering translation
$g$ of $\mi$ with $g(L) = L$. Equivalently $\pi(L)$ contains
a periodic orbit of $\Phi$. In the same way an orbit
$\gamma$ of $\wwp$ is {\em periodic} if $\pi(\gamma)$ is a 
periodic orbit of $\Phi$. Observe that in general the stabilizer
of an element $\alpha$ of $\oo$ is either trivial, or 
an infinite cyclic subgroup of $\pi_1(M)$.

\begin{figure}
\centeredepsfbox{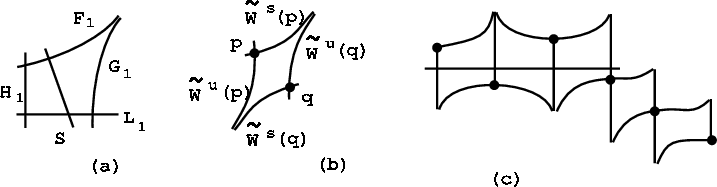}
\caption{a. Perfect fits in $\mi$,
b. A lozenge, c. A chain of lozenges.}
\label{cv1}
\end{figure}

\begin{define}{(\cite{Fe2,Fe4,Fe5})}{}
Perfect fits -
Two leaves $F \in \wls$ and $G \in \wlu$, form
a perfect fit if $F \cap G = \emptyset$ and there
are half leaves $F_1$ of $F$ and $G_1$ of $G$ 
and also flow bands $L_1 \subset L \in \wls$ and
$H_1 \subset H \in \wlu$,
so that 
%$F_0$ is regular on the side containing $L$,
%$G_0$ is regular on the side containing $H$ and:
%
%$$ \overline L_1 \cap \overline G_1 = \partial L_1 \cap \partial G_1,
%\ \ \overline L_1 \cap \overline H_1 = \partial L_1 \cap \partial H_1,
%\ \ \overline H_1 \cap \overline F_1 = 
%\partial H_1 \cap \partial F_1,$$ 
the set 

$$\overline F_1 \cup \overline H_1 \cup 
\overline L_1 \cup \overline G_1$$

\noindent
separates $M$ 
and forms an a rectangle $R$ with a corner removed:
The joint structure of $\wls, \wlu$ in 
$R$ is that of
a rectangle as above without one corner orbit.
The removed corner corresponds to the perfect
fit of $F$ and $G$ which do not intersect.
%The missing corner 
%corresponds to the perfect of $F$ and $G$.
%Specifically, a stable leaf intersects $H_1$ if and only
%if it intersects $G_1$ and similarly for unstable
%leaves intersecting $F_1, L_1$.
% which 
%do not intersect.
%`$${\rm with} \ \ \  \overline L_1 \cap \overline G_1 \ \not = \ \emptyset, 
%\ \ \overline L_1 \cap \overline H_1  \ \not = \ \emptyset \ \ {\rm and}
%\ \ \overline H_1 \cap \overline F_1 \ \not = \ \emptyset.$$
%\noindent
%Furthermore
%$$\forall \ S \in \wlu, \ \ \ \   
%S \cap L_1  \not = \emptyset \ \Rightarrow
%S \cap F_1  \not = \emptyset \ \ \ (1) $$
%$${\rm and} \ \ \forall \ E \in \wls, \ \ \ \ E \cap H_1  
%\not = \emptyset \ \Rightarrow
%E \cap G_1  \not = \emptyset \ \ \ (2).$$
\end{define}

We refer to fig. \ref{cv1}, a for perfect fits.
%We will also denote by rectangles, perfect fits, lozenges
%and product regions the projection of these regions to
%$\oo \cong \rrrr^2$. 
%Implications $(1), (2)$ force
%equivalences (that is 
%$S \cap L_1  \not = \emptyset \ \Leftrightarrow
%S \cap F_1  \not = \emptyset$ and the same for (2)).
%The set 
%$\overline F_1 \cup \overline H_1 \cup 
%\overline L_1 \cup \overline G_1$
%separates $\mi$.
%Let $A$ be the complementary region which does
%not contain $F - F_1$ in its closure.
%An important fact is that there are 
%singularities of $\wwp$
%in $A$.
%Perfect fits produce ``ideal" rectangles, in the sense that
%even though $F$ and $G$ do not intersect, there is
%a product structure (of $\wls$ and $\wlu$) in the 
%interior of $A$.
There is a product structure in the interior of $R$: there are
two stable boundary sides and two unstable boundary sides 
in $R$. An unstable
leaf intersects one stable boundary side (not in the corner) if
and only if it intersects the other stable boundary side
(not in the corner).
We also say that the leaves $F, G$ {\em asymptotic}.

%\begin{define}{\cite{Fe2}}{}
%Given $p \in \mi$ (or $p \in \oo$), and a half leaf $H$ 
%of $\wu(p)$ defined by $\wwp_{\rrrr}(p)$, let 
%
%$$\uo(H) \ = \ \{ F \in \hs  \ | \ F \cap 
%H \not = \emptyset \} \ \subset \ \hs.$$
%
%\noindent
%Notice that $\ws(p) \not \in \uo(H)$.
%Let  also
%
%$${\cal L}^u(H)  \ = 
%\ \bigcup  \ \{ \ p \ \in \mi \ \ 
%| \ \ p \ \in \ F \ \in \ \uo(H) \ \} \ \ \subset \ \ \mi.$$
%
%\noindent
%Then ${\cal L}^u(H) \subset \mi$
%and $\ws(p) \subset \partial {\cal L}^u(H)$.
%Similarly define
%$\so(L),
%{\cal L}^s(L)$ for a stable half leaf $L$.
%\end{define}

\begin{define}{(\cite{Fe4,Fe5})}{}
Lozenges - A lozenge is a region of $\mi$ whose closure
is homeomorphic to a rectangle with two corners removed.
More specifically two points $p, q$ define the corners 
of a lonzenge if there are half leaves $A, B$ of
$\wls(p), \wlu(p)$ defined by $p$
and  $C, D$ half leaves of $\wls(q), \wlu(q)$ so
that $A$ and $D$ form a perfect fit and so do
$B$ and $C$. The sides of $R$ are $A, B, C, D$.
The sides are not contained in the lozenge, but
are in the boundary of the lozenge.
There may be singularities in the boundary of the lozenge.
See fig. \ref{cv1}, b.
\end{define}

This is definition 4.4 of \cite{Fe5}.
Two lozenges are {\em adjacent} if they share a corner and
there is a stable or unstable leaf
intersecting both of the lozenges, see fig. \ref{cv1}, c.
Therefore they share a side.
A {\em chain of lozenges} is a collection $\{ \cc _i \}, 
i \in I$, of lozenges
where $I$ is an interval (finite or not) in ${\bf Z}$,
so that if $i, i+1 \in I$, then 
${\cal C}_i$ and ${\cal C}_{i+1}$ share
a corner, see fig. \ref{cv1}, c.
Consecutive lozenges may be adjacent or not.
The chain is finite if $I$ is finite.

\begin{define}{}{}
Suppose $A$ is a flow band in a leaf of $\wls$.
Suppose that for each orbit $\gamma$ of $\wwp$ in $A$ there is a
half leaf $B_{\gamma}$ of $\wlu(\gamma)$ defined by $\gamma$ so that: 
for any two orbits $\gamma, \beta$ in $A$ then
a stable leaf intersects $B_{\beta}$ if and only if 
it intersects $B_{\gamma}$.
%for any 
%\zeta \subset E \in \wls$ is a (possibly infinite) 
%strong stable segment 
%so that for each $p \in \zeta$ there is a half leaf $H_p$
%of $\wu(p)$ defined by $\wwp_{\rrrr}(p)$ so that
%
%$$ \forall \ p, q \in \zeta, \ \ \uo(H_p) = \uo(H_q).
%\ \ \  In \ that \ case \ let \ {\cal P} = \bigcup _{p \in \zeta} 
%H_p.$$
%
%\noindent
%Then ${\cal P} \subset \mi$ is called an unstable product
%region with base segment $\zeta$. 
%The base segment
%is not uniquely determined by ${\cal P}$.
%Similarly define negative unstable product regions 
This defines a stable product region which is the union
of the $B_{\gamma}$.
Similarly define unstable product regions.
\label{defsta}
\end{define}

The main property of product regions is the following, see
\cite{Fe5} page 641:
for any $F \in \wls$, $G \in \wlu$ so that 
$(i) \ F \cap A  \
\not = \ \emptyset \ \ {\rm  and} \ \ 
 (ii) \ G \cap A \ \not = \ \emptyset,
\ \ \ {\rm then} \ \ 
F \cap G \ \not = \ \emptyset$.
There are no singular orbits of 
$\wwp$ in $A$.

We say that 
two orbits $\gamma, \alpha$ of $\wwp$ 
(or the leaves $\wls(\gamma), \wls(\alpha)$)
are connected by a 
chain of lozenges $\{ {\cal C}_i \}, 1 \leq i \leq n$,
if $\gamma$ is a corner of ${\cal C}_1$ and $\alpha$ 
is a corner of ${\cal C}_n$.
If a lozenge ${\cal C}$ has corners $\beta, \gamma$ and
if $g$ in $\pi_1(M) - id$ satisfies $g(\beta) = \beta$,
$g(\gamma) = \gamma$ (and so $g({\cal C}) = {\cal C}$),
then $\pi(\beta), \pi(\gamma)$ are closed orbits
of $\Phi$ which are freely homotopic to the inverse of each
other.

\begin{theorem}{(\cite{Fe5}, theorem 4.8)}{}
Let $\Phi$ be a pseudo-Anosov flow in $M$ closed and let 
$F_0 \not = F_1 \in \wls$.
Suppose that there is a non trivial covering translation $g$
with $g(F_i) = F_i, i = 0,1$.
Let $\alpha_i, i = 0,1$ be the periodic orbits of $\wwp$
in $F_i$ so that $g(\alpha_i) = \alpha_i$.
Then $\alpha_0$ and $\alpha_1$ are connected
by a finite chain of lozenges 
$\{ {\cal C}_i \}, 1 \leq i \leq n$ and $g$
leaves invariant each lozenge 
${\cal C}_i$ as well as their corners.
\label{chain}
\end{theorem}

%A chain from $\alpha_0$ to $\alpha_1$ is called
%{\em minimal} if all lozenges in the chain are distinct.
%Exactly as proved in \cite{Fe2} for Anosov flows,
%it follows that there
%is a unique minimal chain from $\alpha_0$ to $\alpha_1$
%and also all other chains have to contain all the lozenges
%in the minimal chain.

\begin{figure}
\centeredepsfbox{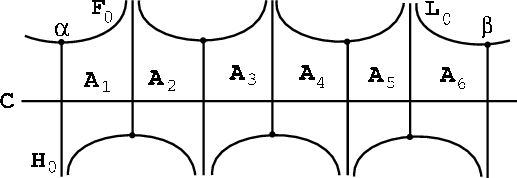}
\caption{
The correct picture between non separated
leaves of $\wls$.}
\label{pict}
\end{figure}

This means that each free homotopy of closed orbits generates a chain of lozenges
preserved by a non trivial element $g$ of $\pi_1(M)$.

The leaf space of $\wls$ (or $\wlu$) is usually not a Hausdorff
space. Two points of this space are non separated if they
do not have disjoint neighborhoods in the respective leaf
space.
The main result concerning non Hausdorff behavior in the leaf spaces
of $\wls, \wlu$ is the following:

\begin{theorem}{(\cite{Fe5}, theorem 4.9)}{}
Let $\Phi$ be a pseudo-Anosov flow in $M^3$. 
Suppose that $F \not = L$
are not separated in the leaf space of $\wls$.
Then $F$ and $L$  are periodic.
Let $F_0, L_0$ be the line leaves of $F, L$ which
are not separated from each other.
Let $V_0$ be the component of $\mi - F$ bounded by
$F_0$ and containing $L$.
Let $\alpha$ be the periodic orbit in $F_0$ and
$H_0$ be the component of $(\wlu(\alpha) - \alpha)$ 
contained in $V_0$.
Let $g$ be a non trivial covering translation
with $g(F_0) = F_0$, $g(H_0) = H_0$ and $g$ leaves
invariant the components of $(F_0 - \alpha)$.
Then $g(L_0) = L_0$.
This produces closed orbits of $\Phi$ which are freely
homotopic in $M$.
Theorem \ref{chain} then implies that $F_0$ and $L_0$ are connected by
a finite chain of lozenges 
$\{ A_i \}, 1 \leq i \leq n$,
%all contained in ${\cal L}^u(H_0)$.
where consecutive lozenges are adjacent.
They all intersect a common stable leaf $C$.
There is an even number of lozenges 
in the chain, see
fig. \ref{pict}.
In addition 
let ${\cal B}_{F,L}$ be the set of leaves of $\wls$ non separated
from $F$ and $L$.
Put an order in ${\cal B}_{F,L}$ as follows:
The set of orbits of $C$ contained
in the union of the lozenges and their sides is an interval.
Put an order in this interval.
If $R_1, R_2 \in {\cal B}_{F,L}$ let $\alpha_1, \alpha_2$
be the respective periodic orbits in $R_1, R_2$. Then
$\wlu(\alpha_i) \cap C \not = \emptyset$ and let 
$a_i = \wlu(\alpha_i) \cap C$.
We define $R_1 < R_2$ in ${\cal B}_{F,L}$ 
if $a_1$ precedes
$a_2$ in the order of the set of orbits of 
$C$.
Then ${\cal B}_{F,L}$
is either order isomorphic to $\{ 1, ..., n \}$ for some
$n \in {\bf N}$; or ${\cal B}_{F,L}$ is order
isomorphic to the integers ${\bf Z}$.
In addition if there are $Z, S \in \wls$ so that
${\cal B}_{Z, S}$ is infinite, then there is 
an incompressible torus in $M$ transverse to 
$\Phi$. In particular $M$ cannot be atoroidal.
Also if there are $F, L$ non separated from each other as above, then there are
closed orbits $\alpha, \beta$ of $\Phi$ which
are freely homotopic to the inverse of each other.
Finally up to covering translations,
there are only finitely many non Hausdorff
points in the leaf space of $\wls$.
\label{theb}
\end{theorem}

Notice that ${\cal B}_{F,L}$ is a discrete set in this order.
For detailed explanations and proofs, see
\cite{Fe4,Fe5}. 

\begin{theorem}{(\cite{Fe5}, theorem 4.10)}{}
Let $\Phi$ be a pseudo-Anosov flow. Suppose that there is
a stable or unstable product region. Then $\Phi$ is 
topologically conjugate to a suspension Anosov flow.
In particular $\Phi$ is nonsingular.
\label{prod}
\end{theorem}

\section{Ideal boundary of orbit spaces of pseudo-Anosov flows}

Let $\Phi$ be an arbitrary  pseudo-Anosov flow in a $3$-manifold  $M$. The orbit
space $\oo$ of $\wwp$ (the lifted flow to $\mi$) 
is homeomorphic to $\rrrr^2$ \cite{Fe-Mo}. 
In \cite{Fe9} we constructed a
natural compactification of $\oo$ with an ideal circle
$\partial \oo$ called the ideal boundary of the orbit space of a pseudo-Anosov
flow. The compactification $\cd = \oo \cup \partial \oo$ 
is homeomorphic to a closed disk.
The induced action of $\pi_1(M)$ on $\oo$ 
extends to an action on $\oo \cup \partial \oo$.
We describe the main objects and results here. Detailed proofs and explanations
are found in section 3 of \cite{Fe9}.

The  construction of $\partial \oo$ uses only the 
foliations $\oos, \oou$.
To illustrate a simple example,  consider 
suspension pseudo-Anosov homeomorphisms of closed surfaces.
Let $S$ be such a surface, suppose it is hyperbolic. Then a lift
$\widetilde S$ to $\mi$ is a cross section to the flow $\wwp$,
hence $\widetilde S$ is naturally identified with the orbit
space $\oo$. The 
stable and unstable foliations of $\wwp$ induce singular foliations
in $\widetilde S \cong \oo$. In this case the ideal boundary of $\partial \oo$ will
be identified to the circle at infinity of $\widetilde S$. A
point in this circle which is not the ideal point of a stable leaf has
a neighborhood system defined by stable leaves: 
for example a nested sequence
of stable leaves which ``shrinks" or converges to the ideal point.
Clearly infinitely many 
such sequences converge to the same ideal point so one 
considers equivalence classes of such sequences.

For an arbitrary pseudo-Anosov flow $\Phi$ this leads to following 
naive approach:
consider sequences $(l_i)$
in say $\oos$ so that the leaves are nested and the sequence
escapes compact sets in $\oo$.  
In general this is not enough because
of perfect fits. Consider a ray of an stable leaf $l$.
The ray should define a point $x$ in $\partial \oo$ and the 
obvious way to try to
define $x$ is to consider a nested sequence $(u_i)$
of unstable leaves intersecting $l$ and so that $(u_i \cap l)$
is in the given ray of $l$ and escapes in $l$. 
This does not work if for example that ray of $l$ makes
a perfect fit with an unstable leaf $u$. If this happens
then the sequence $(u_i)$ does not escape compact
sets in $\oo$: the leaf $u$ is a barrier.
%We stress this difficulty:
%in terms of the ideal boundary of $\oo$, 
%the difference between general pseudo-Anosov flows and
%the suspension case above is exactly due to the possible existence of perfect fits.
%Since in general 
%individual leaves are not enough to define ideal
Because of this we consider polygonal paths in leaves
of $\oos$ or $\oou$.

\begin{define}{(polygonal path)}{}
A polygonal path in $\oo$ is a properly embedded, bi-infinite path
$\zeta$
in $\oo$ satisfying:  either $\zeta$ is a leaf of $\oos$ or $\oou$ or
$\zeta$ is the union of a finite collection
$l_1, ... l_n$ 
of segments and  rays in leaves of $\oos$ or $\oou$
so that
$l_1$ and $l_n$ are rays in $\oos$ or $\oou$ and  the other $l_i$
are finite segments. 
We require that 
$l_i$ intersects $l_j$ if and only if $|i-j| \leq 1$.
% and only in their 
%finite endpoint and 
In addition the $l_i$ are alternatively in $\oos$ and $\oou$.
The number $n$ is the length of the polygonal path.
The points $l_i \cap l_{i+1}$ are the vertices of the path.
The edges of $\zeta$ are the $\{ l_i \}$.
\end{define}

Given $z$ in $\oo$, a {\em sector} at $z$ is a component of $\oo - (\oos(z) \cup \oou(z))$.
If $z$ is nonsingular
there are exactly $4$ sectors, if $z$ is a $k$-prong point
there are $2k$ sectors.
This is also defined in $\mi$ using $\wls, \wlu$.
Finally this is defined locally in $M$ using $\ls, \lu$.

\begin{define}{(convex polygonal paths)}{}
A polygonal path $\delta$ in $\oo$ is convex if there is a complementary region $V$
of $\delta$ in $\oo$ so that at any given vertex $z$ of $\delta$ 
the local region of $V$ near $z$ is not a sector at $z$.
Let $\widetilde \delta  = \oo - (\delta \cup V)$. This region 
$\widetilde \delta$ is the convex
region of $\oo$ associated to the convex polygonal path $\delta$.
\label{convex}
\end{define}

We refer to fig. \ref{id1}.
The definition implies that if the region $\widetilde  \delta$ contains
2 endpoints of a segment in a leaf of $\oos$ or $\oou$, then
it contains the entire segment.
This is why $\delta$ is called convex.
In some restricted situations both complementary components of $\delta$
are convex, for example if $\delta$ is contained in a single
leaf of $\oos$ or $\oou$.

\begin{figure}
\centeredepsfbox{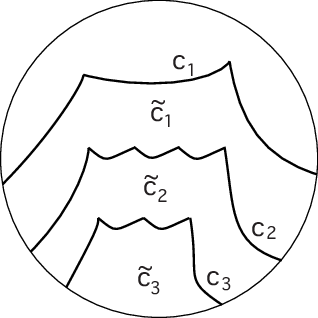}
\caption{The paths $c_1, c_2, c_3$ are examples of convex polygonal paths.
The set $\{ c_1, c_2, c_3 \}$ is nested %as $\widetilde c_3 \subset \widetilde c_2 \subset \widetilde c_1$.
Here $c_1$ has length $3$, $c_2$ has length $5$ and $c_3$ has length $4$.
\ $c_1$ and $c_3$ have both rays in the same foliation (either $\oos$ or
$\oou$), and $c_2$ has one ray in $\oos$ and one ray in $\oou$.}
\label{id1}
\end{figure}

\begin{define}{(equivalent rays)}{}
Two rays $l, r$ in leaves of $\oos, \oou$ are {\em {equivalent}} if there
is a finite collection of distinct rays $l_i, 1 \leq i \leq n$, 
alternatively in $\oos, \oou$ so that $l = l_0, r = l_n$ and
$l_i$ forms a perfect fit with $l_{i+1}$ for \ $1 \leq i < n$.
The two rays can either be in the same foliation ($\oos$ or $\oou$)
or in distinct foliations.
\end{define}

This is strictly about rays in
$\oos, \oou$ and not leaves of $\oos, \oou$.  More specifically,
consecutive perfect fits involve  the same ray
of the in between or middle leaf. This implies for instance that 
if $n \geq 3$ then for all $1 \leq i \leq n - 2$ the
leaves containing $l_i$ and $l_{i+2}$ are non separated from each other in the
respective leaf space.
% of $\oos$ or $\oou$.

\begin{define}{(admissible sequences of paths)}{}
An admissible sequence of polygonal paths in $\oo$ is a sequence of convex 
polygonal paths 
$(v_i)_{i \in {\bf N}}$ so that the associated convex regions 
$\widetilde v_i$ form a nested 
sequence of subsets of $\oo$, which escapes compact sets in $\oo$ and
for any $i$,
the two rays 
of $v_i$ are not equivalent.
\end{define}

The fact that the $\widetilde v_i$ are nested and escape compact sets
in $\oo$ implies that the $\widetilde v_i$ are uniquely defined given
the $v_i$.
An ideal point of $\oo$ will be determined by an 
admissible sequence of paths.
Different admissible sequences may define the same ideal
point, so we explain when two such sequences are equivalent.
%we first need to decide when two such sequences
%are equivalent.
%oo \cup \partial \oo$.
%At first it seems that any 2 sequences associated to 
%%(what we intuitively think is) 
%the same ideal point of $\oo$
%would have to be eventually nested with each other. 
%This is not true as we will explain after the next definition.
%This 
%would make it simple and nice. k

%However it is easy to see that
%such is not the case. For example consider a nested sequence of
%rays of a fixed leaf $l$.
%We will later see how to
%extend each ray on one side of $l$ to form an admissible sequence.
%Extend them also to the other side to form another admissible
%sequence. Intuitively the two sequences should converge to the intrinsic
%ideal point of $l$, but clearly they are not eventually nested.

\begin{define}{}{}
Given two admissible sequences of 
polygonal paths  $C = (c_i)$, \ $D = (d_i)$, we say that 
$C$ is smaller or equal than $D$, denoted by $C \leq D$,
if: for any $i$ there is $k_i > i$ so that
$\widetilde c_{k_i} \subset \widetilde d_i$.
Two admissible sequences of polygonal paths  $C = (c_i), \ D = (d_i)$
are related  or equivalent to each other 
if there is
a third admissible sequence $E = (e_i)$ so that
$C \leq E$ and $D \leq E$.
\label{rela}
\end{define}

Why not require $C$ equivalent to $D$ if $C \leq D$ and
$D \leq C$?
The conditions $C \leq D$ and $D \leq C$ together mean
that $C$ and $D$ are eventually nested, which would seem
to be the natural requirement of the relation.
The reason for the unexpected definition of the 
relation is the following. Suppose $l$ is a leaf
of $\oos$ and $(u_i)$ a nested sequence in $\oou$ all intersecting
$l$ and escaping in $\oo$. Then $(u_i)$ will define the ideal
point $x$ of a ray of $l$. Each $\widetilde u_i$ contains a subray $l_i$ of
$l$ and $l$ cuts $u_i$ into two subrays $u^1_i$ and $u^2_i$.
Let $s^j_i = l_i \cup u^j_i$ for $j = 1, 2$. Then each
$s^j_i$ is a convex polygonal path and both $(s^1_i)$ and $(s^2_i)$
are admissible sequences that should define this same ideal point
$x$. Here the $3$ admissible sequences satisfy \
$(s^1_i) \leq (u_i)$ and $(s^2_i) \leq (u_i)$,
but $\widetilde s^1_k \cap \widetilde s^2_i = \emptyset$ for
any $k, i$ in ${\bf N}$.
Roughly the admissible sequences $(s^1_i)$ and $(s^2_i)$ approach
$x$ from ``opposite" sides of $x$.

\begin{lemma}{}{}
Suppose that $\Phi$ is not topologically conjugate to a suspension
Anosov flow.
Then the relation defined in Definition \ref{rela}
is an equivalence relation in the set of admissible sequences of 
polygonal paths.
\label{equiv}
\end{lemma}

\begin{define}{(Ideal boundary $\partial \oo$)}{}
Suppose that $\Phi$ is not topologically conjugate to a suspension
Anosov flow.
A point in $\partial \oo$ or an ideal point of $\oo$
is an equivalence class of
admissible sequences of polygonal paths.
Let $\cd = \oo \cup \partial \oo$.
\label{idpo}
\end{define}

Given $R$, an admissible sequence of polygonal paths, let
$\overline R$ be its equivalence class under the relation $\cong$
defined in Definition \ref{rela}.

\begin{define}{(master sequences)}{}
Let $R$ be an admissible sequence. An admissible sequence $C$
defining $\overline R$ is a master sequence for $\overline R$
if for any $B \cong R$, then $B \leq C$.
\label{master}
\end{define}

The intuition here is that elements of a master sequence
approach the ideal points from ``both sides".

\begin{lemma}{}{}
Given an admissible sequence $R$,
there is a master sequence
for $\overline R$.
\label{mast}
\end{lemma}

\begin{lemma}{}{}
Let $p, q$ in $\partial \oo$. Then $p, q$ are distinct if and
only if there are master sequences 
$A =  (a_i), \ B =  (b_i)$ associated
to $p, q$ respectively with 
$(a_i \cup \widetilde a_i) \cap (b_j \cup \widetilde b_j)
= \emptyset$ for some $i, j$. 
%Equivalently for some other 
%master sequences this is true for all $i, j$.
\label{idst}
\end{lemma}

We now define the topology in $\cd = \oo \cup \partial \oo$.

\begin{define}{(topology in $\cd = \oo \cup \partial \oo$)}{}
Let ${\cal X}$ be the set of subsets 
$U$ of $\cd = \oo \cup \partial \oo$
satisfying the following two conditions:

(a)  $U \cap \oo$ is open in $\oo$.
%If $x$ is in $U \cap \oo$, then there is an open set
%$O$ in $\oo$ with $x \in O \subset U$.

(b) If $p$ is in $U \cap \partial \oo$
and $A = (a_i)$ is any master sequence associated
to $p$, then there is $i_0$ satisfying
two conditions: (1) $\widetilde a_{i_0} \subset U
\cap \oo$ and (2) For any $z$ in $\partial \oo$,
if it admits a master sequence $B = (b_i)$ so that
for some $j_0$, one has $\widetilde b_{j_0} \subset 
\widetilde a_{i_0}$ then $z$ is in $U$.
\label{topology}
\end{define}

%If the second requirement works
%for a master sequence $A = (a_i)$ with index $i_0$, then
%for any other master sequence $C = (c_k)$ defining $p$,
%we can choose
%$k_0$ with
%$\widetilde c_{k_0} \subset \widetilde a_{i_0}$. 
%Then $\widetilde c_{k_0} \subset U$.
%A point $q$ of $\partial \oo$ which has a master
%sequence $B = (b_j)$  and $j_0$ so that
%
%$$\widetilde b_{j_0} \
%\subset \ \widetilde c_{k_0}; \ \ \ {\rm then } \ \ \ 
%\widetilde b_{j_0} \ \subset \ \widetilde a_{i_0}$$ 

\begin{define}{}{}
(the sets $U_C, V_c$)
For any convex polygonal path $c$ there are associated open sets $U_c, V_c$
of $\cd$ defined as follows.
Let $\widetilde c$ be the corresponding convex set of $\oo$ (
if $c$ has length 1 there are two possibilities). 
Let $U_c = \widetilde c$. In addition let

$$V_c \ = \ \widetilde c \ \cup \ \{ x \in \partial  \oo \ | \
\ \ {\rm there \ is \ a \  master \ sequence} \ \ A = (a_i)
\ \ {\rm with} \ \ \widetilde a_1 \subset \widetilde c \ \}$$

\label{canon}
\end{define}

It is easy to verify that $V_c$ is always an open set in $\cd$.
In particular it is an open neighborhood of any point
in $V_c \cap \partial \oo$.
The rays of the polygonal path $c$ are equivalent if and only if $V_c$ is contained
in $\oo$.
The notation $V_c$ will be used from now on.

\begin{lemma}{}{}
For any ray $l$ of $\oos$ or $\oou$, there is an associated
point in $\partial \oo$.  Two rays generate the same point
of $\oo$ if and only if the rays are equivalent (as rays!).
\label{equiva}
\end{lemma}

%\begin{lemma}{}{}
%Let $A = (a_i)$ be an admissible sequence defining a point
%$p$ in $\partial \oo$. Then one of the following mutually
%exclusive possibilities occurs:
%
%(i) There are infinitely many $i$ in 
%${\bf N}$ and for each such $i$ there is a ray $l_i$ of $a_i$
%which is equivalent to a fixed ray $l$ of $\oos$ 
%or $\oou$. Then $p$ is the ideal
%point of any of the $l_i$ and $A$ is not a master sequence for $p$.
%In fact in this case the hypothesis is true for any $i$ sufficiently
%big.
%
%(ii) There are only finitely many rays of paths in the collection
%$\{ a_i \}$ which are equivalent
%to any given ray of $\oos$ or
%$\oou$. In this case $A$ is a master sequence for $p$.
%\label{choi}
%\end{lemma}

\begin{proposition}{}{}
The space $\partial \oo$ is homeomorphic to a circle.
\label{circle}
\end{proposition}

\begin{theorem}{}{}
The space $\cd = \oo \cup \partial \oo$ is homeomorphic to
the closed disk $D^2$.
\label{disk}
\end{theorem}

Notice that $\pi_1(M)$ acts on $\oo$ by homeomorphisms.
The action preserves the foliations $\oos, \oou$ and also 
preserves convex polygonal paths, admissible sequences, master sequences
and so on. Hence $\pi_1(M)$ also acts by homeomorphisms
on $\cd$.

\begin{define}{}{(standard sequences associated to ideal points of rays)}
Suppose that $\Phi$ is a bounded pseudo-Anosov flow.
Let $l$ be a ray in $\oos$ or $\oou$. 
Let $\mathcal L = \{ l_j, 1 \leq j \leq k \}$ be the  maximal
collection of leaves of $\oos$ or $\oou$ so that each $l_j$ has a ray equivalent to $l$.
%to a ray of $l_{j'}$ 
%and $l$ is a ray of some $l_j$. Since leaves of $\oos, \oou$ separate $\oo$ then
%for each $j$ a fixed ray of $l_j$ is equivalent to rays
%in all other $l_{j'}$.
Then each sequence $\mathcal C$ below is
a master sequence for the ideal points of the corresponding  rays of the leaves $l_j$: \
$\mathcal C  =  (c_n), \ n \in {\bf N}$, where each $(c_n)$ is a convex polygonal
chain of length $k$, 
$c_n = \{ c^1_n, ...., c^k_n \}$ so that the following happens.
For each $n$ and for each $j$, then $c^j_n$  intersects transversely
the fixed ray of $l_j$. 
The $c^1_n, c^k_n$ are rays, and the other elements
of $c_n$ are segments. The sequence $(c_n)$ is nested with $n$ and
escapes compact sets in $\oo$.
In particular for each $j$, $(l_j \cap c^j_n)$ is a sequence that
escapes in $l_j$ when $n$ converges to infinity.
\label{stan}
\end{define}

We refer to fig. \ref{id2} for standard sequences of ideal
points of rays.

\begin{figure}
\centeredepsfbox{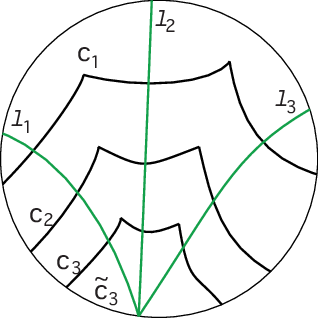}
\caption{
A standard sequence for ideal points of rays. Here $l_1, l_2, l_3$
is a maximal set of leaves with equivalent rays. Each convex
polygonal chain $c_i$ has length $3$ with rays 
$c^1_n, c^3_n$ and a segment $c^2_n$. For each $j$ and $n$, $c^j_n$
intersects $l_j$ transversely. Finally $(c_n)$ is nested and
escapes compact sets in $\oo$.}
\label{id2}
\end{figure}

\begin{proposition}{}{}
Let $\Phi$ be a bounded pseudo-Anosov flow in
$M^3$ closed. Let $p$ be an ideal point of $\oo$.
Then one of 2 mutually exclusive options occurs:

1) There is a master sequence $L = (l_i)$ for $p$ where
$l_i$ are slices in leaves of $\oos$ or $\oou$.

2) $p$ is an ideal point of a ray $l$ of $\oos$ or
$\oou$ so that $l$ makes a perfect fit with another
ray of $\oos$ or $\oou$. There are master sequences
which are standard sequences associated to the ray $l$ 
in $\oos$ or $\oou$ as described in definition \ref{stan}.

%3) $p$ is a corner of a scalloped region
%as described in section \ref{prelim}.
%%Then a master sequence for $p$ is obtained as
%described in lemma \ref{infin}.

%In addition the only conclusion that applies if there are
%no perfect fits is conclusion 1).
\label{options}
\end{proposition}

For unbounded pseudo-Anosov flows there is one more possibility,
which occurs when there is an infinite set of leaves
of say $\oos$ which are pairwise non separated from each other.
For simplicty of exposition we do not describe it here as it
will not be needed for this article.

If $l$ is a leaf of $\oos$
or $\oou$, then we denote by $\partial l$ the collection of
ideal points of rays of $l$.
So if $l$ is a $p$-prong leaf, it has $p$ ideal points.

\section{Properties of perfect fits and convergence of leaves}

In this section we rule out certain structures involving
perfect fits and open product sets. We also discuss the possible
limits of a sequence of points in a sequence of nested leaves.
These properties will be used many times and in fundamental ways 
to prove the main
results of this article.

The following is a closing lemma that works for pseudo-Anosov flows.

\begin{proposition}{(closing lemma)}{(Mangum \cite{Man})} Suppose that $\Phi$ is
a pseudo-Anosov flow in $M^3$ closed. Let $x$ be a non singular 
point. There is $\epsilon_0 > 0$, a priori depending on
$x$ so that if $\gamma$ is a segment in an orbit of $\Phi$
with initial and end points less than $\epsilon < \epsilon_0$ from
$x$, then $\gamma$ is shadowed very close by a closed orbit.
The closeness goes to zero as $\epsilon$ goes to zero.
An analogous statement holds if $x$ is singular
and one assumes that the orbit return to the same sector defined
by $x$.
\label{closing}
\end{proposition}

The $\epsilon_0$ a priori depends on $x$. This is because if
$x$ is very close to a singular orbit, we may have a segment which 
returns very close, but to distinct sectors of the singular
orbit. This cannot be shadowed by a closed orbit.

\begin{define}{}{}
A product open set is an open set $Y$ in $\mi$ or $\oo$ so that
the induced stable and unstable foliations in $Y$ satisfy the
following: every stable leaf intersects every unstable leaf.
\end{define}

For example, if stable leaves $A, B$ both intersect unstable
leaves $C, D$, then the four of them bound a compact rectangle 
whose interior is 
a product open set. In our explorations we will also
consider the closure of the open product set (in $\mi$ or $\oo$) which is the
open set union its boundary components.

We stress that {\underline {product open sets}}
and {\underline {product regions}} in general are not the same,
and the second is a subset of the first. Any pseudo-Anosov flows has
product open sets, but only suspensions have product regions.

\begin{define}{($(3,1)$ ideal quadrilateral)}{} \
A $(3,1)$ ideal quadrilateral in $\mi$ (or $\oo$) is a product open
set $Q$ so that the boundary has four pieces: two pieces $S_1, S_2$ 
are contained in stable leaves,  and two pieces $U_1, U_1$
are contained in 
unstable leaves. In addition they satisfy: $S_1$ and $U_1$ are half
leaves and $S_2, U_2$ are line leaves;
$S_1$ intersects $U_1$, but $S_2$ makes a perfect fit with both $U_1$ and $U_2$
and $S_1$ also makes a perfect fit with $U_2$.
The collection $S_1, S_2, U_1, U_2$ forms the sides of $Q$.
See fig. \ref{qg1}, a.
A $(4,0)$ ideal quadrilateral
is similarly defined.
\end{define}

Notice that the sides of $Q$ are not exactly the same as the boundary 
components of $Q$. In particular $S_1 \cup U_1$ is a boundary component
of $Q$.
In addition we will abuse notation and sometimes also call $S_1$ 
the full stable leaf containing it.
If $Q$ is an ideal $(3,1)$ quadrilateral the following happens:
if an unstable leaf $Z$ intersects $S_2$ then it also  intersects
$S_1$ and similarly if a stable leaf $L$ intersects $U_2$ then it also
intersects $U_1$.
In an ideal $(3,1)$ quadrilateral there is one actual corner in the
boundary and 3 perfect fits. 
We show that these objects do not
exist for any pseudo-Anosov flow.

\begin{proposition}{}{}
Let $\Phi$ be an arbitrary pseudo-Anosov flow. 
Then there are no $(3,1)$ or $(4,0)$ ideal quadrilaterals.
\label{nothreeone}
\end{proposition}

\begin{proof}{}
This is a rigidity result.
We do the proof for $(3,1)$ ideal quadrilaterals.
The proof for $(4,0)$ quadrilaterals is similar and easier, and is left to the
reader.
Suppose there is a $(3,1)$ ideal quadrilateral $Q$ with stable sides
$S_1, S_2$ and unstable sides $U_1, U_2$.
Assume that $S_1$ intersects $U_1$ in an orbit $\gamma$ and all other
pairs of stable/unstable sides make a perfect fit.

First suppose that one side of $Q$ is periodic. 
Without loss of generality
assume that $S_1$ is left invariant under $g$ non trivial in
$\pi_1(M)$. 
What we mean here is that the full stable leaf containing $S_1$ is 
left invariant by $g$.
By taking powers we may assume that $g$ preserves orientation
in $\oo$ and leaves invariant all the prongs in the stable leaf containing
$S_1$. It follows that the perfect fit $S_1$ with $U_2$ is taken
to itself so $g(U_2) = U_2$. Going around the quadrilateral $Q$,
this in turn implies that $g(S_2) = S_2$ and $g(U_1) = U_1$ $-$ here again
we are considering the full unstable leaf containing $U_1$.
Hence $g(U_1 \cap S_1) = U_1 \cap S_1$ or $g(\gamma) = \gamma$.

Let $\gamma_2$ be the periodic orbit in $S_2$.
Then $g(\gamma_2) = \gamma_2$.
Since $\wlu(\gamma_2)$ intersects $S_2$ then it also intersects
$S_1$. Let $\beta = \wlu(\gamma_2) \cap S_1$.
Since both $S_1$ and $\wlu(\gamma_2)$ are $g$-invariant, then
$g(\beta) = \beta$. But then $g$ leaves invariant two distinct
orbits $\gamma$ and $\beta$ in $S_1$. This is a contradiction and
shows this cannot happen.

A very similar proof deals with the case where for instance 
$S_2$ is periodic. 
%So we may assume that none of the sides are periodic.

\vskip .1in
From now on, suppose that no side of $Q$ is periodic. 
The proof will be done by a perturbation/rigidity method. We will look
at forward/backward accumulation points
of an orbit in $M$. This orbit is the projection of an orbit in a side of $Q$.
Then we will use the closing lemma to produce closed orbits and
associated covering translations. The covering translations perturb
the ideal quadrilateral $Q$ and we analyse the perturbation.

Let $\alpha$ be an orbit in $S_2$ and let $\beta = \pi(\alpha)$.
Then $\beta$ is not a periodic orbit of $\wwp$ and neither is
it in the stable or unstable leaf of a periodic orbit.
Since there are finitely many singular orbits, each of which is periodic,
it follows that 
the forward orbit of $\beta$ also limits in a non singular point
$p_0$.  We choose an initial point 
$q \in \alpha$ so that $p = \pi(q)$ is very close to 
$p_0$. 
In order to that we most likely move forward along the orbit $\alpha$ and 
this changes the point of view in $\mi$. Notice however that flowing
forward increases unstable distances along $\widetilde \Lambda^u(\alpha)$ 
and therefore $q$ is not close to $S_1$. This is crucial.
Now choose

$$t_n \rightarrow \infty,  \ \  {\rm with}  \ \ 
g_n(\wwp_{t_n}(q)) \ = \ q_n, \ {\rm and} \ q_n  \rightarrow  q_0
\ {\rm with} \ \pi(q_0) \ = \  p_0.$$

\noindent
We can choose all
$\pi(\wwp_{t_n}(q)) = \Phi_{t_n}(p)$ very close to
$p_0$. Since $p_0$ is not singular we can apply the closing lemma:
 it follows that the $g_n$ are associated to periodic orbits
$\alpha_n$ with points very close to $q$. This means that $g_n(\alpha_n) 
= \alpha_n$ and $\alpha_n$ has a point $z_n$ very close to $q$.
By changing the initial point $q$ and taking a subsequence of $(t_n)$
if necessary, we may assume that every $g_n$ preserves the local
transverse orientations to both $\wls$ and $\wlu$.

We now analyse how this covering translation $g_n$ acts on
the $(3,1)$ ideal quadrilateral $Q$. We refer to fig.
\ref{qg1}, b.

\begin{figure}
\centeredepsfbox{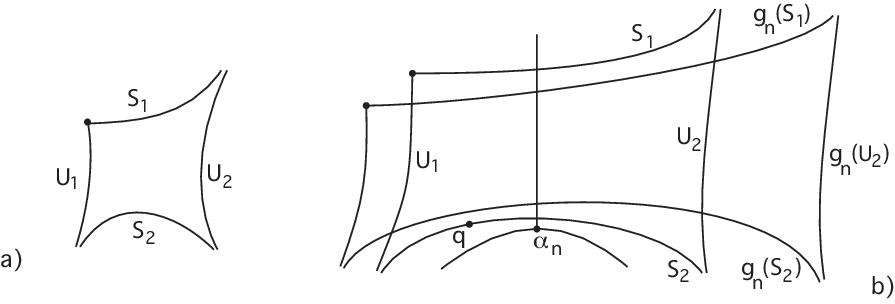}
\caption{a. A $(3,1)$ ideal quadrilateral. Here $S_1$ and $U_1$ intersect,
but $S_1, U_2$ makes a perfect fit, as well as $S_2, U_1$ and
$S_2, U_2$.  \ b) The use of covering translations to slightly
perturb  $(3,1)$ ideal quadrilaterals.}
\label{qg1}
\end{figure}

\vskip .1in
\noindent
{\bf {Case 1}} $-$ Assume up to subsequence that no $\alpha_n$ is in $Q$.

Since the periodic orbit $\alpha_n$ has a point very $z_n$ very close to 
$q \in S_2$ 
%$g_n(\wwp_{t_n}(p))$, then
then $\wlu(\alpha_n)$ intersects 
$S_2$ and hence  $g_n(S_2)$. Because $Q$ is a $(3,1)$ ideal
quadrilateral, this implies that $\wlu(\alpha_n)$
also intersects $S_1$ and $g_n(S_1)$. 
Notice also that $g_n(S_2)$ has the point $g_n(\wwp_{t_n}(q))$
which is very close to $q$. In particular $g_n(S_2)$ separates
$\wls(\alpha_n)$ from $S_1$, because $S_1$ is not
very close to $q$.
This uses the property on local orientations above.

Notice that $g_n$ acts in the backward
flow direction in $\alpha_n$ hence $g_n$ acts as as an expansion in the ordered
set of orbits in 
$\wlu(\alpha_n)$. Hence $g_n(S_2)$ is
farther from $\alpha_n$ than $S_2$ is. 
But since it separates $S_1$ from $\wls(\alpha_n)$ it follows
that $g_n(S_2)$ intersects the ideal quadrilateral $Q$.
We refer to figure \ref{qg1}, b for this situation.
%The local picture near $q_0$ (that is non singular) also implies that
%$w = g_n(\wwp_{t_n})(q)$ is inside the quadrilateral $Q$. In any
Since $g_n(S_2)$ intersects $Q$, then $g_n(S_2)$ intersects $U_1$ and $U_2$. 
Since $g_n(S_2)$ makes a perfect fit with $g_n(U_2)$ it follows that
$g_n(U_2)$ does not intersect the quadrilateral $Q$. Finally
since $g_n(U_2)$ makes a perfect fit with $g_n(S_1)$ and $g_n(S_1)$ intersects
$\wlu(\alpha_n)$, we obtain that $g_n(S_1)$ intersects
$Q$. Then $g_n$ contracts the interval of orbits
$\wlu(\alpha_n) \cap Q$ between $S_1$ and $S_2$. 
This forces another periodic orbit
in $\wlu(\alpha_n)$, besides $\alpha_n$.
This is a contradiction. This finishes the analysis of Case 1.

\vskip.1in
\noindent
{\bf {Case 2}} $-$ Assume up to subsequence that every $\alpha_n$ is in $Q$.

Since for every $n$ the orbit $\alpha_n$ has points close to $q$, we may assume
that the sequence $(\alpha_n)$ converges to an orbit $\alpha$. Let 
$V = \wlu(\alpha)$. Since the lengths of orbits $\pi(\alpha_n)$ converge to
infinity then the sequence

$$(\wlu(\alpha_n) \cap g_n(S_2))$$

\noindent
escapes compact sets in $\oo$.
Let $T_n$ be the component of $g_n(S_1) - \wlu(\alpha_n)$ which makes a perfect with
$g_n(U_2)$ and let $T'_n$ the component of 
$g_n(S_1) - \wlu(\alpha_n)$
 which 
makes a perfect fit with $g_n(U_1)$. If either sequence $(T_n)$ or $(T'_n)$
escapes compact sets in $\oo$ then we produce a product region in $\oo$, contradiction.

It follows that both sequences $(T_n)$ and $(T'_n)$  converge to a collection 
of leaves of $\oos$. No leaf of $\oos$ belongs to both sets of limits because 
the sequence $(\wlu(\alpha_n) \cap g_n(S_1))$ escapes in $\oo$. This is the crucial
fact here. Therefore the union of the leaves in the limits has at least two leaves.
All of these leaves are non separated 
from each other in between $U_1$ and $U_2$. 
By Theorem \ref{theb} these
leaves are periodic and left invariant by a non trivial covering translation $f$.
In addition one of the leaves in the limit of $(T_n)$ makes a perfect fit with
$V = \wlu(\alpha)$. Hence $f(V) = V$ also. Notice that the union of the limits
of $(T_n)$ and $(T'_n)$ is a finite set. It is the maximal collection of leaves non
separated from these leaves between $U_1$ and $U_2$.

	The sequence $(g_n(U_2))$ converges to a collection of unstable leaves. 
Because of the finiteness of the number of leaves in the limit of the sequences
$(T_n)$ and $(T'_n)$,
it follows that one of the leaves 
in the limit of $(g_n(U_2))$ makes a perfect fit
with a leaf in the limit of the sequence  $(T_n)$. Hence the leaves in the limit of $(g_n(U_2))$ are 
also invariant under $f$. Finally consider the component $Y_n$ of $(g_n(S_2) - \wlu(\alpha_n))$
which makes a perfect fit with $g_n(U_2)$. 
Since there are no product regions, then as before  the sequence  $(Y_n)$ converges to a collection of leaves;
and 
one of which makes a perfect fit with a leaf in the limit of $(g_n(U_2))$. Hence
all leaves in the limit of $(Y_n)$ are also invariant under $f$. In addition one
leaf in the limit of $(Y_n)$ intersects $V = \wlu(\alpha)$ and let $\beta$ be this intersection.

Let $Q_1$ be the product open set bounded by the limit leaves of the sequence $(T_n)$, the limit
of $(g_n(U_2))$, the limit of $(Y_n)$ and the half leaf of $\wlu(\beta)$ bounded by $\beta$ 
and making a perfect fit with a leaf in the limit of $(T_n)$. This region has all
the boundary components left invariant by $f$. It has one corner in $\beta$ and all other
interactions between ``consecutive" leaves in the boundary of $Q_1$ are either 
 by perfect fits or non separated leaves. This situation was disallowed
in the beginning of the proof of Proposition \ref{nothreeone} which dealt with 
the periodic case.
This finishes the analysis of Case 2.

This finishes the proof of proposition \ref{nothreeone}.
\end{proof}

\noindent
{\bf {Remark}} $-$ What happens if one tries to apply the rigidity argument
of the second possibility of this proposition to the case when $S_2$ is periodic?
Then for any orbit $\beta'$ in $\pi(S_2)$ it limits forward only in the periodic orbit
in $\pi(S_2)$, hence the perturbation obtained by $g_n$ sends $S_2$ to itself and likewise
for $U_2$ and $S_1$. So this in itself yields nothing. One can only apply the perturbation 
technique for non periodic leaves.
\vskip .08in

We can further eliminate more exotic regions in $\oo$ in the case of
bounded pseudo-Anosov flows.

\begin{lemma}{}{}
Suppose that $\Phi$ is a bounded pseudo-Anosov flow.
Then there are no product open sets $Q$ in $\oo$ so that the boundary has a single
corner or no corner.
\label{exotic}
\end{lemma}

%\noindent
%{\bf {Remark}} $-$ Before we prove this, notice that a regular rectangle is a product open
%set with 4 corners. A perfect fit is described as a product open set with 3 corners and
%one ``ideal" corner, which is the perfect fit. A lozenge is a product region with two
%opposite corners and two ``ideal" corners. A $(3,1)$ ideal quadrilateral would be a product
%open set with one corner. The previous lemma shows that it does not exist.

\begin{proof}{}
Suppose that there is a product open set $Q$ with one corner.
Here a corner is an orbit which is in two leaves, one stable and one
unstable which have half leaves contained in the boundary of $Q$.
This is the case of a $(3,1)$ ideal quadrilateral which has a single
corner. The components of boundary of $Q$ which do not have corners
are full line leaves in their respective leaves of the stable or
unstable foliations.

Since there is a corner in $Q$ $-$ the region $Q$ is not the whole orbit space. 
Since every stable leaf intersects every unstable leaf in $Q$, the induced
stable and unstable foliations in $Q$ have leaf space homeomorphic to $\rrrr$.
There are 
4 possible sections of the boundary of $Q$ corresponding to escaping in stable/unstable
directions. At most two sections are made of stable leaves (or subsets thereof)
and at most two sections are made of unstable leaves.
If any of these sections
has more than one boundary leaf then these leaves are non separated from 
each other in their respective leaf spaces.
By Theorem \ref{theb} these boundary leaves are periodic and
are both left invariant by some non trivial $g$ in $\pi_1(M)$. 
Consider the set of non separated leaves from these leaves.
Under the bounded hypothesis, 
theorem \ref{theb} implies that this set is finite $-$ any two non separated
leaves are connected by a chain of adjacent lozenges hence by a chain
of perfect fits, so this is bounded.
We denote this set of non separated leaves by 
$B_1, ..., B_n$.
Since there is only one corner, there is a full line leaf of some
$B_j$ contained in the boundary of $Q$. We still denote this by $B_j$.
We assume that this is the last one contained in the boundary of $Q$. Since 
there is no product region and the set of non separated leaves from
any leaf is bounded, it follows that $B_j$ makes a perfect
fit with another leaf in the boundary of $Q$. As in the previous
proposition this leaf is also left invariant by $g$. Going around
the boundary of $Q$ we arrive at a contradiction as in the
proof of Proposition \ref{nothreeone}.
%We claim that the structure of the set of non separated leaves implies that a stable 
%boundary $A_1$ of $Q$ makes a 
%perfect fit with $B_1$ and another stable boundary leaf
%$A'_1$ makes a perfect fit with $B_n$. 
%%Then there must be 2 corners. 
%If no such $A_1$ or $A'_1$ exist then we produce a
%product region in $\oo$ and $\Phi$ is topologically equivalent to a suspension Anosov flow, contrary
%to assumption (obviously there are easy counterexamples to the lemma for such flows).
%In the same way $B_n$ makes a perfect fit with another stable leaf $A'_1$ in the boundary of $Q$. 
%Then
%both $A_1$ and $A'_1$ are left invariant under $g$ and are connected by a chain of lozenges 
%$C_1, ..., C_{2n}$. The number of lozenges in the chain it twice the number of the leaves $\{ B_i \}$.
%Let $D_1$ be the other unstable leaf in the boundary of $C_1$ and $D_2$ the
%other unstable leaf in the boundary of $C_{2n}$. Since the stable and foliations
%in $Q$ form a product structure, it follows that neither $D_1$ nor $D_2$
%intersects $Q$. 
%But this forces 2 corners in $\partial Q$: one in $A_1 \cap D_1$ and one in 
%$A'_1 \cap D_2$.  This is a contradiction to the assumption of the lemma.

This reduces the analysis to the case that all 4 boundary areas have only one leaf. 
Since there are no product regions, this forces perfect 
fits producing a $(3,1)$ ideal quadrilateral, contradiction to the previous
proposition.

The same proof works if $Q$ has no corners.
This finishes the proof of lemma \ref{exotic}.
\end{proof}

\noindent
{\bf {Remark}} $-$ The theorem is false without the bounded hypothesis.
It is false for example for suspension Anosov flows. Just take any 
orbit $\gamma$ of $\wwp$ and consider a sector $Q$ defined by
$\gamma$. The sector is a product open set with only one corner in the boundary.
In addition the theorem is also false if there is a collection of
leaves non separated from each other which is infinite. This generates
a scalloped region as described in \cite{Fe9}. One can get a ``quarter"
of this scalloped region which is a product open set and has only
one corner in the boundary.

It turns out that the bounded hypothesis is not needed for regions
without corners, except for the whole of  $\mi$ in the case of suspension Anosov flows.
Since we do not need this fact in this article we do not prove it.
\vskip .05in

The next result will be used many times in this article.

\begin{proposition}{(periodic double perfect fits)}{}
Suppose that two distinct half leaves $S_1, S_2$ of a slice of a leaf $S$ of $\wls$ (or $\wlu$) make a
perfect fit respectively with  $A, B$, which are leaves of $\wlu$ (or $\wls$). Suppose that $S$ does
not separate $A$ from $B$. Then $A, B$ and $S$ are all periodic and leaf invariant by a non trivial
covering translation $g$.
\label{doublefit}
\end{proposition}

\begin{proof}{}
The content of the proposition is in the 
the first conclusion (periodic behavior),
as invariance of perfect fits by the action of $\pi_1(M)$ leads
to the second conclusion. 
In fact this shows that if one of $A, B, S$ is periodic, then
so are the others.

Suppose by way of contradiction that
$A, B, S$ are not periodic.
Without loss of generality assume that $S$ is a stable leaf
and hence $A, B$ are unstable leaves.
By assumption $A, B, S$ do not have singularities and the
slice leaf of $S$ is $S$ itself. Also since $S$ is not singular, there
is a stable leaf $L$ near enough $S$ and intersecting both
$A$ and $B$, as $A, B$ make perfect fits with $S$ and are in
the same complementary component of $S$.
We first will build a structure similar to two adjacent lozenges
with some sides in $A, B, S$.
The preliminary step is to eliminate singularities.
Fix the leaf $L$.

\begin{figure}
\centeredepsfbox{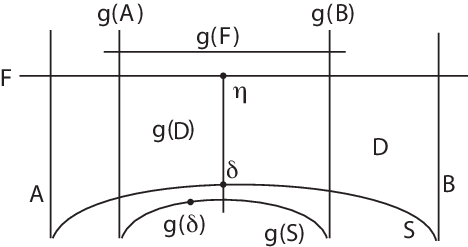}
\caption{A double perfect fit. This figure illustrates the impossibility
of a singular orbit $\eta$ so that its stable leaf $F$ intersects both
$A$ and $B$ and $\eta$ is between $F \cap A$ and $F \cap B$.}
\label{qg2}
\end{figure}

\vskip .1in
\noindent
{\underline {Claim}} $-$ There is no singular stable leaf $F$ 
intersecting both $A$ and $B$ with a singular orbit $\eta$ which is between
$A \cap F$ and $B \cap F$ in $F$.

Suppose by way of contradition there is such a leaf $F$.
Let $F_0$ be the flow band in $F$ from $A \cap F$ to $B \cap F$.
Then $F_0$, half leaves of $A$ and $B$ respectively and $S$ bound
an open region $D$ in $\mi$. If there is prong $F_1$ of $F$ entering
$D$ then $F_1$ cannot intersect $A$ or $B$ as stable and unstable leaves
intersect in at most one orbit. Clearly $F_1$ cannot intersect $S$ as both
are stable leaves. Then this prong $F_0$ would not be properly embedded
in $\mi$, contradiction. Hence there is no such prong $F_1$. In particular
this also shows that the region $D$ has no singularities (because of
perfect fits).  Let $\eta$ be the singular orbit in $F$, so $\eta \subset F_0$.

There is a prong $E_0$ of $\wlu(\eta)$ intersecting
$S$ in an orbit $\delta$, see fig. \ref{qg2}.
Let $v$ a point in $\delta$ and $u = \pi(v)$. Flow $u$ forward 
and let $u_0$ be an accumulation point in $M$. By adjusting the initial point
$v$ we may suppose that $u$ and $\Phi_t(u)$ are extremely close 
$-$ and close to $u_0$. Hence there is a covering translation
$g$ sending a lift $\wwp_t(v)$ very near $v$. Then $g(E_0)$ 
intersects the region $D$. 
%Assume without loss of generality that $g(\delta)$ is outside
%the region $D$. 
As in the previous proposition we can assume that $g$ preserves
the local transverse orientation to $\wls, \wlu$ near $\delta$
$-$ notice that $\delta$ is not singular by hypothesis.

Notice that  $g(S) \not = S$ as $S$ is not periodic. In addition
$g(E_0) \cap E_0 = \emptyset$ because we flow the orbit of $u$ forwards.
Since $D$ has not singularities, then $g(\eta)$ is outside $D$. 
But this forces $\eta$ to be contained in $g(D)$, which is a contradiction,
because $\eta$ is singular and $g(D)$ does not have singularities.
This proves the claim.

\vskip .1in
Now parametrize the stable leaves intersecting $B$ beyond $L$ as
$\{ L_t, t \geq 0, L_0 = L \}$. 
Beyond means that $L_0$ separates $L_t, t > 0$ from $S$.
%No $L_t$ intersects $R$. 
For $t$ small,
$L_t$ intersects $A$. There are several possibilities: 
For such $t$, let $V_t$ be the region in $\mi$
bounded by $L_t, A, B$ and $S$. This is a product open set.

\vskip .1in
\noindent
{\underline {Possibility 1}} $-$ For every $t$, $L_t$ intersects $A$
and $L_t \cap A$ escapes in $A$ as $t \rightarrow \infty$.

If the collection
$(L_t)$ escapes compact sets in $\mi$ as $t \rightarrow \infty$, then
this produces a product region in $\mi$ 
and $\Phi$ is topologically conjugate 
to a suspension Anosov flow, contradiction to assumption.
If the collection $(L_t)$ converges to a set $\{ U_i, i \in I \}$ 
of stable leaves, then 
$V = \cup_{t \geq 0} V_t$ is a product open set with boundary
$A, B,  S$ and (possibly a proper subset of)
 $\{ U_i, i \in I\}$. This product open set 
has no corners and is disallowed
by lemma \ref{exotic}.

\vskip .1in
\noindent
{\underline {Possibility 2}} $-$ There is $t_0$ so that 
$(A \cap L_t)$ escapes compact sets in $A$ as $t \rightarrow \infty$.

Then $V = \cup_{t < t_0} V_t$ is a product open set with only 
one corner $B \cap L_{t_0}$. This is also disallowed by proposition
\ref{nothreeone} and lemma \ref{exotic}.
The same holds if every $L_t$ intersects $A$ and
$\lim_{t \rightarrow \infty} (A \cap L_t) = \beta$ and $\beta$ is
an orbit in $A$. 

The final possibility is the most intricate:

\vskip .1in
\noindent
{\underline {Possibility 3}} $-$ There is $t_0 > 0$ so that for any $t < t_0$,
the leaf $L_t$ intersects $A$ but $L_{t_0}$ does not. Also 
$\lim_{t \rightarrow t_0} (L_t \cap A) = \alpha$ is an orbit in $A$.

Let $U_0 = \wls(\alpha)$ and
$U_1 = \wls(L_{t_0} \cap B)$.
Then $U_0, U_1$ are stable leaves which are non separated from each other.
In particular $U_0, U_1$ are periodic and are left invariant
by some non trivial $h \in \pi_1(M)$.

Let $V = \cup_{t < t_0} V_t$ which again is a product open set in $\mi$.
But now $V$ has 2 corners $\alpha$ and $\beta$. In fact this is quite 
possible for $V$ could be the union of two adjacent lozenges and 
their common side. 
We consider the set of unstable leaves intersecting $S$ starting from $A$.
There is a first
unstable leaf $E$ intersecting $S$ and  not intersecting $U_0$. Then $E$ and $U_0$ make a perfect 
fit. This implies that $E$ is periodic and left invariant by $h$.

 Let $\gamma = E \cap S$.
We want to prove that $\gamma$ is periodic, so we assume that $\gamma$ is not
periodic and use
a perturbation argument as done previously.
Let $v \in \gamma$ and consider $u = \pi(v)$. Flow $u$ forward and consider
forward limit points. Since $\gamma$ is not periodic, then as done before,
we can assume that $v$ is chosen
so that there is $g \in \pi_1(M)$ with $g(\gamma)$ very close to and distinct
from $\gamma$ and in addition $g(\wlu(\gamma)), \wlu(\gamma)$ are distinct.
For simplicity we only do
the proof when $g(\gamma)$ is outside of $V$ and there are only 2 leaves
in the boundary of $V$ non separated from $U_0$. The other cases are similar.

\begin{figure}
\centeredepsfbox{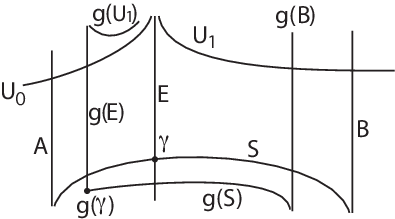}
\caption{Another perturbation of a double perfect fit.}
\label{qg3}
\end{figure}

Without loss of generality assume that $E = \wlu(\gamma)$ 
separates $B$ from $g(E)$. Notice that $g(B)$ makes a perfect with $g(S)$
and so $g(B)$ is separated from $g(E)$ by $E$ and $g(U_1)$ is non separated
from a leaf making a perfect fit with $g(E)$, see fig. \ref{qg2}.
This shows that $g(B)$ cannot intersect $g(U_1)$, which is a contradiction.

The contradiction was obtained from the assumption that $\gamma$ was not periodic.
We conclude that $\gamma$ is periodic. In particular $S$ is periodic
and this also implies that $A, B$ are periodic.
The rest of the proposition follows easily.
This finishes the proof of proposition \ref{doublefit}.
\end{proof}

\begin{lemma}{(escape lemma - bounded version)}{}
Let $\Phi$ be a bounded pseudo-Anosov flow. 
Let $(x_i), i \in {\bf N}$ be a sequence
in $\oo$ with $(x_i)$ converging to $x$ in
$\oo \cup \partial \oo = \cd$. Let $y_i \in
\oos(x_i) \cup \partial \oos(x_i)$ so that
$(y_i)$ converges to $y$ in $\oo \cup \partial \oo$.
Let $l_i = \oos(x_i)$. 

\begin{itemize}

\item Suppose $x$ is in $\oo$. If $y$ is in $\oo$ then 
$\oos(y)$ and $\oos(x)$ are non separated from each
other in the leaf space of $\oos$. If $y$  is in
$\partial \oo$ then $y \in \partial l$ where 
$l \in \oos$ and $l$ non separated from $\oos(x)$.

\item Suppose that $x$  is in $\partial \oo$. If there
is a subsequence $(l_{i_k})$ of $(l_i)$ which escapes
compact sets in $\oo$, then $(l_{i_k} \cup \partial l_{i_k})$
converges to $x$ in $\oo \cup \partial \oo$ and
hence $y = x$. Otherwise assume up to subsequence
that $(l_i)$ converges to a collection
$\{ s_j, j \in I \}$ of non separated leaves
in $\oos$.
Then $x$ is an ideal point of some $s_j$ and
$y \in \bigcup_{j \in J} (s_j \cup \partial s_j)$.
In particular corresponding  ideal points of $\oos(x_i)$ 
converge to an ideal point of one of $s_j$.

\end{itemize}

\label{trapbound}
\end{lemma}

\begin{proof}
Suppose first that $x$ is in $\oo$. If $y$ is in $\oo$ then
the statement is obvious, because for big $i$, then $y_i$ is in $\oo$ and
therefore in $\oos(x_i)$. Let then $y \in \partial \oo$.
If there is a subsequence $(x_{i_k})$ with
$\oos(x_{i_k})$ constant, then the conclusion follows
immediately. Suppose then up to subsequence that
the $\{ \oos(x_i) \}$ are all distinct and
forming a nested sequence of leaves in $\oos$,
all in a complementary component of $\oos(x)$ in $\oo$. 
Since $\Phi$ is bounded there are finitely
many leaves of $\oos$ non separated from $\oos(x)$ and we can
order then so that $z_1$ is the first one and $z_k$ (for some $k$) is the last one.
Let $R$ be the complementary region in $\oo$ 
of the union of the $\{ z_j \}$ which
contains all the $l_i$.
Theorem \ref{theb} shows that consecutive $z_j, z_{j+1}$ are connected
by two adjacent lozenges which have a common side in an unstable leaf $u_j$.
Let $a_j$ be the common ideal point of these 3 leaves. Let 
$a_0$ be the ideal point of $z_0$ which is not $a_1$ but is still an ideal
point of the region $R$. Similarly let $a_k$ be the corresponding ideal
point of $z_k$.
If $k = 1$, then
$a_0, a_1$ are ideal points of $z_1$. 
Choose master sequences defining $a_0, a_k$. Since $z_1$ is the first leaf non
separated from $\oos(x)$, then the master sequence can be chosen to have 
at most one more piece after crossing $z_1$. 
This means that there may be an unstable leaf $u$ contained in $R$ with ideal point
$a_0$, hence we need one more piece in the polygonal path. But there cannot be
another stable leaf in $R$ with ideal point $a_0$.
In particular the ideal
points of $\oos(x_i)$ in this direction converge to $a_0$. So if $y_i$ is in
$\partial \oo$ this shows that the sequence $(y_i)$ converges to either $a_0$ or $a_k$. 

Suppose then that up to subsequence $(y_i)$ is contained in $\oo$.
Choose master sequences for each $a_j, 0 < j < k$. Let $c_j$ be one such
polygonal path in the master sequence for $a_j$. Then the intersection
$U_{c_j} \cap R$ has boundary made up of 5 pieces: 1) one ray in $z_j$ with
ideal point $a_j$,  
2) one finite segment in an unstable leaf with endpoint in $z_j$, 3) one finite
segment in a stable leaf intersecting $u_j$ in the interior, 
4) one finite segment in an unstable leaf with an endpoint in $z_{j+1}$, 
and 5) one ray in $z_{j+1}$ with ideal point $a_j$.
The leaves $l_i$ are converging to the collection $\{ z_j \}$ so they intersect
the $U_{c_j}$, and also corresponding subsets for $a_0$ and $a_k$. If the 
$y_i$ are not in the union of the $U_{c_j}$ (and respective sets for $a_0$
and $a_k$), then up to subsequence the $y_i$ converge to a point
in one of the $z_j$. But this point is $y$ assumed to be in $\partial \oo$,
contradiction. So the $y_i$ are eventually in the union above. Since 
the $U_{c_j}$ can be arbitrarily close to $a_j$ this shows that $y_i$ 
converges to one of $a_j, 0 \leq j \leq k$. This finishes the analyses
in this case.

\vskip .2in
Suppose otherwise that $x$ is in $\partial \oo$.
If $y \in \oo$, reverse the roles of $x$ and $y$
and obtain the same result.  Notice that $y \in \oo$ implies that no subsequence
$(l_{i_k})$ escapes compact sets in $\oo$.
So assume that $y \in \partial \oo$, that is, both $x$ and $y$ are in $\partial \oo$.
Suppose first that up to subsequence $(l_i)$ escapes compact sets
in $\oo$.
Consider a neighborhood of $x$ in
$\oo \cup \partial \oo$ defined by a  polygonal
path $c$ and let $V = V_c$ (see Definition \ref{canon} for 
$V_c$) containing
$x$ in its closure.
For $i$ big enough $\oos(x_i)$ has points $x_i$ in $V$,
since $(x_i)$ converges to $x$ in $\oo \cup \partial \oo$. 
If $\oos(x_{i_k}) \cup \partial \oos(x_{i_k})$ is 
not contained in $V$ for a subsequence 
$(i_k)$ and any $k$,  then
the sequence $(\oos(x_{i_k}))$ limits in a non 
trivial interval $K$ of $\partial \oo$. This is because $(l_{i_k})$ escapes compact sets in $\oo$. 
But then the interval $K$ cannot have any ideal points of leaves.
This is a contradiction \cite{Fe9}. Hence
$\oos(x_{i_k}) \subset V$ for $i$ big and
so $y_{i_k} \in \overline V \subset (\oo \cup \partial \oo)$. This implies that 
$y_i \rightarrow x$ as $i \rightarrow \infty$.

Finally suppose that  $(\oos(x_i))$ does not escape compact sets. Then 
assume up to subsequence that $(l_{i_k})$  converges
to $\{ s_j, j \in J \}$. Then
choose $z_{i_k} \in \oos(x_{i_k})$ with
$(z_{i_k})$ converging to $z$ in $\oo$ and apply the 
proof of the first case twice: once with  $(z_{i_k})$ in the place of $(x_i)$ 
and $(y_{i_k})$ in the place of $(y_i)$.
The second time apply it to $(z_{i_k})$ in place of $(x_i)$ and $(x_{i_k})$ in place of $(y_i)$. 
This completes the proof of the lemma.
\end{proof}

There is also an unbounded version of the escape lemma.
Since it will not be used in this article we do not
state or prove it.

\section{Metric properties and bounds on free homotopies}

A free homotopy bewteen closed orbits of pseudo-Anosov flows lifts to a chain of lozenges
in the universal cover, where all corner orbits project to closed orbits which are freely
homotopic to each other (alternatively reversing flow direction).

\begin{define}{}{}
A free homotopy is called indivisible if any minimal chain of lozenges associated to  it has
only one lozenge.
\end{define}

Minimal means that there is no backtracking in the chain of lozenges.

The goal  of this section is to relate pseudo-Anosov flows with the geometry of the manifold.
We will focus on the size of free homotopies of closed orbits, distance between corner orbits
of lozenges, and how chains of perfect fits produce chains of free homotopies.
This will show a strong connection with the geometry of the manifold.
The first result is fundamental.

%A {\em quadrant} of $p_0$ in $\mi$ is the closure of a component of $\mi - (\widetilde W^s(p_0)
%\cup \widetilde W^u(p_0))$.

\begin{theorem}{}{}
Let $\Phi$ be 
a pseudo-Anosov flow.
There is a constant $a_0$ depending only on the geometry of $M$ and on the flow $\Phi$, so that
if $\alpha, \beta$ are corner orbits of a lozenge $C$ in $\mi$, then they are a bounded
distance from each other: $d_H(\alpha,\beta) < a_0$, where $d_H$ denotes Gromov-Hausdorff
distance.
\label{thickness}
\end{theorem}

\begin{proof}{}
This result says that dynamics (pseudo-Anosov flow) is strongly connected with geometry
(distance in $\mi$).
The proof is by contradiction. 
We get lozenges with corner orbits farther and farther away and in the limit
we produce a product open set of the type ruled out
by lemma \ref{exotic}.

So suppose this is not true. Then there are lozenges $C_i$  with corner orbits
$\alpha_i, \beta_i$ and $p_i$ in say $\alpha_i$ with $d(p_i,\beta_i) \rightarrow \infty$ as 
$i \rightarrow \infty$.
Up to the action of $\pi_1(M)$ we will assume that $p_i \rightarrow p_0$ and $p_i$ is always in
the same sector of $p_0$. It follows that the sequence $\beta_i$ escapes compact sets in $\mi$ and
the projections to $\oo$ also escape compact sets as $i \rightarrow \infty$.

Let $L_i \subset \wls(p_i),  \
U_i \subset \wlu(p_i)$ be he sides of $C_i$ containing $p_i$.
Up to subsequence we assume that $L_i \rightarrow L_0$ contained in a line leaf of $\wls(p_0)$
and $U_i \rightarrow U_0$ contained in a line leaf of $\wlu(p_0)$ (and maybe other leaves as well).
Let $W$ be the sector of $p_0$ bounded by $L_0$ and $U_0$.
Suppose up to subsequences that both sequences $(\wls(p_i))$ and $(\wlu(p_i))$
are nested.

There are 3 possibilities.

\vskip .1in
\noindent
{\bf {Option 1}} $-$ $C_i \subset W$.

\begin{figure}
\centeredepsfbox{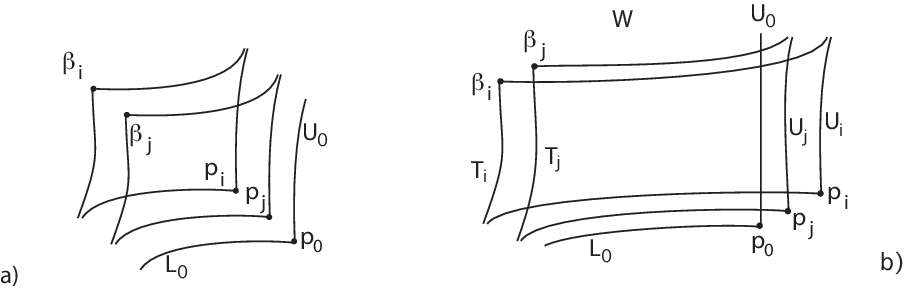}
\caption{The perturbation method applied to lozenges. In both cases $j > i$ so
$p_j$ is closer to $p_0$ than $p_i$. In case a) the lozenges $C_i$ are contained in 
the sector $W$. In part b) the lozenges $C_i$ are not contained in $W$. In particular in case b),
\ $\wlu(p_i)$ is disjoint from $W$ for all $i$.}
\label{qg4}
\end{figure}

Suppose up to subsequence that all $p_j$ are very near
$p_0$. In particular we may assume that $p_i$ is in $C_j$ if $i < j$.
Let $\beta_i$ be the other corner of $C_i$. The conditions above imply that $\beta_j$ is in
$C_i$ if $j > i$, see figure \ref{qg4}, a. 
There is an unstable leaf $Y$ intersecting $L_0$ and every $C_i$ and a stable
leaf $S$ intersecting $U_0$ and very $C_i$.
This is because $\wlu(\beta_i)$ cannot intersect $L_0$ and
$\wls(\beta_i)$ cannot intersect $U_0$.
Fix one such $i$.
This implies that the if $j > i$ then $\beta_j$ 
is contained in a fixed rectangle with sides 
contained in $S, Y, \wls(\beta_i), \wlu(\beta_i)$.
Therefore the sequence $(\beta_j)$ does 
not escape compact sets in $\oo$, contradiction to hypothesis.

\vskip .1in
\noindent
{\bf {Option 2}} $-$ $C_i$ is not contained in $W$ and $p_0$ is not in $C_i$, 
see figure \ref{qg4}, b.

Assume without loss of generality that $U_i \cap W = \emptyset$ as depicted in figure 
\ref{qg4}, b.
This implies that $L_i$ have half leaves contained in $W$. 
Let 

$$S_i \ \subset \ \wls(\beta_i) \ \ \ {\rm and} \ \ \ T_i \ \subset \ \wlu(\beta_i)$$

\noindent
be the other 2 sides of $C_i$. 
Here $p_i$ is not contained in $C_j$ for $j > i$ or $j < i$. Similarly
$\beta_i$ is not contained in $C_j$ for $j > i$ or $j < i$. 
Since the sequence $(U_i)$ converges to $U_0$  and $U_i \cap W = \emptyset$,
then the sequence $(T_i)$ is nested and contained in $W$. 
%If it escapes
%compact sets in $\mi$ one produces a product region, contradiction.
Therefore the sequence $(T_i)$ limits to a collection of leaves $\{ V_j \}, j \in J$,  at  least one of which  (let it
be $V_0$)  makes a perfect fit with $L_0$. 
Consider the sequence $(S_i)$. 
Each $S_i$ intersects $U_0$ and since $(p_i)$ converges to $p_0$, then
$(S_i \cap U_0)$ escapes in $U_0$.
If the sequence $(S_i)$ escapes compact sets in $\mi$, 
this produces an unstable  product region, contradiction.
We conclude that $(S_i)$ converges to a collection of leaves $\{ R_n \}, n \in N$, at least one (let it be $R_0$) makes a 
perfect fit with $U_0$. If one the limit leaves $\{ R_n \}$ intersects one of the 
other limit leaves $\{ V_j \}$ in $q_0$ this forces  the sequence of corners $(\beta_i)$ to converge to $q_0$. This contradicts
the assumption that $(\beta_i)$ escapes compact sets in $\mi$. Therefore these two sets of leaves
are disjoint as subsets of $\mi$. Let $A$ be the limit of the sequence $(C_i)$ of lozenges. It is a region
in $\mi$ bounded by $L_0, U_0, \{ R_j \}, j \in J, \{ R_n \}, n \in N$. It is a product open set
with a single corner $p_0$. This was disallowed by lemma \ref{exotic}.

The last option is:

\vskip .1in
\noindent
{\bf {Option 3}} $-$ $p_0$ is in $C_i$ for all $i$.

Here $p_i$ is in $C_i$ if $j > i$ and $\beta_j$ is not in $C_i$ if $j > i$.
Exactly as 
in Option 2 the sequence $(S_i)$ cannot escape in $\oo$ or it produces
an unstable product region. In fact the same argument shows in Option 3,
that the sequence $(T_i)$ cannot escape in $\oo$ either.
The limits of these sequences have one leaf making a perfect fit with $L_0$ 
and one leaf making a perfect fit
with $U_0$ and as in Option 2, no leaves in the limits of these sequences
can intersect. We obtain a contradiction as in Option 2.

This is the last possibility. Notice that in this last option 
$p_0$ cannot be singular.
This finishes the proof of proposition \ref{thickness}.
\end{proof}

This already has consequences for free homotopies:

\begin{corollary}{}{}
Suppose that periodic orbits $\alpha, \beta$ of a pseudo-Anosov flow $\Phi$ are 
freely homotopic by an indivisible free homotopy. Then there is a free homotopy
so that every point moves at most $a_0$ by the homotopy.
\end{corollary}

\begin{proof}{}
The orbits $\alpha, \beta$ lift to the corner orbits of a lozenge. Then apply
the previous theorem.
\end{proof}

We will now proceed to show that any chain of perfect fits of length $k$ generates a chain 
of free homotopies of length at least $k$. First we need to analyse forward limits of orbits.
We need the following technical result.

\begin{lemma}{}{}
Let $\gamma$ be an orbit of $\Phi$. Then either $\gamma$ is in the stable leaf
of a periodic
orbit or the forward limit set of $\gamma$ intersects infinitely many local stable sheets
near some point.
\label{localinf}
\end{lemma}

\begin{proof}{}
Fix $a_1 > 0$.
There are finitely many rectangle disks transverse to $\Phi$ whose union intersects every segment
of orbit of $\Phi$ of length at least $a_1$.
A rectangle  disk is a closed disk transverse to $\Phi$ so that the induced
stable and unstable foliations form a product structure in the rectangle disk.
Near a $p$-prong singularity
we can use for instance $2p$  rectangle disks. 
The leaves of the induced stable/unstable foliations  in the rectangle disks are 
called local stable/unstable sheets.
Notice that if an orbit intersects a stable sheet twice then this orbit is in the stable leaf of
a periodic orbit. 

Let $\alpha$ be an orbit in the forward limit set of $\gamma$. If $\alpha$ is dense we are done. Otherwise let 
$B$ be a minimal closed set in the closure of $\alpha$ where we may assume that 
$B$ is not a singular orbit. This is because if $\gamma$ only limits in a singular
closed orbit, then $\gamma$ is in the stable leaf of this closed orbit.
If $B$ is the closure of an almost periodic orbit \cite{Bowe}
we obtain the conclusion of the lemma. 
An almost periodic orbit $\delta$ is a non periodic orbit so that
its closure $\delta$ is a minimal set for the flow.
Suppose finally that $B$ is a closed orbit intersecting a  rectangle disk
$D$ in the interior. The orbit $\gamma$ keeps returning to $D$ very close to $B \cap D$.
We can assume $D$ is fairly small so that $B \cap D$ is a single point.
But $\gamma$ does not intersect $\ls(B) \cap D$ more than once
or else we would be in the first option of the lemma.
So the holonomy along $B$ pushes $\gamma$ closer to the unstable local sheet
$\ls(B) \cap D$ and farther from the 
stable local sheet $\ls(B) \cap D$ until the first return escapes $D$. 
Since $\gamma$ has to forward
limit on $B$, then it has to intersect infinitely many stable leaves. This finishes
the proof of the lemma.
\end{proof}

A very useful result in this section is the following:

\begin{proposition}{(from perfect fits to lozenges)}{}
Suppose that leaves $L, U$ form a perfect fit. 
Then there are orbits $\alpha$ in $U$ and $\beta$ in $L$ so
that $\alpha, \beta$ are the corners of a lozenge which has a side in $L$ and a side in $U$
which make the perfect fit above.
\label{perftoloz}
\end{proposition}

\begin{proof}{}
Up to taking a double cover assume that $M$ is orientable. Hence any local return map of the flow
restricted to a transverse disk in $M$ is orientation preserving.
Suppose that $L$ is a stable leaf, $U$ unstable.
Let $\alpha_0$ be an orbit of $L$ $-$ all orbits in $L$ are forward asymptotic.
If $\alpha_0$ is in the stable leaf of a periodic orbit $\alpha_1$, then $\alpha_1 \subset L$ and
$U$ is periodic and has a periodic orbit $\beta_1$ so that $\alpha_1, \beta_1$ are the corners
of a lozenge as desired.

Hence from now on we assume that $L$ is not a periodic leaf. By the previous lemma 
$\pi(\alpha_0)$ forward intersects infinitely many distinct  local stable sheets. 
Consider two of these
intersections which are sufficiently close, and not close to a singular orbit,
 so we can apply the closing lemma. These lift to $p_1, p_2$
in $\alpha_0$, with $p_i = \wwp_{t_i}(p_0)$ and $t_2 >> t_1 >> 0$.
In addition they satisfy $\pi(p_1), \pi(p_2)$ are very close and there is
$g$ in $\pi_1(M)$ so that $g(p_2)$ is very close to $p_1$ and $g(p_2)$ is in the component
of $M - \wls(p_1)$ containing $U$, see figure \ref{qg5}, a.

\begin{figure}
\centeredepsfbox{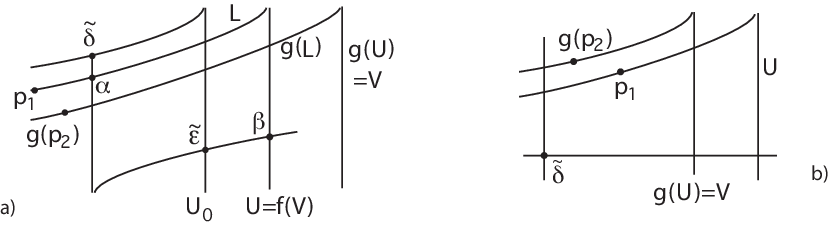}
\caption{Perturbing a perfect fit to produce a lozenge.
 The pictures  depict the situation in $\mi$, therefore 
$g(p_2)$ is very close to $p_1$. Recall that $\oo$ does not
have a metric. }
\label{qg5}
\end{figure}

Here we are using that the forward limit set of $\pi(\alpha_0)$ goes through infinitely many 
stable leaves nearby. 
So we can choose the sides accordingly. 
Notice that $p_1$ is not very close to $U$ or
else $\wls(p_1)$ would intersect $U$ and not make a perfect fit with it.
By the closing lemma the 
element $g$ of $\pi_1(M)$ is associated to a periodic orbit $\delta$. $\delta$ has a lift
$\widetilde \delta$ to $\mi$, which is very close to all of $p_1, p_2$ and $g(p_2)$.
Since $g$ is associated to the
backwards direction of the flow, then  $g$ acts as a contraction on the set of orbits of $\wls(\widetilde \delta)$
 and as an expansion on the set of orbits of $\wlu(\widetilde \delta)$.
This implies that $\widetilde \delta$ is in the component of $\mi - L$ \
{\underline {not}}
 containing $U$.

We first claim that $g(L)$ intersects $U$. We have to be careful. From the starting
point $p_0$ there is a distance $\epsilon  > 0$ so that if $x$ is in $\wlu(p_0)$
in the component of $\mi - \wls(p_0)$ containing $U$,  and
$d(x,p_0) < \epsilon$ then $\wls(x)$ intersects $U$. This is because 
$L, U$ make a perfect fit. But now the basepoint changed from $p_0$ to $p_1$. 
This is not a problem in this case, because going forward from $p_0$ to 
$p_1$ {\underline {increases}} the unstable distances. Therefore from the
point of view of $p_1$ even a stable leaf much farther away will intersect
$U$. It follows that $g(L)$ intersects $U$.

 Let $V = g(U)$ and let $f = g^{-1}$. Then $f(V) = U$ and $f(V)$ intersects
$g(L)$. Therefore $f^2(V)$ intersects $fg(L) = L$ and so intersects $g(L)$. 
This is again because $g$ is acting as an expansion in the set of unstable orbits
in $\wlu(\widetilde \delta)$.
It follows that the sequence
$(f^n(V))$ cannot escape compact sets in $\mi$ and  limits to a leaf $U_0$ making 
a perfect fit with $\wls(\widetilde \delta)$. This leaf is invariant under $g$
and so has a periodic orbit $\widetilde \epsilon$. Then $\widetilde \delta, \widetilde \epsilon$
are the corners of a lozenge $C_0$ which has sides also in $S = \wls(\widetilde \epsilon)$ 
and $T = \wlu(\widetilde \delta)$.
For $n$ big enough the leaf $f^n(V)$ intersects $S$ and therefore $U = f(V)$ also intersects $S$,
by $f$ invariance. 
Let

$$\beta \ = \ S \cap U \ \ \ {\rm and} \ \ \ \alpha \ = \ \wlu(\widetilde \delta) \cap L $$

\noindent
Then $\alpha, \beta$ are the corners of a lozenge $C$ which has sides contained in $L$ and $U$,
which make a perfect fit. The other two sides of $C$ are contained in $S$ and $T$ which
also makes a perfect fit.
This finishes the proof of the proposition.
\end{proof}

\noindent
{\bf {Remark}} $-$ Why was  lemma \ref{localinf} needed? The concern was that we would only get a situation as in 
figure \ref{qg5}, b. 
In this situation the leaves $f^n(V)$ move away from $\wlu(\widetilde \delta)$ when
$n$ increases. Recall that $\widetilde \delta$ is the periodic orbit.
A priori it could well happen that the sequence $(f^n(V))$ escapes compact sets in $\mi$ and therefore
we do not produce an unstable leaf invariant under $g$. 
Notice that the unstable band from $\wls(\widetilde \delta) \cap \wlu(g(p_2))$ to $g(p_2)$ along 
$\wlu(g(p_2))$ is larger than that from $\wls(\widetilde \delta) \cap \wlu(p_1)$ to
$p_1$ along $\wlu(p_1)$. This is because $p_2$ is flow foward of $p_1$ and unstable objects
grow forward. This is the reason it was necessary to have the alignment of $L = \wls(p_1)$ and
$g(L) = \wls(g(p_2))$ as in figure \ref{qg5}, a in order to get an unstable leaf
invariant under $g$.

%\begin{figure}
%\centeredepsfbox{q3.eps}
%\caption{A seemingly difficult situation to produce perfect fits.}
%\label{nopftolo}
%\end{figure}

\vskip .08in
Proposition \ref{perftoloz} immediately implies Theorem B:

\begin{corollary}{}{} Suppose that $\Phi$ is a pseudo-Anosov flow which has a perfect
fit $L, U$. A study  of the asymptotic behavior of orbits in $L, U$ produces  free 
homotopies between closed orbits of $\Phi$.
\end{corollary}

\begin{proof}{}
Use the setup of the previous proposition. The free homotopy is from 
$\delta = \pi(\widetilde \delta)$ to $\epsilon = \pi(\widetilde \epsilon)$. 
The orbits $\widetilde \delta$, $\widetilde \epsilon$ are
very close to $L, U$ respectively and are obtained by a shadowing process.
If $L$ (and hence $U$) is not periodic this process a priori produces infinitely many free homotopies.
This is because 
 longer
and longer segments in $\pi(L)$ are shadowed by a priori different closed orbits.
\end{proof}

We now prove the main metric property we will use:
A {\em forward ray} is just the set of points flow forward from
some point. Similarly one defines a {\em backward ray}.

\vskip .1in
\noindent
{\bf {Remark}} $-$ Notice that up to powers $\delta$ is freely homotopic to $\epsilon^{-1}$.
Therefore a flow forward ray in $\widetilde \delta$ is a bounded distance from 
a {\underline {backward}} ray in $\widetilde \epsilon$. This shows that the flow in
$\mi$ is metrically ``twisted" and is  not pointed  in a single direction. This is  opposed to the
situation when $\Phi$ has no perfect fits.

\begin{theorem}{}{}
Let $\Phi$ be a pseudo-Anosov flow in $M^3$ closed.
There is $a_1 > 0$ so that if $L \in \wls$ makes a perfect fit with $U \in \wlu$ then
for any forward ray in an orbit
in $L$, it is eventually a bounded distance $a_1$ from a backwards ray in $U$. More formally 
given $p$ in $L$, $q$ in $U$, there are $t_0, t_1$ in $\rrrr$ so that if

$$A \ = \ \wwp_{[t_0,\infty)}(p), \ \ B \ = \ \wwp_{(-\infty,t_1]}(q), \ \ \ \ {\rm then} \ \ \ \ 
d_H(A,B) \ < \ a_1,$$

\noindent
where $d_H$ is Hausdorff distance of closed sets in $\mi$.
\label{perfectbound}
\end{theorem}

\begin{proof}{}
This is a very strong result in that {\underline {not only}} a forward ray in $L$ is a bounded
distance from $U$, but rather a bounded distance from a single orbit in $U$.
By the previous proposition there are orbits $\alpha$ in $L$ and $\beta$ in $U$ so that
$\alpha, \beta$ are the corners of a lozenge. Then by theorem \ref{thickness} the Hausdorff distance
$d_H(\alpha,\beta) < a_0$ for some fixed constant $a_0$ depending only on $M$ and $\Phi$.
In particular a forward ray of $\alpha$ is $< a_0$ away from either a forward or
a backward ray of $\beta$. 

This is because orbits of $\wwp$ are properly embedded in $\mi$ and
the distance between the two ends goes to infinity.
Notice that the distance between points in a forward ray and points
in a backward ray of the same orbit
goes to infinity. Otherwise
there is $v$ in $\mi$, and there are $a_2 > 0$, $t_i \rightarrow \infty, s_i \rightarrow  -\infty$
so that 

$$d(\wwp_{t_i}(v),\wwp_{s_i}(v)) \ < \ a_2.$$

\noindent
Up to subsequence assume that the sequences $(\pi(\wwp_{t_i}(v)))$, $(\pi(\wwp_{s_i}(v)))$ 
converge in $M$.
Hence  there are $g_i$ in $\pi_1(M)$ with

$$g_i(\wwp_{t_i}(v)) \ \rightarrow \ p_0, \ \ \ \ 
g_i(\wwp_{s_i}(v)) \ \rightarrow \ p_1.$$

\noindent
If $p_0$ and $p_1$ are in the same orbit of $\wwp$ then since $\oo \cong \rrrr^2$ there is a product 
neighborhood of the orbit segment from $p_0$ to $p_1$ and all segment lenghts are bounded, contradicting
that $t_i \rightarrow \infty$, $s_i \rightarrow -\infty$. Hence $p_0, p_1$ are not in 
the same orbit of $\wwp$, contradicting that $\oo \cong \rrrr^2$ is Hausdorff.

\vskip .08in
We now use the proof of the previous proposition and its setup. The 
corners of the lozenge $\alpha$ in $L$ and $\beta$ in $U$ were obtained
so that $\alpha$ is contained in $\wlu(\widetilde \delta)$, $\beta$ is in 
$\wls(\widetilde \epsilon)$. In addition $\widetilde \delta,
\widetilde \epsilon$ are periodic and their projections $\delta, \epsilon$
to $M$ are freely homotopic. Since $\widetilde \delta, \widetilde \epsilon$
are the corners of a lozenge then $\delta$ is freely homotopic
to the inverse of $\epsilon$. In particular a 
backward ray of $\widetilde \delta$ is less than $a_0$
from a forward ray of $\widetilde \epsilon$. 
But a backward ray of $\alpha$ is asymptotic to a backward ray
of $\widetilde \delta$ $-$ they are in the same unstable leaf.
In the same way a forward ray of $\beta$ is asymptotic to 
a forward ray of $\widetilde \epsilon$. 
The conclusion is that a backward ray of $\alpha$ is less than
say $a_0 + 1$ from a forward ray of $\beta$.
Notice that this is not yet the conclusion that we want.
We want information about forward rays in $\alpha$.

But we know that a forward ray of $\alpha$ is a bounded
distance from either a forward ray of $\beta$ or a backward
ray of $\beta$. Suppose that the forward ray of $\alpha$ is a bounded distance from a forward
ray of $\beta$. We have just proved that a forward ray of $\beta$ is 
a bounded distance from a backward
ray of $\alpha$. Then we would conclude that a forward ray of
$\alpha$ is a bounded distance from a backward ray of $\alpha$.
This is what was disallowed in the first part of the proof.

Therefore we conclude that a forward ray of $\alpha$ is
less than $a_0$ from a backward ray of $\beta$. This finishes the proof
of the theorem.
\end{proof}

\noindent
{\bf {Remark}} $-$ Theorem \ref{perfectbound} is one strong interaction of pseudo-Anosov flows
and geometry in $\mi$. Clearly if $L, U$ make a perfect fit, then points in $L$ 
cannot be too close to points in $U$, because of the local product picture of 
hyperbolic dynamics. However, a priori, appropriate rays in $L, U$ could be as far
away from each other in $\mi$. Theorem \ref{perfectbound} shows this is not the case.
This should be contrasted with flows without perfect fits. For example if 
$\Phi$ is a suspension pseudo-Anosov flow, then no rays in $\mi$ are boundedly 
away from each unless they are either in the same stable leaf or the same unstable
leaf.

\vskip .1in
We now prove theorem C:

\begin{theorem}{}{}
Suppose that no closed orbit of $\Phi$ is non trivially freely homotopic to itself.
Suppose that $L_0, L_1,...,L_k$ is chain of leaves, alternatively in $\wls$ and $\wlu$ satisfying:
$L_i$ makes a perfect with both $L_{i+1}$ and $L_{i-1}$ ($0 < i < k$) and for each $i$
either 1) $L_i$ separates $L_{i-1}$ from $L_{i+1}$ or 2) The half leaves of $L_i$  which make a perfect fit
with $L_{i-1}$ and $L_{i+1}$ respectively are distinct half leaves of $L_i$. This structure generates a
(non unique) free homotopy class of closed orbits of $\Phi$ of cardinality at least $k$.
\label{homot}
\end{theorem}

\begin{proof}{}
Conditions 1) and 2) are  used to guarantee that the leaves $\{ L_i, 0 \leq i \leq k\}$ are distinct from
each other $-$ this is just another way of ensuring  that condition.
First of all we can assume that no leaf $L_i$ is periodic for otherwise the result 
follows easily as perfect fits are preserved by appropriate powers of
covering translations which preserve one of the leaves of the perfect fit.

Let $a_0$ be the global constant produced by the previous theorem (denoted
by $a_1$ in that theorem).

Without loss of generality assume that $L_0$ is a stable leaf. For each $i$ let $\alpha_i$ be an
arbitrary orbit in $L_i$. By the previous proposition there is a forward ray
of $\alpha_0$ which is $< a_0$ Gromov-Hausdorff distance 
from a backwardray of $\alpha_1$. Also  a backward ray
of $\alpha_1$ is $< a_0$ from a forward ray of $\alpha_2$ and so on.
By taking subrays  we may assume that the same ray  in $\alpha_i$ works
for both conditions. 
Choose initial points $p^i_0$ in $\alpha_i$.

Choose sequences $(p^i_n), \ 0 \leq i \leq k, \ n \in {\bf N}$, inductively
with $i$ as follows. We may have to take subsequences at will. First choose

$$p^0_n \ \in \ \alpha_0 \ \ {\rm and} \ \ p^0_n \ = \ \wwp_{t^0_n}(p^i_0), \ \  {\rm with} \ \ 
\lim_{n \rightarrow \infty} t^0_n \ = \ \infty.$$ 

\noindent In addition assume that the  sequence $(t^0_n)$ is monotone with $n$. 
Up to subsequence assume that
$(\pi(p^0_n))$ converges to $q_0$ in $M$ and all of the elements are in the same sector of $q_0$.
Recall that $L_0$ is not periodic.

Now choose $p^1_n$ in $\alpha_1$ with $d(p^1_n,p^0_n) < a_0$. Let $p^1_n = \wwp_{t^1_n}(p^1_0)$. By 
theorem \ref{perfectbound}
 we know that $(t^1_n)$  converges to minus infinity. We assume that $(t^1_n)$ is monotone.
Up to a subsequence assume that $\pi(p^1_n) \rightarrow q_1$ and all $\pi(p^1_n)$ are in the same
sector of $q_1$. 

Continuuing by induction on $i \leq k$, we choose for each $i$ a sequence $(p^i_n)$ satisfying:

\begin{itemize}
\item $d(p^i_n, p^{i-1}_n) < a_0$ for all $n$,

\item  $p^i_n \ = \ \wwp_{t^i_n}(p^i_n)$,

\item  $t^i_n \rightarrow -\infty$ if $i$ is odd and $t^i_n \rightarrow \infty$ if $i$ is even.
Each sequence $(t^i_n)$ is monotone in $n$,

\item  Up to subsequence (in $n$) we may assume that for each $i$, $\pi(p^i_n) \rightarrow q_i$
in $M$ and $\pi(p^i_n)$ are all in the same sector of $q_i$ for each $i$.
Also assume that all $\{ \pi(p^i_n) \}$ are sufficiently close to $q_i$ to be able to
apply the Closing lemma.
\end{itemize}

Notice that $\pi(p^i_n) = q_i$ for at most one $n$ for each $i$. Otherwise 
$\pi(p^i_n)$ is in a periodic orbit of $\Phi$, contrary to assumption that $L_i$ is
not periodic.

\begin{figure}
\centeredepsfbox{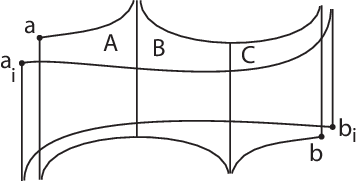}
\caption{A chain of lozenges of length 3 from $a$ to $b$. These orbits are near orbits
$a_i, b_i$ whose stable/unstable leaves form a perfect fit. In this way a sequence of
perfect fits, or a sequence of lozenges with corners $a_i, b_i$ converges to the
union of $3$ lozenges $A, B, C$.}
\label{qg6}
\end{figure}

Now consider minimal geodesic segments $\beta^i_n$, $1 \leq i \leq k$, $n \in {\bf N}$ from
$p^{i-1}_n$ to $p^i_n$. These segments are oriented from $p^{i-1}_n$ to $p^i_n$.
 They all have length smaller than $a_0$ in $\mi$. So now fix $n, m$ sufficiently
big, with $m >> n$ and so that for each $1 \leq i \leq k$, then
$\pi(\beta^i_n), \pi(\beta^i_m)$ are geodesic segments which 
are very close to each other in $M$.
Let $\gamma_i$ be the segment in $\alpha_i$ from $p^i_n$ to $p^i_m$. It projects to an almost closed
orbit segment in $M$. By the closing lemma the segment
 is shadowed by a closed orbit $\tau_i$ of $\Phi$ in $M$ for each
$0 \leq i \leq k$. As done in great detail in \cite{Fe3}  consider the closed curve

$$ \gamma_{i-1} \circ (\beta^i_m)^{-1} \circ ( \gamma_i)^{-1} \circ \beta^i_n$$

\noindent
in $\mi$, where the  inverses mean the segments or flow segments are traversed against their orientations.
This projects to a closed curve which is null homotopic in $M$. The images of $\gamma_{i-1}$ and
$(\gamma_i)^{-1}$ are almost closed and shadowed by the closed orbits $\tau_{i-1}$ and $\tau_i$.
The images $\pi(\beta^i_n), \pi(\beta^i_m)$ are very close geodesic segments which can be closely
connected to each other. This produces a free homotopy from $\tau_{i-1}$ to $\tau_i$ in $M$.
This free homotopy lifts to a free homotopy between coherent lifts
$\widetilde \tau_{i-1}$ and $\widetilde \tau_i$. These are corners
of  a finite chain of lozenges so that the initial corner $\widetilde \tau_{i-1}$
 is near $\alpha_{i-1}$ and the final corner $\widetilde \tau_i$ is near $\alpha_i$. The corresponding
stable or unstable leaves of $\alpha_{i-1}$ and $\alpha_i$ form a perfect fit. 
%This implies that the
%chain of lozenges above has an {\underline {odd}} number of lozenges. 
The chain may have more than one lozenge 
because in the limit leaves can split into a collection of non separated leaves in the chain of lozenges
making up the free homotopy, see figure \ref{qg6}. 
In addition no $\widetilde \tau_i$ is equal to $\widetilde \tau_j$ if $i \not = j$.
Otherwise, since $\widetilde \tau_i$ has points very close to $\alpha_i$,
it would follow that $\alpha_i, \alpha_j$ have points very close to each other which
is impossible.

Since no closed orbit of $\Phi$ is non trivially freely homotopic to itself, 
the $\{ \tau_i \}$ are all distinct closed orbits.
Therefore the free homotopy class of $\tau_0$ has cardinality at least $k$.

This finishes the proof of  theorem \ref{homot}.
\end{proof}

In fact in the unbounded case we can prove that there are chains of perfect
fits of {\underline infinite} length:

\begin{theorem}{}{} Suppose that $\Phi$ is an unbounded pseudo-Anosov flow
and in addition that $\Phi$ is not topologically conjugate to a suspension
Anosov flow. Then $\Phi$ has chains of perfect fits of infinite length.
\label{infinite}
\end{theorem}

\begin{proof}{}
If there is a non trivial free homotopy from a periodic orbit to itself the result is obvious.
Otherwise
Theorem \ref{homot} shows that for each natural $n$  there are chains 
of free homotopies of length $2n+1$. 
These lift to chains of lozenges ${\mathcal L_n}$ of length $2n+1$.
Let the ordered corners of $\mathcal L_n$ be 
denoted by $\{ \gamma^n_i, \ -n \leq i \leq n \}$. Up to covering translations
we may assume that the sequence $(\gamma^n_0)$ converges to an orbit, which
will be denoted by
$\beta_0$.
Let $C_n$ be the lozenge with corners $\gamma^n_0$ and $\gamma^n_1$.
Theorem \ref{thickness} shows that the Gromov-Hausdorff distance 
$d_H(\gamma^n_0,\gamma^n_1)$ is bounded. Therefore up to another 
subsequence, the sequence $(\gamma^n_1)$ also converges to an orbit,
which will be denoted by $\beta_1$. It follows that $\wls(\beta_0)$ and
$\wlu(\beta_1)$ are connected by a chain of perfect fits.

We can use induction in  $i$ and then use a diagonal process of subsequences, 
to show that 
there is a subsequence of $(\mathcal L_n)$ still denoted by 
$(\mathcal L_n)$ so that for each $i \in {\bf N}$ the following limit exists

$$ \lim_{n \rightarrow \infty} \gamma^n_i \ := \ \beta_i$$

\noindent
in $\oo$. Then $\wls(\beta_i), \wlu(\beta_{i+1})$ are connected by a chain of perfect
fits. This proves theorem \ref{infinite}.
\end{proof}

\noindent
{\bf {Remark}} $-$ One natural question is why Theorem \ref{infinite} is not stated
for freely homotopic orbits. That is, why can't one prove there are chains of free
homotopies of infinite length? For example one could start with the infinite chain
of perfect fits $\{ L_i, i \in {\bf N} \}$ given by Theorem \ref{infinite} and use the perturbation
methods of Theorem \ref{homot} to try to produce an infinite free homotopy
class. This is subtle. Start with an infinite chain of perfect fits.
Recall the method of Theorem \ref{homot}: for each $i$ we pick a sequence (in $n$) of points $(p^i_n)$
in $L_i$ and then take subsequences of these so we apply the Closing lemma
to produce closed orbits. The problem is that to apply the Closing lemma, the points
in question have to be very close. In particular one may have to go forward
or backward a lot in that orbit. So when that gets mapped back to the initial
leaf $L_1$ one cannot guarantee that the starting point is in a bounded region.
This  is a finite process. One then takes limits. Even if the orbits are periodic,
when one takes limits, they may not be periodic in the limit, so we cannot
guarantee free homotopies of infinite legth.

\begin{theorem}{}{} Let $\Phi$ be an arbitrary pseudo-Anosov flow in $M^3$ atoroidal. 
Suppose there is an infinite chain of perfect fits. Then there are chains of lozenges
$\mathcal C$  of 
any given finite length satisfying the following:  $\mathcal C$ has periodic corners, $\mathcal C$ does 
not have any singular corner and $\mathcal C$ does not have
any adjacent lozenges.
In the same way if there is an infinite chain of lozenges with periodic corners,
then there is an infinite chain 
of lozenges with periodic corners, and no singular corners and no adjacent lozenges.
\end{theorem}

\begin{proof}{}
An infinite chain of perfect fits is $\{ L_i, i \in {\bf N} \}$ so that $L_i$ are leaves in 
$\wls$ or $\wlu$ (or $\oos, \oou$) and $L_i$ makes a perfect fit with $L_{i+1}$ for
every $i$. The indexing set could be ${\bf Z}$ also.
The previous theorem shows that $\Phi$ admits finite chains of periodic 
lozenges of any finite length. 
There are finitely many singular orbits of $\Phi$. Up to covering translations there 
are only finitely many leaves of $\oos$ or $\oou$ which are non separated from
another leaf in the same foliation \cite{Fe4,Fe5}. In addition any of these leaves
is periodic, by Theorem \ref{theb}. Therefore there is $n_0$ integer so that up to 
covering translations there are finitely many $\leq n_0$ orbits $\gamma$ of $\wwp$ so 
that $\gamma$ is periodic and 
either $\gamma$ singular, or one of $\oos(\gamma)$ or $\oou(\gamma)$ is
non separated from another leaf in its respective foliation.

We prove the first assertion of the theorem. Suppose that there is $n_1 \in {\bf N}$ so that there are
no chains of periodic lozenges of length $n_1$ with no corners which are singular
orbits and no adjacent lozenges. Let $\mathcal C$ be a chain of periodic lozenges
of length bigger than $(n_0 +1)(n_1 + 1)$. 
Since the corners of $\mathcal C$ are periodic there is $g$ in $\pi_1(M)$ non 
trivial so that $g$ leaves invariant all corners of $\mathcal C$.
Start at one end of $\mathcal C$.
By hypothesis after at most $n_1$ steps the chain hits a corner $\gamma$ so that
either 1) $\gamma$ is a singular orbit, \ 2) $\oos(\gamma)$ is non separated from 
another leaf in $\oos$ or \ 3) $\oou(\gamma)$ is non separated from another
unstable leaf. Since the length of $\mathcal C$ is $> (n_0 +1)(n_1 + 1)$ there
are at least $n_0 + 1$ instances of 1), 2) or 3) above. 
By choice of $n_0$ it follows that there are corners   $\alpha$, $\beta$ of $\mathcal C$ 
which project to the same orbit of $\Phi$. So there is $f$ in $\pi_1(M)$
with $f(\alpha) = \beta$. Then $f^{-1} g f(\alpha) = \alpha$. This implies 
that $f^{-1} g f = g^i$ for some non zero $i$ in ${\bf Z}$. In addition 
$f g f^{-1}(\beta) = \beta$ so also $f g f^{-1} = g^j$ for some non zero $j$ in ${\bf Z}$.
Since $\pi_1(M)$ does not have torsion it follows that $f^{-1} g f = g^{\pm 1}$.
It follows that  $f^2, g$ generate
a ${\bf Z}^2$ subgroup of $\pi_1(M)$. This contradicts that $M$ is atoroidal.
This proves the first assertion of the theorem.

Suppose now that $\mathcal C$ is an infinite chain of lozenges with periodic corners.
 If there are infinitely many 
corners which are either singular or in a leaf which is non separated from another
leaf, then the arguments in the proof of the first assertion imply that $M$
is toroidal, contradiction to hypothesis. We conclude that there are only
finitely many corners which are either singular or in a leaf non separated from another leaf.
We conclude that there is an infinite subchain $\mathcal C'$ which has 
the desired property.

This finishes the proof of the theorem.
\end{proof}

\noindent
{\bf {Remark}} $-$ As  in the case of Theorem \ref{infinite} there is an issue with 
the upgrading from perfect fits to lozenges with  periodic corners.
In the second assertion in the theorem, suppose one starts with an infinite chain $\mathcal C$ of
lozenges, not a priori with periodic corners.
Then one can approximate any {\underline {finite}}
 subchain  of $\mathcal C$ by one with periodic corners. But as explained before, the perturbation 
methods do not produce an {\underline {infinite}} chain of lozenges with 
periodic corners.

\section{Convergence group action}

In the next few sections we prove Theorem F and Theorem D, which imply the Main theorem.
In this section we prove that if $\Phi$ is bounded, then
$\pi_1(M)$ acts as a convergence group
on a candidate for the flow ideal boundary of $\mi$. The bounded
hypothesis will be fundamental for many steps and the result
does not work without this hypothesis.

\vskip .1in
\noindent {\bf Decomposition of $\partial (\cd \times I))$ $-$ equivalence relation $\simeq$ in
$\partial (\cd \times I)$}

Let $\Phi$ be a  pseudo-Anosov flow.
In 

$$\partial (\cd \times I) \ \ = \ \ \oo \times \{ 1 \} \ \cup \
\oo \times \{ 0 \} \ \cup \
\partial \oo \times [-1,1]$$

\noindent
we consider the following decomposition which is generated
by:

\begin{itemize}

\item \ 1) For any  $(p,1) \in \oo \times \{ 1 \}$ consider 
the element $I^s_p = (\oos(p) \cup \partial \oos(p)) \times \{ 1 \}$

\item \ 2) For any $(p,-1) \in \oo \times \{ 1 \}$ consider 
the element $I^u_p = (\oou(p) \cup \partial \oou(p)) \times \{ -1 \}$

\item \ 3) For any  $(p,t) \in \partial \oo \times I$ consider 
the element $I^{\partial}_p = \{ p \}   \times I$.
\end{itemize}

\noindent
We let $\simeq$ be the equivalence relation 
in $\partial (\cd \times I)$ generated by these decomposition
elements.

\vskip .07in
Here $p$ is an arbitrary point in $\cd = \oo \cup \partial \oo$.
Notice that if $p \in \partial \oos(x)$ then
$I^{\partial}_p$ and $I^s_x$ intersect in $(p,1)$ and
similarly if $p \in \oou(x)$, then $I^{\partial}_p$ and $I^u_x$ intersect
in $(p,-1)$.

\begin{define}{(equivalence relation $\sim$ in $\partial \oo$)}{}
The equivalence relation $\simeq$ in $\partial (\cd \times I)$ induces an
equivalence relation in $\partial \oo$, denoted by $\sim$.
Explicitly: if $x, y$ are in $\partial \oo$, then $x \sim y$ if and only
if there are leaves $l, u$ each of which can be either stable or
unstable and so that: \ a) $x \in \partial l, \ y \in \partial u$, and \ b) $l, u$ are 
connected by a chain of perfect fits. This includes the case that $x, y$
are ideal points of the same leaf.
\end{define}

In particular $x \sim y$ if and only if $(x,1) \simeq (y,1)$.

\vskip .1in
\noindent
{\bf {Notation}} $-$ If $Z$ is a subset of $\partial \oo$, all of whose elements
are related under $\sim$, then we let

$$\ee(Z) \ = \ \ \ {\rm the \ union \ of \ the \ equivalence \
classes \ of } \ \ \ \sim \ \ \ {\rm intersecting } \ \ \  Z.$$

\noindent 
Examples of this are $\ee(\partial \oos(p)), \ee(\partial \oou(p))$ where $p$ is in $\oo$
and for example $\partial \oos(p)$ is the set of ideal points of prongs of
$\oos(p)$.
If $z$ is a point in $\partial \oo$, we
also denote by $\ee(z)$ the equivalence class $\ee(\{ z \})$.
\vskip .1in

\begin{define}{(flow ideal boundary)}{}
Let $\cal R$ be the quotient space of
$\partial (\cd \times I)$ by the equivalence
relation $\simeq$.
Every point in 
$\oo \times \{ -1, 1 \} \cup \partial \oo \times I$ is
related to a point in
$\partial \oo \times \{ 1\}$. So we may think of $\cal R$
as a quotient space of $\partial \oo$ by the equivalence
relation $\sim$. Here we are naturally identifying $\partial \oo$
with $\partial \oo \times \{ 1 \}$. The topology in $\rr$ is the same a the
the quotient topology from $\partial \oo$.
\end{define}

In other words the equivalence relation induced by $\simeq$ in $\partial \oo$ 
is exactly the relation $\sim$.
Since we obtain $\rr$ as a quotient of either $\partial (\cd \times I)$ or $\partial \oo$,
the last statement means that the quotient topology is the same for both quotients.

The boundary $\partial (\cd \times I)$ of $\cd \times I$ is homeomorphic to the
two sphere ${\bf S}^2$.

\vskip .15in
\noindent
{\bf (Counter) Example} $-$ Consider the 
case of  a skewed 
$\rrrr$-covered Anosov flow in an atoroidal manifold \cite{Fe2}. Then the corresponding orbit space
has boundary  $\partial \oo$ made up of
 2 special points and 2 lines
$l_1, l_2$. Any point in $l_2$ is identified to a point
in $l_1$. In addition
there is a translation in $l_1$ induced by a composition of 
perfect fit maps \cite{Fe2}. This is the slithering map as defined by Thurston in
this situation \cite{Th5}.
In this case the quotient  $\cal R$  of $\partial \oo$ as
in the definition above is as follows: \ 
${\cal R} = {\bf S}^1 \cup \{ a_1, a_2 \}$. The circle ${\bf S}^1$ is the
quotient of $l_1$ (or $l_2$) by the slithering map. Any point
in ${\bf S}^1$ is not separated from both
$a_1$ and $a_2$ and so $a_1, a_2$ are not separated from each other.
The quotient space ${\cal R}$ in this case satisfies only the $T_0$ topological separation property.
This pseudo-Anosov flow (in fact Anosov) is not bounded.
The action of $\pi_1(M)$ on ${\cal R}$ in this case is definitely not a convergence group.

One important property we need is that there are no 
 identifications between points of $\partial \oos(x)$ and
$\partial \oou(x)$
under $\sim$.

\begin{proposition}{}{}
If $x$ is a point in $\oo$ then no point of  $\partial \oos(x)$ 
is equivalent to any point of $\partial \oou(x)$ under $\sim$.
In addition if $r$ is a ray of $\oos(x)$ there is no chain of perfect
fits from $r$ to another ray of $\oos(x)$ except for
$\oos(x)$ itself.
\label{noident}
\end{proposition}

\begin{proof}{}
Observe that we consider $\oos(x)$ a trivial chain of perfect fits between
ideal points of $\oos(x)$. In addition we are only considering minimal
chains $-$ no backtracking allowed.

Suppose there is a chain of perfect fits from a ray $r$ of 
$\oos(x)$ to a ray $r'$ of $\oou(x)$. This is a chain of rays
$r = r_0, r_1,..., r_m = r'$ in leaves of $\oos$ or $\oou$ so that either 

\begin{itemize}

\item $r_i, r_{i+1}$ are rays in the same leaf of $\oos$ or $\oou$
$-$ in which case we assume that $r_i \cup r_{i+1}$ forms a slice
of this leaf, or

\item $r_i, r_{i+1}$ have the same ideal point in $\partial \oo$.
In this case there are $r_i = \tau_0, \tau_1, ..., \tau_k = r_{i+1}$
so that $\tau_j$ are rays alternatively in leaves of $\oos$ and $\oou$, and $\tau_j$
makes a perfect fit with $\tau_{j+1}$. 
By truncating some slices and rays if necessary, we may assume that 
each  point $y$ of $\partial \oo$ occurs as an ideal point of at
most two consecutive rays. Maybe $y$ is also the ideal point
of some other ray in the chain, but not
consecutive with the first two.
\end{itemize}

\begin{figure}
\centeredepsfbox{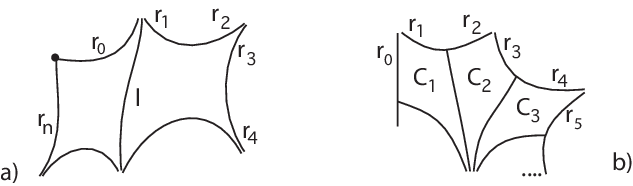}
\caption{a. Closing up of leaves making perfect fits, b.
Long chains of leaves making perfect fits.}
\label{qg7}
\end{figure}

There are several possibilities each of which leads to some contradiction.
The conditions imply that by eliminating rays with the same ideal
points in $\partial \oo$ if necessary, then the chain may be reformatted
to have the following format:
$r_0, r_1 \cup r_2, r_3 \cup r_4, ..., r_n$ where 
$r_{2i}$ and $r_{2i+1}$ have the same ideal point in $\partial \oo$,
and $r_{2i-1}  \cup r_{2i}$ forms a slice in a leaf of $\oos$ or $\oou$.
In particular no 3 consecutive rays have the same ideal point in $\oo$.

Suppose that we have a chain as above with minimum number of rays
amongst all such chains and all points $x$ in $\oo$. 
With the notation above the number of rays
is $n + 1$.
In particular there are no transverse self intersections between all the
rays and slice leaves, except for the first and the last rays. 
Otherwise we could cut the parts before and after
the intersection $z$ to produce a chain with a smaller number of rays.
We can also assume that $r_0$ and $r_n$ start at $x$.
It now follows that 

$$ c \ \ = \ \ r_0 \cup r_1 \cup ... \cup r_n$$

\noindent
bounds an open region $R$ in $\oo$. This is because there are no self
intersections in $\oo$ and any ideal point
of $\oo$ is only traversed once. We cannot skip any leaves in between.

\vskip .06in
There are two possibilities: either for all $i$ the rays  $r_{2i}$ and $r_{2i+1}$ 
make a perfect  fit or for some $i$ this is not true. 

We first analyse the second possibility and suppose without loss of 
generality that $i = 0$. 
 Then there is a $\tau_1$ making a perfect fit with $r_0$ and 
$\tau_1$ contained in a leaf $l$ of say $\oou$ with $l$ separating $r_0$ 
from $r_1 \cup r_2$.

There are several possibilities each of which leads to some
contradiction.
 It could be that  the leaf $l$ 
contains another ray in the chain $\{ r_j \}$. Then we can produce a chain from a ray
of $l$ to another ray of $l$ with less rays than the original
chain, contradicting the minimality of the chain $\{ r_j, 0 \leq j \leq n \}$.
Another option is that $l$ intersects a leaf in
$\partial R$ transversely. 
It cannot be that consecutive rays are 
contained in leaves that intersect transversely. This is because if that
were the case then either the rays make a perfect fit and a nearby pair
of leaves one in $\oos$ and one in $\oou$ intersect twice, or
the rays are separated by more than one perfect fit which also leads
to a contradiction. By similar arguments no two leaves in the union
of $\oos \cup \oou$ can share more than one ideal point in $\partial \oo$.
Hence if $l$ intersects a leaf in $\partial R$ transversely,
then again cut out a  chain of smaller length,
contradiction. Finally it could be that two  rays of $l$
limit in ideal points  of some of the $\{ r_j \}$, see fig. \ref{qg7}, a.
Then again  we can cut the region $R$ along $l$ to decrease the
number of leaves/rays in the chain 
$\{ r_j, 0 \leq j \leq n \}$. We conclude that this situation cannot happen.

We conclude that the only remaining possibility is that for all
$i$,  $r_{2i}$ and $r_{2i+1}$ make
a perfect fit and $r_{2i+1} \cup r_{2i+2}$ does not separate
consecutive perfect fits.
Suppose that $n \geq 3$. Then $r_1 \cup r_2$ is a slice which
makes a double perfect fit with $r_0$ and $r_3$.
By proposition \ref{doublefit}, $r_1 \cup r_2$ is periodic
and in fact, all rays  $r_j, 0 \leq j \leq n$
are periodic and left invariant by the same non trivial element
$g$ of $\pi_1(M)$. 
Then $r_1, r_2$ are in the boundary of adjacent lozenges 
$C_1, C_2$. Therefore $r_3$ cannot intersect $r_0$, see
fig. \ref{qg7}, b.
If $n \geq 4$ then $r_3 \cup r_4$ is in the boundary of adjacent
lozenges $C_2 \cup C_3$
see fig. \ref{qg7} b.
Then it is impossible for the chain $r_0, ..., r_n$ to
close up.

Notice that with the reformatting above $n$ has to be odd.
Hence the only  remaining possibility is that   $n =1 $.
If $n = 1$ this means that a ray of $\oos(x)$ makes a perfect
fit with a ray of $\oou(x)$. This is impossible, because
nearby leaves of $\oos, \oou$ would intersect twice,
a contradiction.

Exactly the same type of arguments show the second statement of the
proposition.

%If $n = 3$ this produces a $(3,1)$ ideal
%quadrilateral which is disallowed by proposition \ref{notreeone}.
%If $n = 4$ this produces a $(4,0)$ ideal quadrilateral, also
%disallowed by proposition \ref{nothreeone}. 
%\vskip .1in
%In the other
%cases $r_1, r_2$ are in the boundary of adjacent lozenges 
%$C_1, C_2$ and $r_4$ is in the boundary of another lozenge $C_3$
%adjacent to $C_2$ and so on. See fig. \ref{perfectchain} b.
%Then it is impossible for the sequence $r_0, ..., r_n$ to
%close up.

This finishes the proof of proposition \ref{noident}.
\end{proof} 

\begin{figure}
\centeredepsfbox{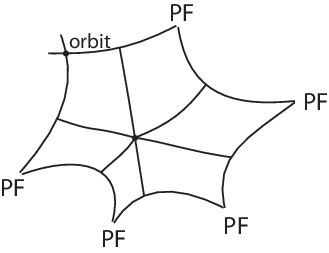}
\caption{a. A sequence of perfect fits closing up around a 
singular periodic orbit and with one actual intersection.}
\label{qg8}
\end{figure}

\noindent
{\bf {Remark}} $-$ This shows that one cannot close up
a chain of perfect fits by going around a singular
orbit as in figure \ref{qg8}.
In this figure each PF is a perfect fit. If this were possible,
then $\partial \oou(p)$ would be identified with $\partial \oos(p)$ under $\sim$. This is
disallowed by proposition \ref{noident}.

\vskip .08in
Before we prove the convergence group theorem we will establish several preliminary
results which will simplify the proof of the theorem. We first prove that the flow
ideal boundary $\mathcal R$ is homeomorphic to the two dimensional sphere.
We also analyse the action on $\partial \oo$ of an element $g$ of $\pi_1(M)$ with a fixed
point in $\oo$. 
We will also establish a rigidity property of the foliations $\wls, \wlu$.
These  results  will establish some easy cases of the convergence theorem.
There are additional useful results.

 We recall Moore's theorem on cellular decompositions. A decomposition $Q$ of a space $X$
is a collection of disjoint nonempty closed sets whose union is $X$. In other words this is
the same as the equivalence classes of an equivalence relation so that the equivalence 
classes are closed subsets of $X$. Consider the quotient space $X/Q$ and
the quotient map $\nu: X \rightarrow X/Q$.
The decomposition $Q$ satisfies the {\em upper semicontinuity property}
provided that, given $q$ in $Q$ and $V$ open in $X$ containing $q$, 
then the union of those $q'$ of $Q$ that are contained in $V$ is an open 
subset of $X$. This is equivalent to the map $\nu$ being a closed map.

A decomposition $Q$ of a closed $2$-manifold $B$ is 
{\em {cellular}} if the following
properties hold: \ 1) $Q$ is
upper semicontinuous, \ 2) Each $q$ in $Q$ is a compact subset of $B$,
\ 3) Each $q$ of $Q$ has a non separating embedding in the plane $\rrrr^2$.
The following result was proved by R. L. Moore for the case of a sphere:

\begin{theorem}{(Moore's theorem)}{(approximating cellular maps)\cite{Mu}}
Let $Q$ denote a cellular decomposition of a $2$-manifold $X$ homeomorphic
to a sphere. Then the quotient map $\nu:  X \rightarrow X/Q$ can
be approximated by homeomorphisms. In particular $X$ and $X/Q$ are
homeomorphic.
\end{theorem}

\begin{theorem}{}{}
If $\Phi$ is a bounded pseudo-Anosov flows then 
the map $\partial \oo \rightarrow \mathcal R$ is
finite to one (bounded) and $\mathcal R $
is homeomorphic to the 2-dimensional sphere.
\label{cellular}
\end{theorem}

\begin{proof}{}
%The  equivalence classes of $\sim$ are called
%decomposition elements. 
For most of the proof we will consider $\rr$ as the quotient
of $\partial (\cd \times I)$ by $\simeq$.

%The set of equivalence classes of $\sim$ is denoted by $\mathcal E$.
Let $e$ be an element of $\mathcal R$. We consider an equivalence class
of $\simeq$ as both an element of $\mathcal R$ and as a subset of
${\bf S}^2 = \partial (\cd \times I)$.
We first show that $e$ is a closed subset of 
${\bf S}^2 = \partial (\cd \times I)$.
Let $pr: \cd \times I \rightarrow \cd$ be the projection to the first factor.
Given a subset $B$ of $\cd \times I$, the projection of $B$ to $\oo$ is the
set $\oo \cap pr(B)$. Similarly for the projection of $B$ to $\partial \oo$. 
The first step is to show the following: 

\begin{lemma}{}{}
Let $e$ be an arbitrary 
element $e$ of $\rr$ thought of as a subset
of $\partial (\cd \times I)$. Suppose that the projection of 
$e$ to $\oo$ is non empty and contains a leaf $l$ of 
$\oos$ or $\oou$. Then this projection to $\oo$ is exactly 
the set of leaves of $\oos \cup \oou$ which can be 
connected to $l$ by a chain of perfect fits.
In addition the projection of $e$ to 
$\partial \oo$ is the union of the ideal points of these
leaves.
\end{lemma}

\begin{proof}{}
Suppose that $l$ is (say) a stable leaf.
Suppose that $z \in e$. Then $z$ can be connected to $l \times \{ 1 \}$ 
by finitely many steps either vertical in $\partial \oo \times I$ $-$ type 3) in 
the decomposition of $\partial (\cd \times I)$ $-$ or horizontal
in $\oo \times \{1, -1 \}$$-$ type 1) and 2) in the decomposition of 
$\partial (\cd \times I)$.
The paths may jump from $\oo \times \{ 1 \}$ to $\oo \times \{ -1 \}$ along 
a vertical fiber which is associated to $p \in \partial \oo$ which is the ideal point of
both stable and unstable leaves $s, u$ respectively. This can only
occur if $s, u$ have ideal point $p$ and hence are connected
by a chain of perfect fits. This yields the result.
\end{proof}

The lemma shows that since $\Phi$ is bounded, then an equivalence class
$e$ of $\simeq$ is a finite union of compact sets in
$\cd \times \{ 1 \}, \cd \times \{ -1 \}$ and finitely many
vertical stalks in $\partial \oo \times I$. Therefore $e$ is a compact
subset of $\partial (\cd \times I)$.
In addition there are no loops in $e$  because in Proposition \ref{noident}
we proved that the only chain of perfect fits between  rays in a leaf of $\oos$ or
$\oou$ is the leaf itself.
Hence the the equivalence classes of $\simeq$ are simply connected
subsets of $\partial (\cd \times I)$.
In addition the lemma shows the first statement of theorem \ref{cellular}.

\vskip .05in
To be able to use Moore's theorem, what
is left to prove is to show the upper semicontinuous property
of the equivalence classes of $\simeq$. Suppose that $e$ is an equivalence
class of $\simeq$
and $B$ is an open subset of $\stwo = \partial (\cd \times I)$ containing $e$.
Let $B'$ be the union of the equivalence classes $e'$ of $\simeq$ entirely
contained in $B$. We need to show that $B'$ is open.
Given $x$ in $\partial (\cd \times I)$ we denote by $e(x)$ the 
equivalence class of $\simeq$ containing $x$.
It suffices to show the following: if $x_i \in 
\partial (\cd \times I)$ converges to $x$ in $e$ then $e(x_i)$ is 
eventually contained in $B$. Let $n_0$ be the upper bound on the cardinality
of chains of 
perfect fits $-$ that is $-$ the number of leaves which 
are connected to any given leaf by a chain of perfect fits.

\vskip .1in
\noindent
{\underline {Case 1}} $-$ 
Suppose $x$ is in $ \oo \times \{ 1 \}$.

Let $x = p \times \{ 1 \}$, $x_i = p_i \times \{ 1 \}$.
If the $\{ p_i \}$ are all in $\oos(p)$ the fact $e(x_i) \subset B$ is 
obvious.
Assume without loss of generality that the $x_i$ are all in
a fixed complementary component of
$\oos(p)$.
Let $\{ d_j, 1 \leq j \leq n_1 \}$ be the set of leaves of $\oos$ non separated from
$\oos(p)$ in the component of $\oo - \oos(p)$ containing $p_i$.
We call this the {\em side} containing $p_i$.
The escape lemma (Lemma \ref{trapbound}) shows that any point
of $\cd = \oo \cup \partial \oo$ which is in
a limit of a sequence in $\oos(p_i) \cup \partial \oos(p_i)$ is in 
$\cup_{j} \ d_j \cup \partial d_j$. Since

$$\left(( \cup_{j} \ d_j \cup \partial d_j \right) \times \{ 1 \}) \ \ \ 
\subset \ \ \ e$$

\noindent
it follows that for $i$ big enough,
then 
$(\oos(p_i) \cup \partial \oos(p_i)) \times \{ 1 \} \  \subset \ B$.
Let $q_i$ be the ideal points of $\oos(p_i)$ so that the sequence
$(q_i)$  converges to $q$ ideal point
of say $d_1$.
Since $d_1 \times \{ 1\} \subset e$, then $q \times I \subset e$
and also a neighborhood of it in $\partial (\cd \times I)$ is contained
in $B$.
 Suppose for instance that $q_i$ is also an ideal point of
leaves different from $\oos(p_i)$. Without loss of generality assume that
$q_i \in \partial u_i$, $u_i \in \oou$. If the sequence
$(u_i)$ escapes compact sets in
$\oo$, then the escape lemma shows that $u_i \rightarrow q$
in $\oo \cup \partial \oo$. Hence
for $i$ big $u_i \times \{ -1 \} \subset B$.

If on the other hand the sequence $(u_i)$ does not escape
compact sets  in $\oo$, assume up to subsequence
that $(u_i)$ converges to 
$\{ c_j, \ 1 \leq j \leq n_2 \}$ $-$ a finite collection of unstable leaves.
As $\partial u_i$ contains $q_i$ and $(q_i)$ converges to $q$, it follows again from
the escape lemma that one $c_j$, say $c_1$ has ideal point $q$. 
Then $(c_j \cup \partial c_j) \times \{ -1 \}$ is contained in $e$. 
The escape lemma applied to these sequences shows that 
$u_i \times \{ -1 \} \subset B$ for $i$ big.

We can iterate this process: suppose that $u_i$ has ideal point $t_i$ and 
$t_i$ is also an ideal point of $\delta_i$ leaf of $\oos$. We apply the same
argument as above now switching unstable and stable leaves.
The important point is that since $\Phi$ is bounded, then
any such chain of perfect fits is bounded in length, therefore eventually
all $e(x_i) \subset B$. 

If $x$ is in $\partial \oo \times \{ -1 \}$ this is treated similarly.
Finally:

\vskip .1in
\noindent
{\underline {Case 2}} $-$ Suppose that $x$ is in $\partial \oo \times I$.

Suppose first that $x_i \in \partial \oo \times I$.
Here $x \in \{ p \} \times I, \ p \in \partial \oo$.
If $e(x_i) = p_i \times I$, then $e(x_i) \subset B$
for $i$ big as $B$ is open. Otherwise $p_i$ is
an ideal point of $l_i$, say without loss of generality
that $l_i$ are in $\oos$. If the sequence $(l_i)$ escapes compact sets
in $\oo$, then $(l_i)$ converges to $p$ in $\cd$ $-$ 
otherwise we obtain a non trivial segment in $\partial \oo$ without
ideal points of leaves of $\oos$ or $\oou$. 
It follows that

$$(l_i \cup \partial l_i) \times \{ 1 \} \ \ \subset \ \ B
\ \ \ \ {\rm for} \ \ i \ \ {\rm big}.$$

\noindent
If the sequence $(l_i)$ does not escape compact sets in $\oo$ then
there are $z_i \in l_i \times \{ 1 \}$ with
$(z_i)$ converging to $z$ in $\oo$. The escape lemma then implies
that $p$ is an ideal point of $s \in \oos$ with $s$ non separated from
$\oos(z)$. This reduces the proof the arguments in Case 1.

Finally assume that (say)  $x_i \in \oo \times \{ 1 \}$.
If $(\oos(x_i))$ does not escape compact sets in $\oo$, use the
argument in the previous paragraph. Otherwise $(\oos(x_i))$
converges to $p$ in $\cd$ and we are done.
In the same way one deals with $x_i \in \oo \times \{ -1 \}$.

This finishes the proof of theorem \ref{cellular}.
This is because Moore's theorem implies that
${\mathcal R}$ is homeomorphic to the two sphere 
$\stwo$.
\end{proof}

\noindent {\bf {Remark}} $-$ We stress that the main tool used in
the above proof was the escape lemma bounded version
(Lemma \ref{trapbound}).
As the reader can attest,  it simplifies the proof tremendously.

\begin{define}{(attracting set for an action)}{} Suppose that $(g_n)$ is a sequence
acting on a compact metric space $X$. We say that $Y$ is the attracting set for the
sequence $(g_n)$ if $Y$ is closed and 
 the following happens: let $z$ be a point in $X$ which
is not fixed for any $g_n$ with sufficiently high $n$. Then the distance
from $g_n(z)$ to $Y$ converges to zero. In addition $Y$ is minimal
with respect to this property.
In the same way define the repelling fixed set of $(g_n)$. 
Finally if $g$ is a single transformation of $X$, we let the attracting fixed
set of $g$ to be that of the sequence $(g_n)$ where $g_n = g^n$.
\end{define}

It is crucial here that $Y$ need not be a single point. In general it may not
even be finite.

\begin{lemma}{(action of periodic transformations on $\oo$)}{}
Let $\Phi$ be a bounded pseudo-Anosov flow. Let $g$ in $\pi_1(M)$ so that
$g$ has a fixed point $p$ in $\oo$. Then $g$ has finitely many fixed 
points in $\partial \oo$ (possibly zero). There is an even number of
fixed points of $g$ acting on $\partial \oo$ and the fixed points
alternate between attracting and repelling fixed
points.
\label{perioconv}
\end{lemma}

\begin{proof}{}
Let $\gamma$ be an orbit of $\wwp$ with $\Theta(\gamma) = p$. Without
loss of generality assume that $g$ acts on $\gamma$ in the flow forward
direction. This means that if $z \in \gamma$ then $g(z) = \wwp_{t(z)}(z)$
with $t(z) > 0$. Let $A = \ee(\partial \oos(p)), \ B = \ee(\partial \oou(p))$.
Then $A, B$ are finite sets. We will show: 

I) points in $A, B$ alternate in 
$\partial \oo$, 

II)  both $A, B$ are invariant under the action of $g$ on
$\partial \oo$, and 

III)  the set $A$ is the attracting set for $g$ acting on
$\partial \oo$, and $B$ is the repelling set
for $g$ acting on $\partial \oo$.

\vskip .05in
The set $\oo$ is one dimensional like the real numbers. Instead of
speaking of convergence from the left or right as in $\rrrr$, we prove
convergence on each side of the point in question. The side can
be determined by the closure of sets in $\oo$.

Suppose first that $g$ leaves invariant all prongs at $p$.
Let $l$ be a prong of $\oos(p)$ at $p$ with ideal point $x$.
The goal is to show that $x$ is an attracting fixed point of $g$ acting on $\partial \oo$.
We will show that locally $g$ contracts any interval small interval
$I$ in $\partial \oo$ with one endpoint $x$.
In order to analyse that fix a line leaf $l'$ of $l$ with one ideal point $x$.
Then the components $V_1, V_2$ of $\oo - l'$ define the two complementary
intervals in $\partial \oo$ each of which has $x$ as an endpoint.
 Since $g$ is associated with the
forward direction of the flow then $g$ acts as a contraction on 
$\oou(p)$ with $p$ as a fixed point,
and as an expansion on $\oos(p)$.
We can depict this in $\oo \cup \partial \oo$ as follows: 
if $z$ is a fixed point of $g$ in $\oo$ we put an arrow away from $z$ in each prong
of $\oos(z)$ if $g$ acts as an expansion on $\oos(z)$, otherwise we put
an arrow towards $z$. Similarly we put arrows in the prongs of $\oou(z)$.
Hence
in each prong of 
$\oos(p)$ put an arrow which moves away from $p$
and put an arrow in each prong of $\oou(p)$ which points towards $p$.
We refer to fig. \ref{qg9}.

\begin{figure}
\centeredepsfbox{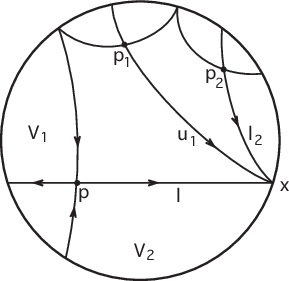}
\caption{The action of periodic transformations on the orbit space and its
boundary.}
\label{qg9}
\end{figure}

Suppose now that this prong $l$ makes a perfect fit with a leaf $u_1$ of $\oou$.
Because $g$ leaves invariant all prongs at $p$, then $g$ leaves invariant $u_1$
(in the general case $g^n$ leaves 
$u_1$ invariant for some $n > 0$).
Hence $u_1$ has a fixed point $p_1$ under $g$. Since $g$ acts as a contraction
on $\oou(p)$, it follows that $g$ acts as
an expansion on $u_1$.
This is the key point and we explain this. Here $p$ and $p_1$ are the corners
of a lozenge $C$ in $\oo$. The transformation $g$ acts as an expansion in $\oos(p)$, hence
as a contraction in $\oou(p)$. It follows that it acts as an expansion in $\oou(p_1)$.
This is because both $\oou(p), \oou(p_1)$ have a prong in the boundary of the
lozenge $C$. The structure of the lozenge means that contraction in the
prong of $\oou(p)$
is equivalent to expansion in the prong of  $\oou(p_1)$.
Therefore the prong $u'_1$ of $u_1$ with ideal point $x$ has an arrow pointing
towards $x$. In other words, just as in the prong  of $\oos(p)$, the arrow
in $\oou(p_1)$ 
points towards $x$.
Assume that $u_1$ is contained in $V_1$. 
Suppose now that $u_1$ makes a perfect fit with a stable leaf $l_2$
with ideal point $x$ and $l_2$ contained in $V_1$. In particular $l$ and $l_2$ are not separated
from each other.
Then $l_2$ has a fixed point $p_2$ under $g$. Using the same arguments as in the case
from $l$ to $u_1$, it now follows that $g$ acts as an expansion on 
 $l_2 = \oos(p_2)$.

%\begin{figure}
%\centeredepsfbox{qg2.eps}
%\caption{The last step in the process and the induced action
%boundary.}
%\label{peraction2}
%\end{figure}

By the bounded condition on the flow $\Phi$, there is a last leaf in 
this process. Call this last leaf $v_0$. 
Then $g$ leaves $v_0$ invariant.
Without loss of generality assume that $v_0$ is
an unstable leaf.
The arguments in the previous paragraph show that $g$ acts
as an expansion on $v_0$.
 Let $p'$ be the fixed point in $v_0$.
Let $v_1$ be the stable prong of $\oos(p')$ which together with 
the prong of $\oou(p')$ with ideal point $x$ bounds a sector  $Q$
 which is contained in $V_1$ and $Q$ not contained in the same
component of $\oo - v_0$ which contains $l$.
 Then $g$ acts as a contraction on $v_1$. 
Let now $u$ be an unstable leaf intersecting $v_1$. Then 
$(g^n(u))$ converges to $v_0$ as $n \rightarrow +\infty$. If there are other
leaves in the limit  of the sequence 
$(g^n(u))$ in the  sector  $Q$  then by theorem \ref{theb}  there is a stable leaf
$e'$ making a perfect fit with $v_0$ and in this sector $Q$. This 
contradicts the construction of $v_0$ as the last leaf with a prong with ideal prong $x$ in 
$V_1$.
%By the trapping lemma $(g^n(t)) \rightarrow x$ as $n \rightarrow \infty$.
This shows that $x$ is an attracting fixed point for $g$ in the side contained in
$\partial V_1$. The same proof applied to the other side of $x$ in $\partial \oo$
shows that 
$x$ is an attracting fixed point for $g$.

\vskip .05in
Let $y$ be the ideal point of the prong $v_1$. Then $v_1$ is invariant
under $g$ and $y$ is fixed by $g$. The same proof as above applied
to $v_1$ and $y$ shows that $y$ is a repelling fixed point of $g$
acting on $\partial \oo$.

We need to show that $g$ has no other fixed points in the interval
$(x,y)$ of $\partial \oo$ in the ideal boundary of $Q$.
 Let then $s_1$ be a regular stable leaf intersecting the prong of
$\oou(p') = v_0$ which has ideal point $x$.  Let $s$ be the prong of $s_1 - v_0$ contained
in the sector $Q$. Let $z$ be the ideal point of $s$. 
Consider the limit of $(g^n(s))$ as $n \rightarrow +\infty$. 
The intersection with $v_0$ escapes in $v_0$ as $n \rightarrow +\infty$.
If the sequence does not escape compact sets in $\oo$, then 
it has at least one limit leaf $b$ which makes a perfect fit with
$v_0$ and $b$ contained in $Q$. This contradicts the choice of
$v_0$ as the last leaf in $V_1$ with ideal point $x$.
Hence $(g^n(s))$ escapes compact sets as $n \rightarrow +\infty$ and
therefore the limit of $(g^n(z))$ as $n \rightarrow +\infty$ is $x$.

Now consider $(g^n(s))$ as $n \rightarrow -\infty$.
Clearly this limits on $v_1$. If this sequence has other limits, then they
are non separated from $v_1$ and contained in $Q$. In particular
they are connected to $v_1$ by a finite chain of adjacent 
{\underline {lozenges}} which are invariant under $g$ and all contained in $Q$.
If follows that the first lozenge has a corner in $p'$.
But this implies that this lozenge has another side (contained in $Q$ as well)
 which makes a perfect fit with $v_0$. This again contradicts the choice of
$v_0$. 
We conclude that there are no fixed points of $g$ in $(x,y)$.

We can restart the proof from $y$ and go around the circle to obtain
that $g$ has a finite, even number of fixed points in $\partial \oo$,
which are alternatively attracting and repelling. By the arguments $A$ is
the attracting set and $B$ is the repelling set for the action of $g$ on $\partial \oo$.
This finishes the proof of Lemma \ref{perioconv} in the case that $g$ leaves
all prongs of $\oos(p), \oou(p)$ invariant.

% If no such leaf exists then as we take the limit of 
%$s \cap v_0 \rightarrow x$ either
%\begin{itemize}
%\item it escapes in $\oo$. This produces a product region in $\oo$,
%contradiction.
%\item It limits to a collection of leaves $\{ s_j, j \in J \}$.
%This collection is invariant undre $g$. Since $(g^n(t))$
%converges to $x$, either 
%\item $g(s_{j_0}) = s_{j_0}$ for some $j_0$. Then by 
%theorem ??? the periodic orbit $p_2$ in $s_{j_0}$ and $p_1$ are connected
%by a chain of lozenges and there is a perfect fit with $v_0$ in that
%side. This is a contradiction. 
%\item Otherwise the collection is infinite and produces a scalloped region.
%This contradicts the property of $v_0$. Hence $(g^n(s))$ converges to 
%$x$ as $n \rightarrow \infty$. If $(g^{-n}(s))$ dos not converge to $y$ then 
%it limits to another collection $\{ w_j, j \in J\}$. As in the case above
%this produces a perfect fit with $v_1$. But since $v_1$ is invariant
%under $g$ then this produces another perfect fit with $v_0$, contradiction.
%\item We conclude that $(g^n(s))$ converges to $y$ when $n \rightarrow 
%-\infty$. Hence $g$ has no other fixed points in the interval $(x,y)$ of
%$\partial \oo$ and $y$ is an attracting fixed point for $g$ in this interval.
%Now go around the circle and obtain the desired property.
%\end{itemize}

\vskip .06in
Finally if $g$ does not leave each prong of $\oos(p)$ invariant, it still
leaves the collection of such prongs invariant. Therefore it leaves invariant
the collection of leaves making perfect fits with $\oos(p)$ and likewise
the collection of other leaves making perfect fits with these and so on.
Hence $g(A) = A$ and $g(B) = B$. On the other hand, for some fixed $i_0 > 0$,
$g^{i_0}$ leaves invariant each prong of $\oos(p), \oou(p)$ so the proof
above can be applied to powers to $g^{i_0}$. This implies that $A$ is
the attracting set of $g^{i_0}$ and $B$ is the repelling set for $g^{i_0}$.
Since $g$ itself leaves $A$ and $B$ invariant then $A$ is the attracting
set of $g$ and $B$ is the repelling set for $g$. Notice that it may well be
that $g$ itself does not have fixed points in $\partial \oo$.

This finishes the proof of lemma \ref{perioconv}.
\end{proof}

This has a quick consequence:

\begin{corollary}{}{}
Let $g$ be a non trivial covering translation. Then the action of $g$ on $\mathcal R$ has
a source and sink and they are different from each other.
\label{perioaction}
\end{corollary}

\begin{proof}{}
Suppose first that $g$ has a fixed point $p$ in $\oo$.
Without loss of generality assume  that $g$ is associated with the forward direction
in the flow line $\gamma$ projecting to $p$. The previous Lemma
shows that the equivalence class of $\partial \oos(p)$ is the 
attracting set for the action of $g$ on $\partial \oo$ and
the equivalence class of $\oou(p)$ is the repelling set.
Each class projects to a point in $\mathcal R$ and this proves the source
sink property for the action of $g$ on $\mathcal R$.
In addition Proposition \ref{noident} shows that the source and sink are
different from each other.

Suppose otherwise that $g$ does not fix any point in $\oo$.
Let $\hhs$ be the leaf space of $\wls$ or equivalently $\oos$.
Similarly define $\hhu$.
Then as shown in \cite{Fe4} the action of $g$ on $\hhs$ has an
axis $A_s$. 
Let $l$ be a leaf in $A_s$.
If this axis is not properly embedded in $\hhs$, say in the forward
$g$ direction then $(g^n(l))$ converges to a finite collection of
leaves $\{ S_i \}, 1 \leq i \leq j$. Since $g$ does not fix any leaf,
then $j$ is even, that is,   $j = 2k$, and $g(S_k) = S_{k+1}$. The stable  leaves
$S_k, S_{k+1}$ 
are in the boundary of two adjacent lozenges which share an
unstable side $U$. Then $g(U) = U$ which proves that $g$ does
not act freely on $\oo$, contradiction to assumption. 
We conclude that $A_s$ is properly embedded.
If the sequence is not a master sequence for some point in $\partial \oo$, 
then all $g_n(l)$ have a common ideal point. This is not possible under
the bounded condition.
It follows that the sequence $(g^n(l))$ is a master  sequence for a point $a$ in $\partial \oo$
and $(g^{-n}(l))$ (still with $n \rightarrow +\infty)$) is a master sequence for
a point $b$ in $\partial \oo$. Clearly $b, a$ are the source and sink for
the action of $g$ on $\partial \oo$. Their projections to $\mathcal R$ are
the source/sink for the action of $g$ on $\mathcal R$.

The arguments above 
with the sequence $(g^n(l))$ show that $a$ is not an ideal point of a stable leaf.
In the same way, applying the same arguments to the action of $g$ on
$\hhu$ shows that the point $a$ is not an ideal point of an unstable leaf.
It follows that $\mathcal E(a) = \{ a \},  \   \mathcal E (b) = \{ b \}$. 
If  $a = b$ then $a$ is the source and sink of the transformation $g$ acting 
on $\partial \oo$. But then all the leaves $g_n(l)$ have to have an ideal
point in $a$. This contradicts the fact that $a$ is not an ideal point of a stable leaf.
It follows that the source and sink of the action of $g$ on $\mathcal R$ 
are different from each other.

This finishes the proof of corollary \ref{perioaction}.
\end{proof}

Recall that two rays $s_0, s_1$ in leaves of $\oos$ or $\oou$ define the same ideal
point in $\partial \oo$ if and only if there is a chain of rays
$s_0 = \tau_0, ...., \tau_1 = s_1$ in $\oos$ and $\oou$ 
alternatively so that $\tau_i$ and
$\tau_{i+1}$ form a perfect fit.
% and all the rays $\tau_j, 0 \leq j \leq k$ 
%define the same ideal point in $\partial \oo$.

\begin{corollary}{}{}
Let $a, b$ be points in $\partial \oo$. Then there is a most one 
finite chain $\{ r_j, 1 \leq j \leq k \}$ of slice leaves of $\oos$ and $\oou$
so that:
$a$ is an ideal point of $r_1$, $b$ is an ideal point of $r_k$; \ for
every $1 \leq j \leq k -1$
 $r_j, r_{j+1}$ share an ideal point, and no two consecutive
slices share a subray.
\label{path1}
\end{corollary}

\begin{proof}{}
Otherwise take a minimal subchain and obtain a non trivial path
from $\oos(x)$ (or $\oou(x)$) to itself (for some $x$). This contradicts
the second assertion of Proposition \ref{noident}.
\end{proof}

\begin{define}{}{}The chain described in corollary \ref{path1} is
called the minimal path from $a$ to $b$.
\end{define}

\begin{corollary}{}{}
Suppose that a sequence $(g_n)$ in $\pi_1(M)$ satisfies the following:
there are $h, f$ in $\pi_1(M)$ and $i_n$ in ${\bf Z}$ so that for all $n$ in ${\bf N}$,
either $g_n = h f^{i_n}$ or $g_n = f^{i_n} h$.
Suppose that either $i_n \rightarrow \infty$ or
$i_n \rightarrow -\infty$.
Then the sequence $(g_n)$ has a source and sink for its action
on $\ro$ and they are different from each other.
\label{repeat}
\end{corollary}

\begin{proof}{}
The point is that the transformation $f$ is fixed for these sequences.
For simplicity assume that $i_n \rightarrow \infty$.
Let $a_1, b_1$ be the source/sink pair of $f$ acting on $\ro$,
which exist by corollary \ref{perioaction}. If $g_n = h f^{i_n}$ then
it is immediate to see that the source of the sequence
$(g_n)$ is $a_1$ and the sink if $h(b_1)$. If $g_n = f^{i_n} h$, then
the source of $(g_n)$ is $h^{-1}(a_1)$ and the sink is $b_1$.
If $i_n \rightarrow -\infty$ then the roles of $a_1$ and $b_1$
get reversed.
\end{proof}

We can now improve some properties of the escape lemma
\ref{trapbound}:

\begin{lemma}{}{}
Let $\Phi$ be a bounded pseudo-Anosov flow. Suppose that $(l_n)$ is a sequence
of leaves in $\oos$ (or in $\oou$). Suppose that for each $n$ there are 
distinct ideal points
$a_n, b_n$ of $l_n$ so that both $\lim_{n \rightarrow \infty} a_n,
\  \   \lim_{n \rightarrow \infty} b_n$ 
exist and are equal to $a_{\infty},
b_{\infty}$ respectively. 
Then

\begin{itemize}

\item If $a_{\infty} = b_{\infty} = z$ then $(l_n)$ converges
to $a_{\infty}$ in $\oo \cup
\partial \oo$, unless $(l_n)$ has a subsequence $(l_{n_k})$ satisfying the
following: each $l_{n_k}$ is a singular leaf, $(l_{n_k})$ does not escape
compact sets in $\oo$, and $(a_{n_k}), (b_{n_k})$ converge to 
$z$ from the same side.

\item If $a_{\infty} \not = b_{\infty}$ then the sequence $(l_n)$ converges to a finite
collection $\{ E_j, 1 \leq j \leq m_0 \}$ of leaves non separated from each other,
so that $a_{\infty} \in \partial E_1, \  b_{\infty} \in \partial E_{m_0}$.
The collection $\{ E_j \}$ is completely determined by $a_{\infty}, b_{\infty}$.

\end{itemize}
\label{path2}
\end{lemma}

\begin{proof}{}
Suppose first that $a_{\infty} = b_{\infty}  = z$. Suppose that $(l_n)$ does not escape
compact sets in $\oo$. Then a subsequence $(l_{n_k})$ converges to a sequence
of non separated leaves $\{ E_j, 1 \leq j \leq m_0 \}$.  We may assume that
the sequence $(l_{n_k})$  is nested. The escape lemma
shows that $a_{\infty}, b_{\infty}$ are ideal points of the first and last of 
these leaves. 
Here we need to use the property on ``sides" of $z$. 
The concern is that the leaves 
$l_{n_k}$ could be singular and two ideal points $a_{n_k}, b_{n_k}$
of $l_{n_k}$ could
converge to the same point in $\partial \oo$ and still the sequence $(l_{n_k})$
does not escape in $\oo$. With the additional hypothesis we obtain that $a_{\infty},
b_{\infty}$ could not be equal, contradiction.
We conclude that in this case $(l_n)$ escapes compact sets. In addition
the escape lemma implies that the sequence 
$(l_n)$ can only accumulate in $a_{\infty}$.
This finishes the proof in this case.

 Now suppose that $a_{\infty}, b_{\infty}$ are distinct. Take a subsequence $(l_{n_k})$
which converges to a finite collection of leaves $\{ E_j, 1 \leq j \leq m_0 \}$. 
The escape lemma shows that this collection produces a path from 
$a_{\infty}$ to $b_{\infty}$.  This path is minimal, it only 
includes the leaves non separated from each other.
Lemma \ref{path1} shows that this path
is unique.  Consider any other subsequence $(l_{m_k})$ which converges
in $\oo \cup \partial \oo$. If it escapes in $\cd$, then the escape lemma
implies that $a_{\infty} = b_{\infty}$, contradiction. Otherwise the arguments
above show that the subsequence limits to a chain producing a path
from $a_{\infty}$ to $b_{\infty}$. As such it must be the path above.
This implies that the full sequence $(l_n)$ converges to 
$\{ E_j, 1 \leq j \leq k_0 \}$. 

This finishes the proof of the lemma.
\end{proof}

The next proposition will be used throughout the arguments in this section.
It involves a type of rigidity of the foliations $\wls, \wlu$, which implies the 
convergence group property for certain sequences.

\begin{proposition}{}{}
Let $(g_n)$ be a sequence of distinct elements in $\pi_1(M)$, so that
one of the following conditions occur:

\noindent i) There is a periodic point $x$ in $\oo$ with $(g_n(x))$
not escaping in $\oo$; or 

\noindent ii) There are distinct leaves $l_0, l_1$ of $\oos$ (or $\oou$)
which are non separated from each other in their respective
leaf space and so that: for a 
subsequence $(n_i), i \in {\bf N}$ then 
both $g_{n_i}(l_0)$ and $g_{n_i}(l_1)$  intersect a fixed compact
set $K$ of $\oo$ for all $i$.

Then: in Case i) there is a subsequence $(n_i)$ with 
$g_{n_i}(x) = y$ for all $i$, which implies that:
there are fixed $h, f$ in $\pi_1(M)$ so that
$g_{n_i} =  f^{m_i} h$ for some $m_i$ in ${\bf Z}$,
where $h(x) = y$ and
$f$ is a generator $Stab(y)$. In Case ii) there is a subsequence 
$(n_i)$ 
%$(n_i)$ 
and $f, h$ in $\pi_1(M)$  so that  for all $i$

$$g_{n_i}(l_0) = e_0, \ 
g_{n_i}(l_1) = e_1; \ \ 
g_{n_i} =  f^{m_i} h \ \ {\rm where} \ \ 
h(l_0) = e_0, \ h(l_1) = e_1  \ {\rm and}  \ 
f \in Stab\{ e_0 \cup e_1 \}$$

\noindent
By the previous corollary, in either case it follows that
$(g_n)$ has a subsequence with a source/sink for its action
on $\ro$ and they are different from each other.
\label{rigid1}
\end{proposition}

\begin{proof}{}
In case i) since $g_n(x)$ does not escape in $\oo$, there is
a subsequence $(n_i)$ with
$g_{n_i}(x) \rightarrow y$ and $y \in \oo$.
Since the orbit of $x$ under $\pi_1(M)$
is discrete in $\oo$ as $x$ is {\underline {periodic}}, we may
assume that $g_{n_i}(x) = y$ for all $i$. The conclusion
of case i) follows.

In case ii) up to subsequence assume that 
$(g_{n_{i}}(l_0))$ converges to $e_0$ and
$(g_{n_{i}}(l_1))$ converges to $e_1$.
Furthermore $e_0$ cannot be equal to $e_1$ because
$g_n(l_0), g_n(l_1)$ are distinct.
In particular $e_0$ is non separated from $e_1$.
In addition the only distinct non separated leaves which are very
close to both $e_0$ and $e_1$ respectively are 
%$g_{n_{i_k}}(l_0)$ and 
%$g_{n_{i_k}}(l_1)$. 
$e_0, e_1$ themselves. It follows that for $i$ big enough
$g_{n_{i}}(l_0) = e_0$ and 
$g_{n_{i}}(l_1) = e_1$. 
We may assume this is true for all $i$. The conclusion of case 
ii) follows. This finishes the proof of the proposition.
\end{proof}
 
We are now ready to prove the convergence group theorem.

\begin{theorem}{}{}
Let $\Phi$ be a bounded pseudo-Anosov flow and let
${\mathcal R}$ be the flow ideal boundary. Then $\pi_1(M)$
acts as a convergence group on ${\mathcal R}$.
\label{conver1}
\end{theorem}

\begin{proof}{}
The proof is somewhat tricky. This is in great part due to the
existence of perfect fits and/or pairs of non separated leaves in
$\oos, \oou$. For example it may be that a sequence of leaves
$(l_n)$ in $\oos$ converges to more than one leaf $e_1 \cup ...
\cup e_j$. Then in the limit many more points in $\partial \oo$ 
are identified by the equivalence relation.
Since we are considering the action on $\ro$, one has to be really
careful when considering these additional possible identifications.

We will use the induced actions on $\mathcal R, \oo, \partial \oo,
\mi$ as needed. 
We use the same notation for an element $g$ in $\pi_1(M)$
acting on any of these spaces.
Let $(g_n)$ be a sequence of distinct elements
of $\pi_1(M)$. In each case we prove the convergence group
property in that particular situation. 
Since the convergence group property concerns {\underline
{subsequences}}, we will take subsequence at will and
many times abuse notation and keep the same notation for
the subsequence.
Almost all of the proof will be done looking at the action on 
$\cd = \oo \cup \partial \oo$. This is because, as we will see, there is a 
very good control of the action on the foliations $\oos, \oou$. 

First a preliminary step concerning whether $g_n$ preserves orientation
in $\oo$ or not. Up to subsequence we may assume that either all $g_n$
preserve orientation in $\oo$, or all reverse orientation.
In the second case consider a second sequence $h_n = g_n (g_1)^{-1}$.
The sequence $(h_n)$ preserves orientation in $\oo$. If we prove 
the convergence group property for a subsequence of $(h_n)$ then the same
follows for $(g_n)$. Hence we can do the following:

\vskip .08in
\noindent
{\bf {Assumption in all cases}} $-$ Every element in the sequence $(g_n)$
preserves orientation in $\oo$. It follows that they preserve orientation in $\mi$.
If $M$ is non orientable then since all $g_n$ preserve orientation
in $\mi$, we can lift to a double cover if necessary and
assume that $M$ is orientable.

\vskip .1in
\noindent
{\bf {Case 1}} $-$ No open interval in $\partial \oo$ 
converges to a point in $\partial \oo$ under some subsequence 
of $(g_n)$.

This case cannot happen and it is reasonably simple to deal with,
so we eliminate it first. 
First we show that if $l$ is a leaf of $\oos$ or $\oou$, there
cannot be a subsequence $(g_{n_k})$ of $(g_n)$ so that
$(g_{n_k}(l))$ escapes compact sets in $\oo$.
Suppose this is not true and let $l$ and $(g_{n_k})$ satisfying
this. If the endpoints of $g_{n_k}(l)$ are not getting close
together, then we produce a non trivial interval of $\partial \oo$
which does not contain any ideal point of a leaf of $\oos$ or $\oou$.
This is impossible. Hence the endpoints of $g_{n_k}(l)$ are
getting arbitrarily close together. Up to another sequence these
endpoints converge to a single point $b$ in $\partial \oo$.
Then one of the intervals $J$ of $\partial \oo$ defined by the
ideal points of $l$ will satisfy that $(g_{n_k}(J))$ converges
to $b$. This contradicts the hypothesis in this case.

Suppose that the leaf space of one of $\oos$ or $\oou$,
say $\oos$ is non Hausdorff. Let $l, r$ be leaves of $\oos$ which
are non separated from each other. Since the images of these
under any subsequence of $(g_n)$ cannot escape compact sets we 
take a subsequence $(g_{n_k})$ so that both the sequences
$(g_{n_k}(l))$ and $(g_{n_k}(r))$ converge to (possibly more than
one) leaf of $\oos$.
Then all $g_{n_k}(l)$ and $g_{n_k}(r)$ intersect a fixed compact
set $K$ of $\oo$. Since $l$ and $r$ are non separated from
each other, then we can apply Proposition \ref{rigid1} part ii).
It follows that there are $f, h$ in $\pi_1(M)$ and a further
subsequence $n_{k_i}$ so that 
$g_{n_{k_i}} = f^{m_i} h$ for all $i$. Since $h$ is fixed,
Lemma \ref{perioconv} implies that there are non trivial 
intervals of $\partial \oo$ which converge to a single
point under the subsequence $(g_{n_{k_i}})$. Again this contradicts
the hypothesis in this case. 

We conclude that the hypothesis of case 1 implies that the
leaf spaces of $\oos, \oou$ are Hausdorff.
In particular if $l$ is a leaf of $\oos$ or $\oou$
and $(g_{n_k})$ is a subsequence so that $(g_{n_k}(l))$ converges,
then it converges to a {\underline {single}} leaf of either $\oos$
or $\oou$.

Let now $p$ in $\oo$ non singular. Up to a subsequence still denoted
by $(g_n)$ assume
that $(g_n(x))$ converges for any $x$ ideal point of $\oos(p)$ or
$\oou(p)$. There are 4 such ideal points. The ideal points $\partial \oos(p)$
link $\partial \oou(p)$ in $\partial \oo$.
The limits cannot be the same
or else some non degenerate interval of $\partial \oo$ converges
to a point. Therefore both $(g_n(\oos(p))$ and $(g_n(\oou(p))$
converge. Let $A, B$ be the respective limits. If $A$ does not
intersect $B$ then some of the limits of ideal points of $(g_n(\oos(p))))$
or $(g_n(\oou(p)))$ collapse together, which again is not allowed.
We conclude that $A$ intersects $B$ and let the intersection
be $y$. Then $(g_n(p))$ converges to $y$.

This is impossible. 
Suppose that $p$ is periodic and non singular. Since the sequence $(g_n(p))$ converges
to $y$, Proposition \ref{rigid1} implies that there is a subsequence with
a source and sink for its action on $\mathcal R$. Again this contrary
to hypothesis in this case.

This contradiction finally shows that Case 1 cannot happen.
This finishes the analysis of Case 1.

%\vskip .1in
%\noindent
%{\bf {Remark}} $-$ Originally the proof of the theorem had one additional
%case before Case 2: suppose that for some $p$ in $\oo$, the sequence
%$(g_n(p))$ does not escape compact sets in $\oo$. 
%This simplified to a certain extent the analysis of Case 2
%here. The problem was that the analysis of the case $(g_n(p))$ converges in $\oo$
%turned out to be extremely intricate, due to the non Hausdorff
%behavior of $\oos, \oou$.
%In case 2 we will only use a condition about {\underline {finitely}}
%many points not escaping in $\oo$ under action by a subsequence of $(g_n)$.

\vskip .3in
\noindent
{\bf {Case 2}} $-$ There is a non trivial interval
$I_1$ of $\partial \oo$ so under some sequence $(g_{n_k})$ all
points of $I_1$ converge to a single point $w_0$ of $\partial \oo$.

%Up to a subsquence assume that $g_n(I_0) \rightarrow a$.

In order to analyse this case we need a couple of preliminary results.

\begin{proposition}{}{}
Suppose that $(g_n)$ is a sequence of distinct elements of $\pi_1(M)$.
Suppose that there are $p \not = q \in \oo$  so that
$\lim_{n \rightarrow \infty} g_n(p) =
\lim_{n \rightarrow \infty} g_n(q)$ and this is a point
in $\oo$. Let $y_0$ be the limit.
Then either $\oos(p) = \oos(q)$ or $\oou(p) = \oou(q)$. 
In addition suppose that all $g_n(p), g_n(q)$ are very near $y_0$ so
that $g_n$ has a fixed point $u_n$ near $y_0$. Let $\gamma_n = \Theta^{-1}(u_n)$.
Then 

$-$ If for all $n$ big enough, $g_n$ is associated with the forward direction
of $\gamma_n$ then $\oou(p) = \oou(q)$,

$-$ If for all $n$ big enough, $g_n$ is associated with the backwards direction
of $\gamma_n$ then $\oos(p) = \oos(q)$.
\label{samestable}
\end{proposition}

\begin{proof}{}
%We first prove that $p, q$ are either in the same stable or unstable leaf in $\oo$.
Up to subsequence assume that all $g_n(p), g_n(q)$ are in the closure of a 
sector of $y_0$. 
Suppose first that there is some subsequence of $(g_n(p))$ (or $(g_n(q))$
(still denoted in the same way)
which is constant; say the first option.
Up to precomposition with $g_1^{-1}$ we may assume that $p = y_0$ also.
Let $f$ be the generator of $Stab(y_0)$ associated to the positive
flow direction in $\gamma_0 = \Theta^{-1}(y_0)$.
Then $g_n = f^{i_n}$ where $|i_n| \rightarrow +\infty$ as $n \rightarrow \infty$.
Suppose that $i_n \rightarrow +\infty$.
Then the following happens: locally near $y_0$, $g_n$ expands
the stable direction and contracts the unstable direction.
If $g_n(q) \rightarrow y_0$ this can only happen if 
$q \in \oou(p)$.

Hence from now on assume that each of the sequences $(g_n(p))$, $(g_n(q))$
is a sequence of distinct points. 
By way of contradiction, 
up to taking subsequences,
we may assume that all sequences $(\oos(g_n(p))), (\oou(g_n(p))),
(\oos(g_n(q)))$ and $(\oou(g_n(q)))$ are nested sequences of leaves.
Notice they do not escape in $\oo$ because $y_0$ is a point in $\oo$.

Since all $(g_n(p))$ are in same sector of $y_0$ and very close to
$y_0$ then $g_n g_1^{-1}$ has a fixed point $u_n$ very close to $y_0$.
Remove a few initial terms and replace $p, q$ by $g_1(p), g_1(q)$ and
$(g_n)$ by $(g_n g_1^{-1})$. After this modification $g_n$ has
a fixed point $u_n$ near $y_0$. 
Since $g_n(p), p$ are close to $u_n$ then $\oos(u_n)$ intersects
both $\oou(p), \oou(g_n(p))$ and likewise $\oou(u_n)$ intersects
both $\oos(p), \oos(g_n(p))$.
Let $\gamma_n = \Theta^{-1}(u_n)$. 
Since $(g_n)$ are all distinct then the length of the periodic orbits
$\pi(\gamma_n)$ converges to infinity. 
This length is counted with multiplicity if all the orbits
$\pi(\gamma_n)$ are traversed more than once.

Up to subsequence we may assume that either all $g_n$ are associated with
the forward or backwards direction in $\pi(\gamma_n)$. 
Without loss of generality assume that $g_n$ is associated with 
the forward direction. We will prove that 
$\oou(p) = \oou(q)$. 
If on the other hand we assume that $g_n$ is associated with the backwards
direction of $\pi(\gamma_n)$ then an entirely analogous proof shows that
$\oos(p) = \oou(q)$.

As $g_n(p) \rightarrow y_0$ and $g_n$ associated to positive flow direction
in $\pi(\gamma_n)$ then 

\begin{itemize}

\item $\oou(u_n) \rightarrow \oou(p)$,

\item $\oos(u_n) \rightarrow \oos(y_0)$.

\end{itemize}

The reason for this is the following. The lengths of $\pi(\gamma_n)$
go to infinity. If $(\oou(u_n))$ does not converge to $\oou(p)$
suppose that $(\oou(u_n))$ converges to an unstable leaf
$w$ which is not $\oou(p)$. Then $\oou(p) \cap \oos(u_n)$
gets pushed farther and farther away from $u_n$ under $g_n$.
This is because $\oou(u_n)$ is not very close to $\oou(p)$.
But then it follows that $(g_n(p))$ cannot converge to $y_0$, contradiction.
This shows that $(\oou(u_n))$ converges to $\oou(p)$. An entirely
analogous argument proves that $(\oos(u_n))$ converges to $\oos(y_0)$.

Now suppose that $\oou(p) \not = \oou(q)$. Then first notice that
$\oou(q)$ also intersects $\oos(u_n)$. This is because the same arguments
as above applied to $q, g_n(q)$ produce a fixed point $u'_n$ of $g_n$ very
close to $y_0$. But the transformation $g_n$ can only have one
fixed point near $y_0$, so it follows that $u_n = u'_n$ and hence
$\oou(q)$ intersects $\oos(u_n)$. Once we get that then the same arguments
we had before shows that $(g_n(q))$ cannot converge to $y_0$.
We conclude that $\oou(p) = \oou(q)$.

This finishes the proof of proposition \ref{samestable}.

\end{proof}

\noindent
{\bf {Remark}} $-$ Notice that this proposition does not assume 
the bounded hypothesis on $\Phi$.

\begin{lemma}{}{}
Let $\Phi$ be a bounded pseudo-Anosov flow.
There is $m_0$ in ${\bf N}$ so that the following happens.
Suppose $(g_n)$ is a sequence of distinct elements of $\pi_1(M)$.
Let $p_1, ..., p_m$ be a finite collection of points in $\oo$ so that:

\begin{itemize}

\item The sequence $(g_n(p_i))$ converges to a point in $\oo$ for each $i$,

\item $\{ \oos(p_i), 1 \leq i \leq m \}$ is a collection of pairwise distinct leaves of $\oos$,

\item $\{ \oou(p_i), 1 \leq i \leq m \}$ is a collection of pairwise distinct leaves of $\oou$.
\end{itemize}

Then $m \leq m_0$.
\label{upbound}
\end{lemma}

\begin{proof}{}
Since $\Phi$ is a bounded pseudo-Anosov flow there is an upper bound to the
number of fixed points of any $g$ in $\pi_1(M)$ by Theorem \ref{chain}.
Let $m_0$ be such an upper bound.

Let $y_i = \lim_{n \rightarrow \infty} g_n(p_i)$ for each $i$. By Proposition
\ref{samestable} the last two conditions of the hypothesis imply that 
the set $\{ y_1, ..., y_m \}$ is a collection of distinct points.
Up to a subsequence of $(g_n)$  assume that for each $i$, then
all $g_n(p_i)$ are very near $p_i$ and in the same sector of $p_i$. 
So for big enough $n_0$ and fixed $n > j > n_0$ it follows that
$g_n g_j^{-1}$ has a fixed point $u_i$ very near $y_i$. If the $u_i$ 
are sufficiently near $y_i$ then the points $\{ u_1, ..., u_m \}$
are distinct. Since they are all fixed points of the fixed transformation
$g_n g_j^{-1}$ it follows that $m \leq m_0$.
This proves the lemma.
\end{proof}

\begin{define}{}{} We say that the sequence $(g_n |_I)$ locally
uniformly converges to $z$ if for any compact set $K \subset I$,
the functions $g_n |_K$ converge uniformly to the constant
function with value $z$.
\end{define}

\vskip .2in
\noindent
{\underline {Analysis of case 2 of theorem \ref{conver1}}}.

Recall that in this case there 
is a non degenerate interval $I_1$ in $\partial \oo$ with 
$(g_n(I_1))$ converging to a point up to subsequence.
We assume up to subsequence that $\lim g_n(I_1) = w_0$. Let $J$ be the maximal
open interval so that $(g_n | J)$ locally uniformly converges to $w_0$. Let $a,
b$ be the endpoints of $J$. First we consider the case that $a = b$, that is,
$J = \partial \oo - \{ a \}$. Then we are done: $a$ is the source for the
sequence $(g_n)$ and $w_0$ is the sink. 
Projecting to $\rr$ we obtain the convergence group property for
$(g_n)$.

Therefore we 
assume from now on that $a, b$ are
distinct.

We will prove 2 facts which will be enough to finish the analysis of Case 2.

\vskip .1in
\noindent 
{\bf {Fact 1}} $-$ There is a non trivial open interval $I \subset \partial \oo - J$ and 
with an endpoint $a$ so that $(g_n | I)$ locally uniformly converges to $w_1$ with
$w_1 \sim w_0$.

\vskip .1in
\noindent 
{\bf {Fact 2}} $-$ $b \sim a$.

\vskip .1in
\noindent
{\bf {Proof of fact 1}} $-$ 
As an initial subcase suppose first
that $a$ is not an ideal point of a
leaf of $\oos$ or $\oou$. Then let ${\cal T} = (l_n)$ be a master 
sequence for the point
$a$ made up of stable leaves and likewise let ${\cal T}_1 = (u_n)$ be a 
master sequence for $a$
made up of unstable leaves. 
For simplicity we assume that no $l_n$ or $u_n$ is singular.
For $i$ big one of the endpoints $a_i$ of $l_i$ is in
$J$, so $(g_n(a_i))$ converges to $w_0$.
The other endpoint $b_i$ of
$l_i$ is not in $J$ (for $i$ big). 
Up to taking subsequences we assume that $(g_n(b_i))$ converges to a point
$x_i$ for $i$ big. 
If there is a subsequence $i_k$ so that $x_{i_k} = w_0$ for all $k$,
this implies that $a = b$ which was dealt with before. Hence
assume that $x_i \not = w_0$ for all $i$ big.

By lemma \ref{path2} $x_i$ and $w_0$ are connected by a path  of non separated
stable leaves. Hence we may assume that for $i$ big $x_i$ is constant
and so equal to some point $x'$.  The same argument applies to the master
sequence of unstable leaves.  
Since $(l_n)$ and $(w_n)$ are eventually nested it follows that the $x'$ is
also connected to $w_0$ by a path of non separated unstable leaves.
By proposition \ref{noident} the path connecting $w_0$ to $x'$ is unique
hence

$$ \lim_{n \rightarrow \infty} g_n(l_i) \ \ = \ \ 
 \lim_{n \rightarrow \infty} g_n(u_i)$$

\noindent
for $i$ big. But this is impossible as one is made of stable leaves
and the other is made of unstable leaves..

We conclude that the point $a$ and likewise $b$ is an ideal point of a leaf of
$\oos$ or $\oou$. Let $l_1, ..., l_{k_0}$ be the leaves with ideal point $a$ and
order then so that $l_k$ separates $l_{k'}$ from $l_{k"}$ if 
$k' < k < k"$.
Then $l_k$ makes a perfect fit with the leaves 
$l_{k-1}$  and $l_{k+1}$ both of which are in the
other foliation.

As in Definition \ref{stan}
we use a standard sequence ${\mathcal V} = (E_i)$ of convex chains defining
the point
$a$, where each $E_i$ is made
up of $k_0 - 2$ segments in $\oos$ or $\oou$ and 2 rays of
these foliations. Call these segments/rays $e^k_i$, where 
$1 \leq k \leq k_0$. Then each $e^k_i$ intersects the ray $l_i$ with
ideal point $a$. Both $e^1_i, e^{k_0}_i$ are rays and the other $e^k_i$ 
are compact segments.
Let $y^0_i$ be the ideal point of the first ray (with ideal point in $J$)
and $y^{k}_i$, $1 \leq k \leq k_0$ be the other corners of $E_i$.
Notice that $y^k_i$ is in $\oo$ for any $i$ and for $1 \leq k < k_0$.
%and let $E^k_i, 0 \leq k \leq k_0$ be the edges of $E_i$.
The collection $\{ y^{k_0}_i, i \in {\bf N}   \}$ is 
a collection of points in $\partial \oo$.
We refer to to fig. \ref{qg10}.
The goal here is to show that for big enough 
$i$ then $(g_n(y^{k_0}_i))$ converges to a fixed point $w_1$ (independent of $i$)
which is equivalent to
$w_0$ under $\sim$. Then we will obtain the interval $I$ as required.
This will prove fact 1.

\begin{figure}
\centeredepsfbox{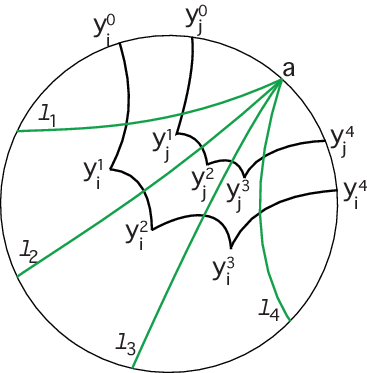}
\caption{A standard sequence for the ideal point $a$. In this case $k_0 = 4$.
We depict the convex chains $E_i$ and $E_j$ where $i < j$.}
\label{qg10}
\end{figure}

We do the proof for $k_0 \geq 2$, that is, there is more than
one leaf of $\oos \cup \oou$ with ideal point $a$. The
case $k_0 = 1$ is much easier.

Since there are countably many $\{ y^k_i, \ 1 \leq k \leq k_0,
\ i \in {\bf N} \}$ we assume up to subsequence
(in $n$) that for each $i$ and for each $k, \ 0 \leq k \leq k_0$, then 

$$\lim_{n \rightarrow \infty} \ g_n(y^k_i) \ \ \ {\rm exists \ in} \ 
\oo \cup \partial \oo. \ \ 
{\rm Notice \ that}
\ \ w_0 \ = \ \lim_{n \rightarrow \infty} g_n(y^0_i).$$

\noindent
The difficulty is that we do not know for any given $i$,
whether the limits are in $\oo$ 
or in $\partial \oo$, for $1 \leq k < k_0$. 
Consider the points $y^k_i, y^r_j$, where $1 \leq k, r < k_0$. Then these
points are in $\oo$. If $| k - r | \geq 2$
% and say $j > i$ then 
%the leaves $l_{k+1}, l_{k+2}$ and $l_{k+3}$ separate 
%$y^k_i$ from $y^r_j$. 
%%These leaves are not in the same foliation.
%In addition $l_{k+3}$ is non separated from $l_{k+1}$.
%It follows that  $\oos(y^k_i) \cap \oou(y^r_j) = \emptyset$
%and $\oos(y^k_i) \cap \oou(y^r_j) = \emptyset$.
%In particular 
then $\oos(y^k_i), \oos(y^r_j)$ are distinct leaves
and so are $\oou(y^k_i), \oou(y^r_j)$. 
Even for $|k-r| = 1$ this is true for $i \not = j$.
By  Lemma 
\ref{upbound} there are at most $m_0$ points
in the collection $\{ y^k_i, \ 1 \leq k < k_0, i \geq 1 \}$
so that $\lim_{n \rightarrow \infty} g_n(y^k_i)$ is a point
in $\oo$. Therefore

\vskip .1in
\noindent
{\underline {Conclusion}} $-$ There is $i_0$ so that if $i > i_0$ then
 $\lim_{n \rightarrow \infty} g_n(y^k_i)$
is in $\partial \oo$ for any $1 \leq k < k_0$.
\vskip .08in

We now proceed by induction on $k, 1 \leq k \leq k_0$.
Fix $i \geq i_0$.
If $(g_n(e^1_i))$ escapes compact sets in $\oo$, then 
$\lim_{n \rightarrow \infty} g_n(y^1_i) = w_0$ as well, because
$\lim_{n \rightarrow \infty} g_n(y^0_i) = w_0$.
Suppose this is not the case.
%Then $e^0_i$ are contained in $r^0_i$ (say) stable leaves.
Then $\lim_{n \rightarrow \infty} g_n(e^1_i)$ is a collection of leaves of $\oos$ or $\oou$,
which are 
non separated from each other, at least one of which has one ideal point $a$.
Since $y^1_i$ are in $e^1_i$, then by the escape lemma, Lemma  \ref{trapbound}, it follows that 

$$\lim_{n \rightarrow \infty} g_n(y^1_i) = x^1_i$$

\noindent
is an ideal point of one of these non separated leaves. 
%For simplicity of
%notation we omit the dependence on $i$ here. 
In particular $x^1_i$ and $w_0$ are
ideal points of stable leaves which are non separated from each other and in
particular $x^1_i \sim w_0$.
%There is a monotonicity in $i$ with the collection $\{ e^1_i \}$, hence
%for $i$ big enough $x_i$ is constant equal to $w_1$.

\vskip .1in
Now proceed by induction on $k$.
Suppose that

$$\lim_{n \rightarrow \infty} g_n(y^{k-1}_i) \ = \ x^{k-1}_{i} \ \ \ 
{\rm and} \ \ \ x^{k-1}_i \sim w_0.$$

\noindent
If $(g_n(e^{k}_i))$ escapes compact sets in $\oo$, then 
$\lim_{n \rightarrow \infty} g^n(y^{k}_i) = x^{k-1}_i$. Otherwise the same proof as
above shows that that $\lim_{n \rightarrow \infty} g_n(y^{k}_i) = x^k_i$ and $x^k_i
\sim x^{k-1}_i$. Consequently $x^k_i \sim w_0$. We conclude that 

$$\lim_{n \rightarrow \infty} g_n(y^{k_0}_i) = t_i \sim w_0.$$

\noindent
%Now we recall that 
%this $w_{k_0}$  depends on $i$ so we denote it by $w_{k_0}(i)$.
We stress that this works for any $i \geq i_0$.
Since there are only finitely many points in $\partial \oo$ which are equivalent
to $w_0$ and there is a weak monotonicity property, it follows that 
$t_i$ is constant equal to $w_1$ for $i$ big. In particular this produces an open
interval $I \subset \partial \oo - J$ with one ideal point $a$ so that for any $z
\in I$ then $\lim_{n \rightarrow \infty} g_n(z) = w_1$ for $i \geq i_0$ and
so that $w_1 \sim w_0$. Then $(g_n | I)$ locally uniformly converges to 
$w_1 \sim w_0$. 

This finishes the proof of fact 1.

\vskip .1in
\noindent
{\bf {Proof of fact 2}} 

The proof will make essential use of Fact 1 and its proof as well.

We apply the arguments we have done so far to the
sequence $(f_n)$ where
$f_n = g^{-1}_n$. 
This is a sequence of distinct elements of $\pi_1(M$).
As before we know that case 1 cannot happen to the sequence $(f_n)$.
%Then up to taking a subsquence it follows that
%there is a bounded number of points $\{ q \}$ in $\oo$ so that
%$\lim_{n \rightarrow \infty} g_n(q)$ exists, the limit is in $\oo$ and
%the limits are distinct for distinct points.
%First we remark that for any ???? $p \in \Theta(g^{-1}_n(p))$ then $(g_n(p))$ 
%escapes compact sets. Suppose this is not the case. Let $f_n = g^{-1}_n$.
%The arguments of ???? show the following: assume up to subsequence that
%$(f_n(p))$ converges to $p_0$. Now let $h_n = f_n f^{-1}_{n_0}$. Then $h_n$
%has fixed points. 
%But $h^{-1}_n = f_{n_0} f^{-1}_n = g^{-1}_{n_0} g_n$ then has fixed points.
%The proof of Case 1 shows that the sequence $(g^{-1}_{n_0})$ has a source 
%$d_0$ and a sink $d_0$. Then $c_0$ and $g_{n_0}(d_0)$ are the source and sink
%for the sequence $(g_n)$.
Recall the maximal open interval $J$ so that $(g_n | _J)$ converges locally
uniformly to $w_0$. We may assume that $J$ is not $\partial \oo - \{ a \}$ for
otherwise we are done. Hence there is an interval $I$ as in the proof of fact 1.

We now apply a similar construction as in the proof of fact 1
to the sequence $(f_n)$ and $w_0$. Let $\mathcal D = (d_i)$
be a master sequence defining the ideal point $w_0$. 
We assume that all convex chains $d_i$ have length
$k_1$ and ideal points/corners  $\{ v^k_i \},  0 \leq k \leq k_1$, where
$v^0_i, v^{k_1}_i$ are in $\partial \oo$ and the rest in $\oo$. 
As in the proof of fact 1, lemma \ref{upbound} implies that there is
$i_1$ so that  

$$\forall i \geq i_1, \ \ \ \lim_{n \rightarrow \infty} f_n(v^k_i)) \ \ \ {\rm exists
\ and  \ is \ in} \ \ \partial \oo
\ \ \ {\rm for \ all } \ \ 1 \leq k < k_1.$$

\noindent
Let 

$$a_i \ = \ \lim_{n \rightarrow \infty} f_n(v^0_i), \ \ \ \ 
b_i \ = \ \lim_{n \rightarrow \infty} f_n(v^{k_1}_i).$$

\noindent
Exactly as in the proof of fact 1, we obtain 
that $a_i \sim b_i$ for any $i \geq i_1$.

\vskip .1in
Now let $I_i$ be the interval of
$\partial \oo$ defined by $v^0_i, v^{m_0}_i$ and containing $w_0$.
Fix a compact set $C$ contained in the interval $J$
from the proof of fact 1. Recall that $J$ is an 
{\underline {open}} interval in $\partial \oo$.
For any fixed $i$ then 
$g_n(C) \subset I_i$ for $n$ sufficiently big. This is because $(g_n | J)$
converges locally uniformly to $w_0$. 

Up to a subsequence in $i$ we may assume that
$\lim_{i \rightarrow \infty} a_i$
exists and similarly $\lim_{i \rightarrow \infty} b_i$ exists. Let these limits
be $a', b'$ respectively. 
In particular

$$a' \ = \ \lim_{i \rightarrow \infty} \ \left(\lim_{n \rightarrow \infty}
f_n(v^0_i) \right),
\ \ \ \ b' \ = \ 
\lim_{i \rightarrow \infty} \ \left(\lim_{n \rightarrow \infty} f_n(v^{k_1}_i)
\right).$$

\noindent
Since $g_n(C) \subset I_i$ for $n$ big, then $C \subset f_n(I_i)$ for
$n$ big.
Therefore one of the intervals of $\partial \oo$ determined by $a', b'$,
call it $[a',b']$ contains $J$. Suppose that the interval is strictly bigger than $J$, that
is the closure of $J$ is not equal to $[a',b']$.  Recall that
$\partial J = \{ a , b \}$. For example suppose that
$b$ is in the interior of $[a',b']$, so there is $c$ in $[a',b']$,
with $[c,b]$ disjoint from $J$ and $[c,b]$ contained in $[a',b']$.
The definition of $a', b'$ then implies that
$J \cup (c,b]$ is an open interval strictly bigger than $J$ where
$(g_n)$ locally converges to $w_0$.
This is a contradiction to the assumption of maximality of $J$.

We conclude that $[a',b']$ is equal to the closure of $J$. Since 
$a' = \lim_{i \rightarrow \infty} a_i$ \ and \
$b' = \lim_{i \rightarrow \infty} b_i$. Up to switching $a', b'$ then
$a' = a$, $b' = b$. But $a_i \sim b_i$ and
$\sim$ is a  closed equivalence relation in $\partial \oo$.
It follows  that $a \sim b$ as we wanted to prove.

This finishes the proof of Fact 2.

\vskip .1in
With this property we can quickly finish the proof of Case 2. By Fact 2  if
$J$ is the maximal open interval with $(g_n | J)$ locally uniformly converges to 
$w_0$, then $\partial J = \{ a, b \}$ and $a \sim b$. Using Fact 1, there is an
open interval $I_0$, with 

$$I_0 \cap J = \emptyset, \ \ a \in \partial I, \ \ \ {\rm and} \ \ \  
(g_n | I_0) \ \ {\rm converges \ to } \ \ w_1, \ \ \   w_1 \sim w_0.$$

\noindent
Let $I_1$ be the maximal open interval with $(g_n | I_1)$ locally uniformly
converges to $w_1$. Then Fact 2 shows that $\partial I_1 = \{ a, c  \}$ and $c
\sim a$ and so $c \sim b$. Since there are finitely many intervals in
$\partial \oo - \ee(a)$ we show that for any such interval $I'$, then $(g_n|_{I'})$
converges locally uniformly to a point $w$ with $w \sim w_0$. This shows 
that $\ee(a)$ is the
source and $\ee(w_0)$ is the sink for the appropriate subsequence of $(g_n)$
acting on $\partial \oo$. This
finishes the proof of Case 2.

\vskip .1in
This shows that $(g_n)$ always has a subsequence with source/sink behavior. This
finishes the proof of theorem \ref{conver1}.
\end{proof}

\section{Uniform convergence group}

The purpose of this section is to prove the following result.

\begin{theorem}{}{}
Let $\Phi$ be a bounded pseudo-Anosov flow. Let $\mathcal R$ be the quotient
of $\partial \oo$ by the equivalence relation $\sim$. Then the action of 
$\pi_1(M)$ on $\mathcal R$ is a uniform convergence group.
\label{uniform}
\end{theorem}

\begin{proof}{}
By the convergence group theorem, theorem \ref{conver1},
 we only have to prove that any point in $\mathcal R$ is
a conic limit point for the action of $\pi_1(M)$ on $\mathcal R$: given
$p$ in $\mathcal R$ there is a
sequence $(g_n)$ in $\pi_1(M)$ and $a \not = b$ in $\mathcal R$ so that 
the sequence 
$(g_n(p))$ converges to $a$ and the sequence
$(g_n(q))$ converges to $b$ for any $q$ in $\mathcal R$, with
$q \not = p$.

\vskip .1in
\noindent
{\bf {Notation}} $-$ 
We denote by 
$\eta: \partial \oo \rightarrow \mathcal R$ the projection map.
\vskip .1in

We will mostly work in $\partial \oo$ and $\cd = \oo \cup \partial \oo$
analysing the actions of $\pi_1(M)$ on these spaces. Let $x$ in $\partial \oo$ with 
$p = \eta(x)$.
Very roughly the sequence $(g_n)$ will be obtained by zooming in to $x$. This is easily
done in $\cd = \oo \cup \partial \oo$. This will need adjustments to take
into account the three dimensional situation in $\mi$.

Being a conical limit point is associated with a {\em {geometrical}} property
in $\mi$. Let us recall the situation of $\pi_1(M)$ Gromov hyperbolic:
if $y$ is a point in $\partial \mi$, then to show that $y$ is a conical limit
point, one ``zooms in" to $y$. To do that one gets a geodesic
ray $r$ in $\mi$ with ideal point $y$. Then using that $M$ is compact, take
accumulation points of the projection of $r$ in $M$. Use this to produce
covering translations $g_n$ and points $v_n$ in $r$  with $(v_n)$
converging to $y$ in $\mi \cup \si$   so that $(g_n(v_n))$ converges to a point $v^*$ in 
$\mi$. Assuming that $(g_n(r))$ also converges 
then one gets the following. The limit of $(g_n(r))$ is a full geodesic $r'$ and
the ideal points of $(g_n(r))$ converge to an ideal point $a$ of $r'$.
Let $b$ be the other ideal point of $r'$.
Then one can easily show that 
$(g_n(y))$ converges to $a$ and $(g_n(z))$ converges to $b$ for any 
$z \in \si = \partial \mi$ distinct from $y$. Here $a \not = b$ as they are
the ideal points of a geodesic $r'$.

The major problem that we have in the flow setting is that, at this point,
we do not have 
any connection between the flow ideal boundary and the geometry
of $\mi$. In particular one cannot do the ``geometrical zooming in" which
easily proves the conical limit point property in the case that
$\pi_1(M)$ is Gromov hyperbolic. 
The proof here will be to use the flow $\Phi$ and the foliations $\oos, \oou$ to 
zoom in to a point in $\mathcal R$ or in $\partial \oo$. So if $p$ is a point
in $\mathcal R$ one can produce a sequence $(g_n)$ with $(g_n(p))$
converging to $a$ and $(g_n(z))$ converging to $b$ for any $z \not = p$.
By far the biggest problem is that one does not know a priori that 
$b \not = a$. As explained above this comes essentially for free in the
geometric situation. In our case this is much, much trickier because of
the existence of perfect fits, which produce many identifications between
points of $\partial \oo$ under $\sim$.
It is complicated to rule out identifications in the limit.

There are three cases in the proof. Two of them are very simple and are
called Preliminary cases 1 and 2. The much, much harder case is called the main
case. The main case will have two main subcases, denoted by Case A and Case B
and subcases within these.
In our setup $p$ is an arbitrary point in $\mathcal R$ and $x$ is
a point in $\partial \oo$ with $\eta(x) = p$.

\vskip .1in
\noindent
{\bf {Preliminary case 1}} $-$ Periodic ideal point.

Suppose that $x$ is an ideal point of a periodic leaf $l$ of $\oos$ or $\oou$.
Let $g$ be a generator of the stabilizer  of $l$ so that in $\mathcal R$, $p = \eta(x)$
is the source for the action of $g$ on $\mathcal R$. 
This is guaranteed by Corollary \ref{perioaction}.
Then $g_n(p) = p$ and $(g_n(q))$ converges to $b$ $-$ the attracting fixed
point of $g$ for any $q \not = p$. Since the two fixed points of $g$ acting
on $\mathcal R$ are distinct, this proves the conical limit point property for $p$.

\vskip .1in
\noindent
{\bf {Preliminary case 2}} $-$ The point $x$ is an 
ideal point of a leaf $l$ of $\oos$ or $\oou$, which is not periodic.

Without loss of generality assume that $l$ is a stable leaf. Let $L = l \times \rrrr$ 
and fix an orbit $\gamma$ in $L$. Consider a sequence $(q_n)$ in $\gamma$ escaping in
the positive flow direction so that the sequence $(\pi(q_n))$ converges to $v*$ in
$M$. Up to a subsequence assume that all $\pi(q_n)$ are in a fixed local sector
of $v^*$.  Let $\tau_n$ be the segment of $\gamma$ between $q_0$ and $q_n$.
By the closing Lemma, 
up to subsequence (removing a few terms may change $q_0$) we can assume that,
for each $n$,  the flow segment
$\pi(\tau_n)$ is shadowed by a closed orbit, denoted by $\delta_n$.
Let $g_n$ in $\pi_1(M)$ associated to $\delta_n$ and so that $g_n(q_n)$ is
very near $q_0$. In that way $g_n$ is associated with negative flow direction
in the invariant orbit near $\tau_n$. Let $\gamma_n = g_n(\gamma)$.
By construction, the orbit
 $\gamma_n$ has a point $g_n(q_n)$ very near $q_0$.

Consider $\gamma, g_n(\gamma)$ as points in $\oo$.
Up to subsequence assume that $(g_n(\gamma))$ converges in $\oo$ to $\alpha$.
Let $u_n$ in $\oo$ near $\gamma$ with $g_n(u_n) = u_n$.
That is, $u_n$ are the orbits associated to coherent lifts of the
closed orbits $\delta_n$.
Since $g_n$ is associated with the negative flow direction in 
$\Theta^{-1}(u_n)$ then the following happens.
The arguments in the end of the proof of Proposition \ref{samestable} 
imply that 

$$(\oos(u_n)) \ {\rm converges \ to} \ l = \oos(\gamma)  \ \ {\rm and} \ \ 
(\oou(u_n)) \ {\rm converges \ to} \ \oou(\alpha).$$

\noindent
Consequently
$(u_n)$ converges to $l \cap \oou(\alpha)$.
We already proved the convergence group property for the action
of $\pi_1(M)$ on $\rr$ and we will use that to great effect here.
By the convergence group theorem, Theorem \ref{conver1}, we may assume up to
subsequence that
$(g_n)$ has a source/sink for the action on ${\mathcal R}$.
Equivalently this sequence has source and sink
sets   (or equivalence classes of $\sim$) for the action on $\partial \oo$.
By the escape lemma, we know that
$(g_n(x))$ converges to a point in $\ee(\partial \oos(\alpha))$.
The leaf 
$l$ is not periodic, so in particular it is not singular.
Therefore for any $z$ in $l$, then for big enough $n$, the unstable
leaf $\oou(z)$ intersects $\oos(u_n)$. We stress that this is necessarily
true because $l$ is not singular.
Then again because $g_n$ is associated with the negative flow direction
along $\Theta^{-1}(u_n)$ it follows that,
for any $z$ in $l$, the sequence $(g_n(\oou(z)))$ converges to $\oou(\alpha)$
- and perhaps to  other leaves of $\oou$. 
In particular for any such $z$ the ideal points of $g_n(\oou(z))$ converge
to points in $\ee(\partial \oou(\alpha))$. But there are uncountably
many such $z$, only boundedly many of which can generate points in
$\partial \oo$ which are equivalent to each other under $\sim$.
%Since this encompasses a non trivial
%interval in $\partial \oo$, 
It follows that the sink set for the sequence 
$(g_n)$ acting on $\partial \oo$ is 
 $\ee(\partial \oou(\alpha))$. In other words the sink for the sequence
$(g_n)$ acting on $\rr$ is $\eta(\partial \oou(\alpha))$.

On the other hand the points in the sequences $(g_n(\partial l))$
converge to points in  $\ee(\partial \oos(\alpha))$. 
By Proposition \ref{noident}, \
$\ee(\partial \oos(\alpha)), \ee(\partial \oou(\alpha))$ \ are disjoint subsets of $\partial \oo$.
This implies that 
$\ee(\partial l)$ is the source set for the sequence $(g_n)$ acting on $\partial \oo$.
Again because $\ee(\partial \oos(\alpha))$ and $\ee(\partial \oou(\alpha))$ are disjoint
subsets of $\partial \oo$, this 
shows that $\eta(\partial l) = \eta(x) = p$ is a conic limit point for the action
of $\pi_1(M)$ on $\rr$.

This finishes the analysis of preliminary case 2.

\vskip .1in
\noindent
{\bf {Main case}} $-$ $x$ is not an ideal point of a leaf of $\oos$ or $\oou$.

In particular this implies that $\ee(x) = \{ x \}$ is a singleton.
 This is by far the hardest case. For simplicity of notation
the subcases will be denoted without explicit referral to the main case. There
will be many steps. 
It is much more convenient to prove the result in $\partial \oo$. In this setup we will
 first prove that there is a sequence $(g_n)$ in $\pi_1(M)$ so that
when acting on $\partial \oo$: $(g_n(x))$ converges to a point $z$ and for any $y$ in $\partial
\oo$ with $y \not = x$ then 
$(g_n(y))$ converges to a point $w$.
This is not too hard. Then we will show that we can choose, perhaps
another sequence $(g'_n)$, so that the corresponding limits
$z', w'$ of $(g'_n(x))$, $(g'_n(y))$ are not equivalent under $\sim$.
This will prove the conical limit property for $p =  \eta(x)$.

\vskip .1in
\noindent
{\bf {Terminology: leaf separating ideal points}} $-$ The following terminology
will be extremely useful. Let $l$ be a leaf or line leaf or slice leaf in $\oos$ or $\oou$.
Let $A, B$ be connected subsets of $\partial \oo$. We say that 
$l$ separates $A$ from $B$ if $l$ does not have any ideal points in $A$ or $B$
and the set of ideal points of $l$, (the set $\partial l$) disconnects $A$ from $B$ in
$\partial \oo$. That is, $A$ and $B$ are contained in distinct components
of $\partial \oo - \partial l$.
A lot of the time this will be used when $A$ is a point $-$ the image of $x$ under
$g_n$, and $B$ is $g_n(K)$ for $K$ a compact set in $\partial \oo$. In the same
way given sets $C, D$ in $\cd = \oo \cup \partial \oo$ and $l$ a leaf
of $\oos$ or $\oou$ we say that $l$ separates $C$ from $D$ if $C, D$
are in different components of $\cd - (l \cup \partial l)$.

\vskip .1in
\noindent
{\bf {Step 1}} $-$ Standard path in $\oo$ associated to the ideal point $x$.

This is made up of rays, segments and slice leaves of $\oos$ or $\oou$. 
There are infinitely many parts of this path.
The starting leaf of the path is non canonical but once a starting leaf is
chosen, everything else will be canonical. Let $l_x$ be (say) a stable, non
periodic leaf. By hypothesis  $x$ is not an ideal point of $l_x$.

\begin{figure}
\centeredepsfbox{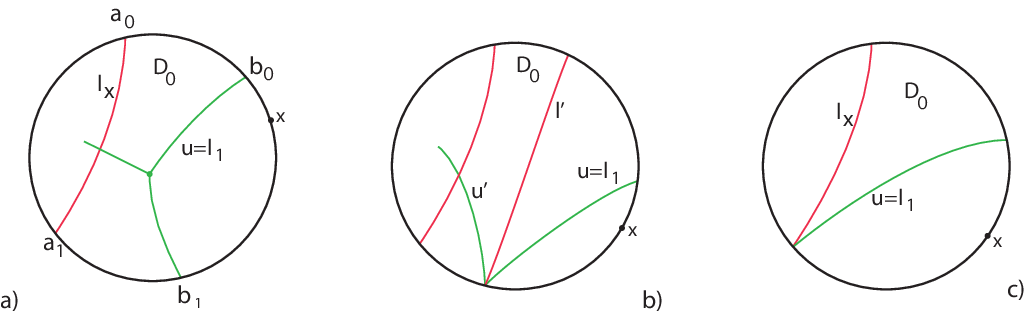}
\caption{Production of the next leaf in the standard path.
a. Singular next leaf, b. Non Hausdorff next  leaf,
c. Perfect fit next leaf. The figures depict $\oo \cup \oo$. The red curves
are stable leaves, the green curves are unstable leaves.}
\label{qg11}
\end{figure}

Let the ideal points of $l_x$ be $a_0, a_1$. 
Let $C$ be the interval of $\partial \oo$ bounded by $a_0, a_1$
and {\underline {not}} containing $x$. This compact set $C$ will be used
throughout the proof.

Consider the collection
$\{ u \in \oou, \ u \cap l_x \not = \emptyset \}$. We are interested in
the component of $u - l_x$ contained in the component $D_0$ of $\oo - l_x$ that 
limits on $x$. There are 3 possibilities:
\begin{itemize}

\item
1) There is a unique unstable leaf $u$ intersecting $l_x$ so that
$u$ has a singularity $s_1$ in $D_0$ and two full prongs $P_1, P_2$ 
of $u$ are contained in $D_0$ with ideal points $b_0, b_1$ with 
$a_1, b_1, x , b_0, a_0$ circularly ordered in $\partial \oo$.
We refer to figure \ref{qg11}, a.
In this case let $l_1 = u$. Notice that the union of the two prongs
of $l_1$ in question separate $x$ from $l_x$ and $\oos(s_1) \cap l_x = \emptyset$.

In this case the path has to cross two prongs of $s_1$ at least one stable and one
unstable to get closer to $x$. In other words, both $\oos(u_1)$ and $\oou(u_1)$
separate $x$ from $l_x$.

\item
2) There is a unique unstable leaf $u'$ intersecting $l_x$ satisfying:
$u'$ is non separated from a leaf $u$ contained in $D_0$ so that $u$
separates $x$ from $l_x$, see figure \ref{qg11}, b.
In this case let $l_1 = u$. Notice that there is a stable leaf $l'$ contained in 
$D_0$, having
an ideal point in common with $u$  and separating $u$ from $u'$.
This leaf $l'$ separates $x$ from $l_x$.
Let $s_1 = l'$.

In this case the path has to cross the leaf $l'$ and the leaf $l_1$ to get
closer to $x$ and $l', l_1$ form a perfect fit.

\item
3) There is a unique unstable leaf $u$ making a perfect fit with $l_x$, contained
in $D_0$ and separating $x$ from $l_x$. In this case let $l_1 = u$, see figure
\ref{qg11}, c.
\end{itemize}

\vskip .05in
\noindent
{\bf {Conclusion}} $-$ We stress the very important fact that in
 situations 1), 2) and 3) there is a unique  stable (or unstable) 
leaf  $l_1$ produced by the process and $l_1$
 has a line leaf which separates $x$ from $l_x$.
\vskip .1in

We now proceed by induction, starting with $l_1$ which is unstable and reversing
the roles of stable and  unstable to produce $l_2$ stable, contained in $D_0$,
with a line leaf separating $x$ from $l_1$ and: either $l_2$ is singular and intersecting $l_1$,
or $l_2$ is non separated from a leaf intersecting $l_1$, or $l_2$ makes  a perfect fit with 
$l_1$. By induction we produce a sequence of leaves
$(l_i)$ which are alternatively stable and unstable. In the same way as above under option 2) we
define  leaves $s_i$ as we defined $s_1$.
See fig. \ref{qg12}.

\begin{figure}
\centeredepsfbox{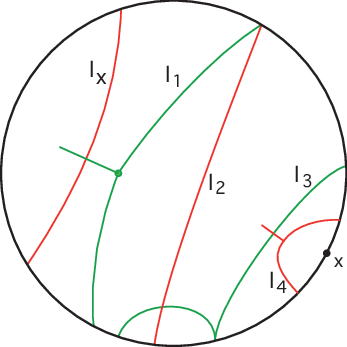}
\caption{The standard path to $x$ with starting leaf $l_x$. This figure depicts
$\oo \cup \partial \oo$ and the first 4 steps $l_1, ..., l_4$ of the standard path
to $x$. Stable leaves are red and unstable leaves are green. 
Leaves $l_1$ and $l_4$ are chosen according to Option 1 of Step 1.
Leaf \ $l_2$ is chosen according to Option 3 of Step 1. Leaf \ $l_3$ is
chosen according to Option 2 of Step 1. Accordingly there is a stable
leaf intersecting $l_2$ which is non separated from $l_3$ and $l_3$
separates $x$ from $l_x$.}
\label{qg12}
\end{figure}

Once the first leaf $l_x$ is chosen the process is canonical. It produces a way to zoom
to $x$ in $\cd = \oo \cup \partial \oo$.
Notice that the sequence of even numbered leaves $(l_{2i})$ with $ i \in {\bf N}$ provides a
master   sequence for $x$ with stable leaves, whereas the 
sequence of odd numbered leaves $(l_{2i+1})$ with $i \in {\bf N}$
provides a master  sequence for $x$ with unstable leaves. 
Let

$$\mathcal P  \ \ = \ \ \{ \  l_i , \  s_i \ \}, \ \  i \in {\bf N}$$ 

\noindent
be the  path that zooms in to $x$ in $\cd$. This path is called
a {\em {standard path}} associated to $x$. 
This is the canonical  path associated to $x$ given the initial leaf $l_x$.
Notice that if $s_i$ exists
if and only if  for such $i$, the leaf  $l_i$ is chosen according to option 2 in Step 1.

This finishes Step 1. 

\vskip .05in
There is a bound on any
chain of perfect fits, so a bound on how many consecutive times option 3) can
occur. Then we have to have at least one instance of option 1) or 2).
Therefore there is a subsequence $(m^*_n)$ of ${\bf N}$ so that for
each $m^*_n$, $l_{m^*_n}$ is produced by either option 1) or option 2).
Since there are only finitely many singular orbits of $\Phi$ and finitely many
non Hausdorff pairs up to covering translations, there is a subsequence
$(m_n)$ of the sequence $(m^*_n)$ so that for each $m_n$, 
$(l_{m_n})$ projects to the same leaf in $M$ (stable or unstable). Assume
without loss of generality that these leaves are stable leaves.
Since each of these leaves is periodic let $p_{m_n}$ be the periodic point in $l_{m_n}$.

\vskip .1in
\noindent
{\bf {Setup}} $-$ At this point we have a subsequence $(m_n)$ so that
for each $n$, $l_{m_n}$ is produced according to either option 1) or
option 2) in Step 1) and in addition every $l_{m_n}$ projects to the
same leaf in $M$, which is assumed to be stable.

\vskip .1in
\noindent
{\bf {Step 2}} $-$ Pulling back to a compact   set.
 Let

$$g_n \in
\pi_1(M) \ \ \ {\rm with} \ \ \ g_n(l_{m_n}) = l_0, \ \ \ \  {\rm where} \ \ l_0 
\ \ \ {\rm is \ a \ fixed \
leaf}.$$

\noindent
Let $v_0$ be the periodic point in $l_0$.
%Let $C$ be a compact interval in $\partial \oo - \{ x \}$. 
By the convergence group theorem, up to another
subsequence assume that the sequence $(g_n)$ has a source and sink for the
action on $\mathcal R$. Similarly there are source and sink sets for the 
sequence $(g_n)$ acting on $\partial \oo$. Each of these is an equivalence
class of $\sim$, possibly the same class. 
%We will adjust the original
%sequence $(g_n)$ to obtain a sequence that proves the conical limit point
%behavior of $p_0 = \eta(x)$.

We define
line leaves $l'_i$ of $l_i$ as follows.
\ 1) The union of the two ideal points of $l'_i$ 
separates $x$ from $\partial l_x$ in $\partial \oo$;
\ 2) The ideal points of $l'_i$ are the ideal points of $l_i$ closest to
$x$ satisfying property 1). 
The $l'_i$ are uniquely defined under these properties.
Let $A_i$ be the closed interval of $\partial \oo$
bounded by the ideal points of these line leaves $l'_i$ and not containing $x$. 
Clearly $\cup_{i \in {\bf N}} \ A_i = \ \partial \oo - \{ x \}$.
The $l'_i$ form a nested sequence
of line leaves converging to $x$ in $\oo \cup \partial \oo$.

\vskip .1in
\noindent
{\bf {Claim 1}} $-$ The point $x$ is the source for the action of the sequence
$(g_n)$ acting on $\partial \oo$.

This implies that $p = \eta(x)$ is the source for the sequence $(g_n)$ acting on
$\mathcal R$, since $\ee(x) = \{ x \}$.
We first
prove:

\vskip .1in
\noindent
{\bf {Claim 2}} $-$ For any  fixed $i$, $(g_n(A_i))$ shrinks to a point, that is, 
$(diam \ g_n(A_i))$ converges to zero as $n \rightarrow \infty$.

We prove Claim 2. If the claim is not true then
 this is not true for some $i_0$ and a subsequence
of $(g_n)$ $-$ which we assume here is the original sequence. Then for any $i > i_0$
the set $g_n(A_i)$ also does not shrink to a point when $n \rightarrow \infty$, 
because $A_{i_0} \subset A_i$.
Using a diagonal process of subsequences in $n$, we can
assume that for any $i$, the sequence $(g_n(A_i))$
converges as $n \rightarrow \infty$. The limit is an interval $(a_i,b_i)$
in $\partial \oo$ bounded by points $a_i, b_i$.
The limit cannot be a point by assumption when $i > i_0$.
In addition it cannot be  the whole of $\partial \oo$ minus a point,
because for each $n$, one has $g_n(l_{m_n}) = l_0$ and 
$m_n \rightarrow \infty$ when $n \rightarrow \infty$.
Therefore for each $i > i_0$, $a_i, b_i$ are distinct from each other.
Furthermore notice that there is a monotonicity involved, if $j > i$
then $(a_j,b_j) \supset (a_i,b_i)$.
The ideal points of $l'_i$ for different values of $i$ can be equivalent under $\sim$ for only
finitely many values of  $i$. Increasing $i_0$ if necessary, we may assume that
no such ideal point is in the source set for the action of the
sequence $(g_n)$ on $\partial \oo$. Therefore for all $i > i_0$ the
sequences $(g_n(\partial l'_i))$ converge to points in the sink
set for the action of $(g_n)$ on $\partial \oo$. There are finitely
many points in this set; for each $j$ there are finitely many $i$ for which
the ideal points of $l_i$ are equivalent to the ideal points of $l_j$;
finally there is the monotonicity property above. This means that
there is $i_1$ so that if $i > i_1$ then $a_i = a_{i_1}, \ b_i = b_{i_1}$.
Then for each such $i$, the sequence 
$(g_n(l'_i)), n \in {\bf N}$ converges to collections of stable and unstable leaves
which form a path from the point $a_{i_1}$ to $b_{i_1}$. The paths  are
alternatively stable/unstable with $i$.
This is a contradiction to the second property stated in Proposition
\ref{noident}.
%, so both sets of endpoints cannot be the same. Also
%the bounded condition on perfect fits means that after a bounded number of leaves,
%the ideal points converge to inequivalent point. This contradicts the source/sink
%behavior of the sequence $(g_n)$. 
This proves claim (2). 

Claim 1 follows immediately from Claim 2 and the fact that $(l_i), i \in {\bf N}$ forms a
master sequence for $x$.
 It now follows that

$$\lim_{n \rightarrow \infty} g_n(x) \ \ = \ \ z_1, \ \ \ {\rm and} \ \ \ 
\lim_{n \rightarrow \infty} g_n(c) \ \ = \ \ w_0$$

\noindent
for every $c$ in $\partial \oo$ with $c \not = x$. 
Suppose first that 
$z_1$ is equivalent to the ideal points of $l_0$ and 
$w_0$ is not equivalent to the ideal points of $l_0$.
Then we are done.
What we mean is that if this is true, it proves the conical limit point property
for $p = \eta(x)$ with the sequence $(g_n)$ in question. This will be phrased 
like this throughout the proof of this theorem. 

Suppose this is not the case.

Recall that $v_0$ is the periodic orbit in $l_0$.

\vskip .1in
\noindent
{\bf {Push off method}} $-$ This method keeps the property that $g_n(l_{m_n}) = l_0$ and pushes
the limit of $(g_n(x))$ away from the equivalence classes
$\ee(\partial \oos(v_0))$ and
$\ee(\partial \oou(v_0))$.

Let $f_0$
be one of the generators of the stabilizer of $v_0$ and each of  its prongs. There is a sequence
$(k_n)$ so that 
$(f^{k_n}_0 g_n(x))$ converges to a point $z_0$ that is 
not in the equivalence class of  $\partial
\oos(v_0)$ or  $\partial \oou(v_0)$ under $\sim$. 
Up to a subsequence assume that $(f^{k_n}_0 g_n)$ has a source and sink set in $\partial \oo$.
Then as proved in the arguments of Claim 2, 
$(f^{k_n}_0 g_n(C))$ converges to a point $w$. If $w \sim \partial \oos(v_0)$
or $w \sim \partial \oou(v_0)$, again
we are done because $z_0$ is not equivalent to $\partial \oos(p_0)$
or to $\partial \oou(p_0)$. 
This is the Push off method.

\vskip .05 in
Notice that

$$f^{k_n}_0 g_n(l_{i_n}) \ \ = \ \ g_n(l_{i_n}),$$

\noindent
so now we can rename $g_n$ to be $f^{k_n}_0 g_n$.

%$z_0$ and $w$ are distinct.
%If $z_0 \not \sim w$ we are done. 
We have to deal with the case $z_0 \sim w$.
Notice that $z_0, w$ are distinct.
We will adjust the sequence $(g_n)$ as needed.
% to produce another sequence which
%shows that $x$ is a conic limit point.

\vskip .05in
\noindent
{\bf {Intermediate set up}} $-$ At this point we only have to deal 
with the case that $\lim_{n \rightarrow \infty}
g_n(x) = z_0, \ \lim_{n \rightarrow \infty} g_n(C) = w$ and $z_0 \sim w$.
In addition $g_n(l_{i_n}) = l_0$ is periodic with periodic point $v_0$ and
$z_0 \not \sim \partial \oos(v_0)$, \ $z_0 \not \sim \partial \oou(v_0)$.
Furthermore $l_0$ separates $g_n(C)$ from $z_0$.

\vskip .1in
\noindent
{\bf {Step 3}} $-$ For each fixed $i \geq i_0$, then for big enough $n$, $g_n(A_i)$ does
not contain $w$.

Suppose this is not true. Then there is a {\underline {fixed}} $j$ so that
for a subsequence $(n_k), k \in {\bf N}$, the set $g_{n_k}(A_j)$ contains
$w$. Since $w \sim z_0$ and $z_0 \not = w$, there is a leaf $l$ of $\oos$ or $\oou$
with ideal point $w$. 

For each $i \geq j$, then $w \in g_{n_k}(A_i)$. Suppose $k$ big enough, that is,
$k \geq k(i)$ depending on $i$. Since $n_k$ is very
big and $g_n(A_i) \rightarrow w$ as $n \rightarrow \infty$ for any $i \geq i_0$, then
we have 

$$g_{n_k}(\partial l'_i) \ \ \ \ {\rm separates \ the \ ideal \ points \ of}
\ \ l  \ \ {\rm in} \ \ \partial \oo.$$

\noindent
Apply $g^{-1}_{n_k}$. Notice that $w \in g_{n_k}(A_j)$ for the fixed $j$.
Hence $r_i = g^{-1}_{n_{k(i)}} (l)$ intersects both $l'_j$ which is fixed
and $l'_i$. 
For $i$ big, 
$l'_j$ and $l_i$ do not share ideal points. Therefore the
sequence $(r_i), i \in {\bf N}, i > j$ cannot escape compact sets  in $\oo$
and has a convergent subsequence to a
collection of leaves non separated from each other. As $\partial l_i$ shrinks to
$x$ when $i \rightarrow \infty$, then by the escape
lemma, one of the limits leaves has to have ideal
point $x$. This is a contradiction to the hypothesis in the main case.
This proves Step 3.

\vskip .1in
We now consider the unique minimal chain  $\mathcal T = \{ e_1, ..., e_{k_0} \}$ 
from $w$ to 
$z_0$  so that
consecutive leaves share an ideal point and $w$  is an
ideal point of $e_1$ and $z_0$ and ideal point of $e_{k_0}$.
We consider the curves
$e_k$ as slice leaves in leaves of $\oos$ or $\oou$.

\vskip .1in
\noindent
{\bf {Step 4}} $-$ We may assume that every leaf in the minimal
chain   ${\mathcal T}  =  \{ e_k, 1 \leq k \leq k_0 \}$ is separating
for this collection: $e_k$ separates $e_{k'}$ from $e_{k"}$ if
$k' < k < k"$.

Suppose this is not true. Then there is a fixed $k$ so that $e_k$ does not
separate $e_{k-1}$ from $e_{k+1}$.

There is a component $U$ of $\oo - e_k$ which contains both 
$e_{k-1}$ and $e_{k+1}$. We will assume without loss of generality that
$e_k$ is a stable leaf. 
Suppose first that $e_k$ does not make a perfect fit with say $e_{k-1}$.
Since $e_k$ and $e_{k_1}$ share an ideal point in $\partial \oo$, there
is at least one other leaf of $\oos$ or $\oou$ with this ideal point and
which separates $e_k$ from $e_{k-1}$. It follows that $e_k$ 
is non separated from some leaf in $\oos$, and in particular $e_k$
is periodic.
Suppose now that $e_k$ makes a perfect fit with both $e_{k-1}$ and
$e_{k+1}$. Since $e_k$ does not separate these other two leaves
the three leaves $e_k, e_{k-1}, e_{k+1}$ form a double
perfect fit. By Proposition \ref{doublefit} all three leaves are periodic.
It follows that in either case $e_k$ is in  the union of the boundary of 
two adjacent lozenges $C_1$ and $C_2$, see fig. \ref{qg13}.
It may be that $e_k$ is singular .
Then  $e_k$ does not have a prong contained in
$U$. 
%The case when there is a prong contained in $U$ is similar and easier to deal
%with.

\begin{figure}
\centeredepsfbox{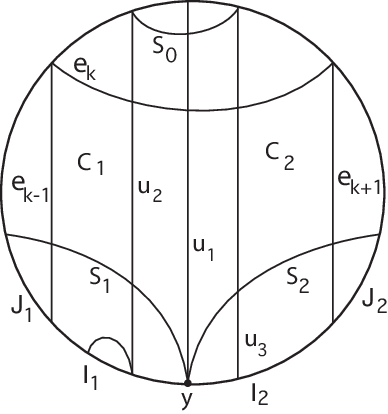}
\caption{
The situation where $e_k$ does not separate $e_{k-1}$ from
$e_{k+1}$.
The figure illustrates the
 pair of adjacent lozenges barrier method. The adjacent lozenges
here are $C_1, C_2$. The leaves $S_1, S_2$ are non separated from each
other and contain sides in the lozenges $S_1, S_2$ respectively.}
\label{qg13}
\end{figure}

Up to switching $C_1, C_2$ we
have the following properties: either $e_{k-1}$ has a half leaf in the 
boundary of $C_1$ or there is a leaf containing a side of $C_1$ and
separating $C_1$ from $e_{k-1}$. Similarly for
 $e_{k+1}$ and a half leaf in the boundary of $C_2$.
There are two stable leaves $S_1, S_2$ distinct from $e_k$ 
so that $S_1$ has a half leaf
in the boundary of $C_1$ and similarly for $S_2$. There is an unstable leaf
$u_1$ which has a half leaf in the boundary of both $C_1$ and $C_2$, see
fig. \ref{qg13}.

To prove Step 4 we employ a method that will be used many times 
in the proof of the theorem.  The method is called the 
pair of adjacent lozenges barrier method.
%Roughly the method says that after adjusting the sequence $(g_n)$ the pair
%of adjacent lozenges $C_1, C_2$ creates a barrier to a path from
%the limit of $(g_n(x))$ and the limit of $(g_n(C))$.

\vskip .1in
\noindent
{\bf {The pair of adjacent lozenges barrier method}}
$-$ This method  uses the pair of 
%chain $\mathcal E$ with associated two
adjacent lozenges $C_1, C_2$, and it produces another sequence
$(g'_n)$ in $\pi_1(M)$ which shows the conical limit point property for $p = \eta(x)$.
This will be done by showing that the limits of $(g'_n(x))$ and $(g'_n(C))$
cannot be equivalent, because any hypothetical chain connecting
them cannot ``cross the barrier" of the adjacent lozenges $C_1, C_2$.
\vskip .08in

Let $y$ in $\partial \oo$
be the ideal  point of $u_1$ which is in the boundary of both $C_1$ and $C_2$.
For simplicity of exposition we will assume that all leaves 
$e_{k-1}, e_{k+1}, S_1, S_2$ are non singular. 
We will also assume that $e_{k-1}$ and $e_{k+1}$ have half leaves
in the boundary of $C_1, C_2$ respectively. The same proof holds in
general. Let $Z_i, i = 1, 2$ be the intervals in $\partial \oo - \partial S_i$ not
containing any point of $\partial e_k$. Let $I_i$ be the open interval of $Z_i$
with one endpoint $y$ and the another in either $\partial e_{k-1}$ or
$\partial e_{k+1}$. Let $J_i$ be the interior of $Z_i - I_i$. 
The path $\mathcal T$ goes from $w$ to $z_0$ as $k$ increases.
Then for any $j$
and for any $n$ big:

$$g_n(A_j) \subset J_1 \ \ {\rm or} \ \ I_1 \ \ \  \ {\rm and} \ \ \ \ 
g_n(x) \in I_2 \ \ {\rm or} \ \ J_2.$$

\noindent
This is because $g_n(A_j) \subset Z_1$ for $n$ big, $g_n(x) \in Z_2$ for $n$ big
and $g_n(x)$ is not an ideal point of any leaf and by Step 3 the ideal point
of $e_{k-1}$ is not in $g_n(A_j)$ for $n$ big.
Here we let $f$ be a generator of the stabilizer of $C_1$ and $C_2$. Post composing
$g_n$ with powers $f^{i_n}$ of $f$, we may assume  that the sequence $(f^{i_n} g_n(x))$ converges to 
$z_2 \not \sim z_0$ and in addition $z_2 \not \sim y$. This is saying that
$z_2  \not \in \ee(y) \cup \ee(z_0)$, which is possible by the Push off method.
In addition the limit of $(f^{i_n} g_n(x))$ is in $J_2$ or $I_2$.
As before we can assume that the sequence
$(f^{i_n} g_n)$ has a source and sink and therefore  $(f^{i_n} g^n(\partial \oo - \{ x \}))$ 
converges locally uniformly to a point $w_1$ which is in 
$\overline Z_1$.  
%Then $z_2$ is in $A_2$.
If $w_1 \sim z_0$ or $w \sim y$ then we are done,
 because  then $\lim f^{i_n} g_n(x) = z_2$,
$\lim f^{i_n} g_n(t) = w$ for any $t \not = x$;  \ and
$z_2 \not \sim z_0$, $z_2 \not \sim y$, while 
either $w_1 \sim z_0$ or $w_1 \sim y$. This would prove the
conical limit point property for $p = \eta(x)$.

Hence we can assume that $w_1 \not \sim z_0$.
The goal is to show that $w_1 \not \sim z_2$. This will prove the conical
limit point property for $p = \eta(x)$ and finish the proof of Step 4.

\vskip .1in
 Suppose first that $w_1$ is in $I_1$. 
There has to be a chain $\mathcal V$  of slice leaves, consecutive ones making perfect fits
or same ideal points so that this chain connects  $w_1$ to $z_2$.
We refer to fig. \ref{qg13}.

%\begin{figure}
%\centeredepsfbox{qg22.eps}
%\caption{Trying to go from $I_2$ to $I_1$.}
%\label{skipping1}
%\end{figure}

Since $w_1 \sim z_2$ and $z_2 \not \sim y$ and $z_2 \not \sim z_0$, then the ideal
points of the chain $\mathcal V$
 cannot go through $y$. Hence the chain has to intersect $S_1$
transversely. This intersection is contained in the unstable leaf $u_2$ 
which is part of the chain
$\mathcal V$. Since $C_1$ is a lozenge then
$u_2$ intersects $e_k$ also, see fig. \ref{qg13}. 
Some subsequent leaf in the chain $\mathcal V$ has to be stable and has to intersect the unstable leaf
$u_1$ - otherwise the chain will not be able to go to the other component
of  $\oo - u_1$ which contains $z_2$ in its ideal boundary. Let this stable leaf
in the chain  be denoted by 
$S_0$. 
If for example $z_2$ is in $J_2$ the only possibility is that the chain 
$\mathcal V$ has to first have an ideal point in $I_2$ and then cross
$S_2$ to have an ideal point in $J_2$. So in any case $\mathcal V$
has an ideal point in $I_2$. So the next leaf in $\mathcal V$
has to be unstable, call it $u_3$. In addition $u_3$ has to intersect $e_k$ (and
hence $S_2$). The construction implies that the stable leaf containing $S_0$ does
not have  a prong in the component 
of $\oo - S_0$ containing $e_k$. Again by Proposition \ref{doublefit}, 
$u_2, u_3$ and $S_0$ are periodic and on the boundary of two adjacent lozenges
$C_3$ and $C_4$.

This is an impossible situation and that is the barrier method. Here is why:
Let $u_4$ be the unstable leaf which has a half leaf in the boundary of both
$C_3$ and $C_4$. If $u_4$ intersects $C_2$ then $u_4$ cannot make a perfect 
fit with any stable leaf $l^*$ intersecting $u_2$ $-$ because of the adjacent lozenges $C_1$ and
$C_2$. We explain this. Since $l^*$ makes a perfect fit with $u_4$ and
$u_4$ has ideal point in $I_2$ then $l^*$ is 
contained in the component of $\oo - S_2$ limiting on $x$. 
Since $S_2$ separates this component from $u_2$ then $l^*$ cannot
intersect $u_2$.
 This is a contradiction. 
If on the other hand $u_4$ intersects $C_1$, then  $u_4$ cannot make a perfect fit with a 
stable leaf $l^*$
intersecting $u_3$, also contradiction.
If $u_4$ and $u_1$ are in the same unstable leaf, then the periodic orbits
are the same and $S_0 = e_k$, also leading to a contradiction.
 We conclude that this case cannot occur.

The second  possibility here  is that $w_1$ is in $J_1$. By a similar argument, the path
from $w_1$ to $z_2$ has a leaf $u_2$ intersecting $S_1$. 
If this intersection is in the closure of $C_1$ then we apply the proof of the first situation.
But here it may be that
$u_2$ does not intersect $C_1$ $-$ that is, $u_2$ is contained in the component
of $\oo - e_{k-1}$ disjoint from $C_1$. 
 If this happens then the next leaf in the path $\mathcal V$  is stable ($S'$) and has
to intersect both lozenges $C_1$ and $C_2$, as well as the
leaf $e_{k+1}$. The next leaf ($u_3$) in the path
$\mathcal V$ has to be unstable and
$S'$ does not separate $u_2$ from $u_3$. Then as in the first possibility $S', u_2$ and
$u_3$ have half leaves in the boundary of the union of 2 adjacent lozenges $C_3,
C_4$. An argument exactly as in the first possibility shows this is not possible.

These arguments show that $w_1 \not \sim z_2$ and hence in this case $p = \eta(x)$ is a
conic limit point. 
%The reader can quickly verify that the same holds 
%if $\tau_{k-1}$ does not contain a side of $C_1$. In that case it is even harder for the 
%chain $\mathcal V$ to go from (the corresponding points)  $w_1$ to $z_2$. 

Therefore we may assume from now on that Step 4 holds.

\vskip .1in
\noindent
{\bf {Remark}} $-$ The barrier method uses that two pairs of adjacent lozenges
$C_1, C_2$ and $D_1, D_2$ cannot intersect in certain ways as disallowed
in the proof of Step 4. However it is not true that they cannot always
intersect: it could be that $D_1$ intersects both $C_1$ and $C_2$ but
$D_2$ does not intersects either of them. We will have to rule out this
possibility in future uses of the barrier method.

\vskip .1in
\noindent
{\bf {Step 4.a}} $-$ In the same way we may assume  that $e_k$ makes a perfect fit
with $e_{k+1}$ for every $k$. 

This means that there is no leaf $e$ sharing an ideal point with 
 both $e_k$ and $e_{k+1}$ and separating them. If that were the case, 
a proof entirely analogous to Step 4 would show that $p = \eta(x)$ is a conical limit point.
It is even harder for the corresponding chain $\mathcal V$ to go from
$w_1$ to $z_2$.

\vskip .1in
\noindent
{\bf {Claim 3}} $-$
The chain $\mathcal T$ from $w$ to $z_0$ has to have length at least $2$. 

Roughly this 
is because for
$n$ big, $g_n(x)$ is close to $z_0$ and $g_n(C)$ is close to $w$.  
If $\mathcal T$ has length one then  $\mathcal T$ is a slice in a single leaf
of $\oos$ or $\oou$.
By the pushoff method we know that $z_0$ is not an ideal
point of $\oos(v_0)$ or $\oou(v_0)$. So if $\mathcal T$ contains
either of these leaves, then $\mathcal T$ will have length at least two.
Suppose then that $\oos(v_0)$ and $\oou(v_0)$ are not part of the
chain $\mathcal T$.
Recall that $l_0 = g_n(l_{m_n})$. 
If $l_{m_n}$ is chosen according to possibility
 1) of Step 1, then  the chain $\mathcal T$ has to cross at
least 2 prongs of $v_0$: at least one stable and one unstable
prong of $v_0$. This implies that the chain $\mathcal T$ cannot have length
one, because a single leaf of $\oos$ or $\oou$ could not
cross both of these leaves. 
If on the other hand $l_{m_n}$ is chosen according to possibility 2 of Step 1, 
then the chain $\mathcal T$ has to cross a pair of leaves forming a perfect
fit, one of which is $l_0$ and the other 
is $g_n(s)$ $-$ where $s$ is the leaf described in
possibility 2 of Step 1. This proves claim 3.

\vskip .1in
The intermediate setup is that $z_0 \not \sim \partial l_0$.
Therefore  $l_0$ cannot be part of the chain $\mathcal T$,
but $\mathcal T$ has to cross $l_0$.
 Let $z$ be the first ideal point of the chain $\mathcal T$ attained 
after crossing
the leaf $l_0$.
% 2 prongs of the singular point $p_0$
% or 2 leaves making a perfect fit coming from a non Hausdorff pair.
Let $\mathcal T_0$ be the subpath of $\mathcal T$ from $w$ to $z$.

The proof of Claim 3 shows that $\mathcal T_0$ has length at least $2$.

\vskip .1in
\noindent
{\bf {Claim 4}} $-$ We may assume that 
the line leaf $l^*_0$ of $l_0$ which separates $g_n(x)$
from $g_n(C)$ intersects some leaf $e_{k_1}$ of $\mathcal T_0$ transversely.

This is stronger than $l_0$ intersects a leaf of $\mathcal T_0$ transversely.
Suppose that the claim is not true. If $l_0$ shares an ideal point with  a leaf
of $\mathcal T_0$, then $\partial l_0 \sim w$. But $(g_n(C))$ converges 
to $w$ and $(g_n(x))$ converges to a point $z_0 \not \sim \partial l_0$.
This proves the conical limit point property for $p = \eta(x)$.

So we may assume that $l_0$ does not share an ideal point with a leaf 
in $\mathcal T$. Therefore the union of the leaves in $\mathcal T$ is
contained 
in a single complementary component $V$ of $l^*_0$ in $\oo$. 
Notice that $l^*_0$ separates $g_n(C)$ from $g_n(x)$, so
it now follows that $\mathcal T$ cannot be contained $V$.
This complementary component $V$ does not limit on 
$z_0 = \lim_{n \rightarrow \infty} g_n(x)$ because $l^*_0$ separates
$\mathcal T$ from $g_n(x)$ and $z_0 \not \in \partial l_0$.
This contradicts the fact that the chain $\mathcal T$ connects
$w$ to $z_0$.

This proves Claim 4.

\vskip .1in
Let $k_1$ so that $l^*_0$ intersects $e_{k_1}$ transversely.
Obviously $k_1 \leq k_0$, the length of $\mathcal T$.
Here $k_1$ is the length of the chain $\mathcal T_0$, so we
know that $k_1 \geq 2$.
For simplicity for the rest of the proof let

$$d \ = \ e_{k_1}, \ \ \ \ \ c \ = \ e_{k_1 - 1}$$.

\noindent
Since $l_0$ intersects $d$ transversely, then $d$ is an unstable
leaf and $c$ is a stable leaf.

Recall that $l_0 = g_n(l_{m_n})$ for a subsequence $(m_n)$ in ${\bf N}$.
Consider the previous step in $g_n(\mathcal P)$, that is, the leaf
$g_n(l_{m_n -1})$, which we denote here by $H_n$. 
We will do this operation many times in the proof of theorem
\ref{uniform}. We stress that the path $g_n(\mathcal P)$ is
standard {\underline {from}} $g_n(l_1)$ {\underline {to}} $g_n(x)$.
In other words $g_n(l_{m_n-1})$ determines $g(l_{m_n})$ but not
the other way round. So whenever we consider a previous leaf
in $g_n(\mathcal P)$ such as $g_n(l_{m_n-1})$ we will
discuss the 3 options of Step 1 to obtain $g_n(l_{m_n})$ from $g_n(l_{m_n-1})$.

Let $U$ be the component of $\oo - c$ which limits in $z$.

\vskip .1in
\noindent
{\bf {Claim 5}} $-$ We can assume that the leaves $H_n$ intersect $c$.

Suppose that this is not true.
Recall that $H_n = g_n(l_{m_n-1})$ either intersects $l_0$, or
makes a perfect fit with $l_0$ or intersects a leaf non separated
from $l_0$. Therefore $H_n$ intersects $U$ as $l_0$ is contained
in $U$.

\begin{figure}
\centeredepsfbox{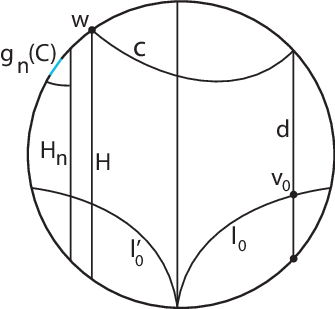}
\caption{The leaves $H, c, d$ are in the boundary of adjacent lozenges
$C_1, C_2$. The position of the leaves $H_n$ forces the next leaf in
the path $g_n(\mathcal P)$ to be $l'_0$ and $l_0$.}
\label{qg14}
\end{figure}

Since $H_n$ does not intersect $c$ then it is contained in $U$.
%Suppose first that up to subsequence every $H_n$ is contained in $U$.
Recall that $H_n$ has a line leaf separating $g_n(x)$ from $g_n(C)$.
Since the sequence $(g_n(C))$ converges to $w$, it now follows that
$k_1 = 2$, $c = e_1, \ d = e_2$ and $w$ is an ideal point of $c$.
In addition $H_n$ has an ideal point $y_n$ so that the sequence
$(y_n)$ converges to $w$. Since $H_n$ satisfies one of the
3 conditions of the previous paragraph, it follows that
$(H_n)$ converges to a leaf $H$ making a perfect fit with
$c$ and so that $c$ does not separate $H$ from $d$. This is because
all $H_n$ intersect a  leaf which does not share an
ideal point with $c$. Hence the sequence $(H_n)$ cannot escape 
compact sets in $\oo$.
Then $H, c, d$ form a double perfect fit, see fig. \ref{qg14}.
By Proposition \ref{doublefit}, there are adjacent lozenges
$C_1, C_2$ both with a side in $c$ and other sides respectively
in $H$ and $d$. If $l_0 = g_n(l_{m_n})$ is obtained from the previous
step by option 1 of Step 1, then the following happens.
The leaf $l_0$ is singular and intersects both $H$ and $d$,
 and $l_0$ has a singularity
between $H \cap l_0$ and $d \cap l_0$. This was disallowed in the
proof of Proposition \ref{doublefit}. Therefore $l_0$ is obtained
using Option 2 of Step 1 and there is $l'_0 = g_n(s_{m_n})$
non separated from $l_0$ and intersecting $H_n$. Analysing
the interaction of this with the two adjacent lozenges $C_1, C_2$
one sees that the only possibility is that $l_0$ and $l'_0$ contain
sides of $C_2$ and $C_1$ respectively, see fig. \ref{qg14}.
In this situation $(g_n(C))$ converges to $w$ in $\partial H$.
But 
$\partial H \sim \partial c \sim \partial d$.
Since $d$ has a side in the boundary of the lozenge $C_2$ as does $l_0$ and
$l_0$ is periodic with periodic orbit $v_0$, it follows that
$l_0 = \oos(v_0), \ d = \oou(v_0)$. In addition $v_0$ is a corner of 
$C_2$, see fig. \ref{qg14}. This implies that $w \sim \partial \oou(v_0)$.
Since $z_0 \not \sim \partial \oou(v_0)$ by the push off method, this
proves the conical limit point property for $p = \eta(x)$.

This proves Claim 5.
Therefore from now on, we can assume that $H_n$ intersects $c$ for all $n$.

\vskip .08in
Notice the following fact. It may be that $H_n$ intersects $l_0$ if
$l_0 = g_n(l_{m_n})$ is chosen according to Option 1 of Step 1, but
in any case  $H_n$ does not intersect the line leaf $l^*_0$ of $l_0$. In particular
no subsequence of $(H_n)$ can converge to the leaf $d$.

In Step 3 we proved that we can assume $w \not \in g_n(C)$ for $n$ sufficiently big. 

Now there are two options depending on whether the subchain $\mathcal T_0$
has length $2$ or higher.

\vskip .1in
\noindent
{\bf {Case A}} $-$ The chain $\mathcal T_0$ has length $2$.

Here $\mathcal T_0 = \{ c, d \}$. 
Here $w$ 
is an ideal point of $c = e_1$ and $z$ is an ideal point of $d = e_2$.

Now we will consider the preceeding leaves in $g_n(\mathcal P)$, that is,
the leaves $g_n(l_{m_n-2})$.

\vskip .1in
\noindent
{\bf {Claim 6}}  $-$  The sequence $(g_n(l_{m_n-2}))$ escapes compact sets and
therefore converges to $w$.

We have the following facts.  1) The leaf  $g_n(l_{m_n-2})$ 
has a line leaf separating $g_n(C)$ from $g_n(x)$,
\ 2)  
$(g_n(C))$ converges to $w$ and does not contain $w$ for $n$ big,
\ 3) $w$ is an ideal point of $c$ and $g_n(l_{m_n-2})$ is disjoint from $c$ 
$-$ since $g_n(l_{m_n-1})$ intersects $c$ by Claim 5..
Therefore
$g_n(l_{m_n-2})$ has an ideal point, call it $q_n$, so that $(q_n)$ converges
to $w$. 

%In case A.1 the leaves $g_n(l_{m_n-2})$ are not contained in $U$.
Suppose that the sequence
$(g_n(l_{m_n-2}))$ does not escape in $\oo$. Then up to subsequence
it converges to a stable leaf $t$. 

Suppose first that $t = c$. 
Since no subsequence of $(H_n)$ can converge to $d$, then this
can only happen if $g_n(l_{m_n-2})$ is not contained in $U$.
Here we initially deal with the case that $c$ is singular.
Then there is a singular orbit $v_2$ in $c$. As $(g_n(l_{m_n-2}))$
converges to $c$ it now follows that for all $n$ big
$H_n = \oou(v_2)$, and in addition $H_n = g_n(l_{m_n-1})$ is obtained
by option 1 in Step 1.
Also $v_2$ is a corner of a lozenge $C^*$
which has one side in a half leaf of $d$ and a corner $v_3$ that 
is the periodic orbit in $d$.
It follows that this lozenge has a stable side in $\oos(v_3)$.
In particular the construction of the standard path $g_n(\mathcal P)$
implies that the next leaf of $g_n(\mathcal P)$ has to be 
$\oos(v_3)$ as this separates $g_n(x)$ from $c$ (and hence from $g_n(C)$).
This is the leaf $l_0 = g_n(l_{m_n}) = \oos(v_3)$.
Finally $\oos(v_3)$ makes a perfect fit with $H_n = g_n(l_{m_n-1})$.
This means that $l_0 = l_{m_n}$ is chosen according to option 3
in Step 1. But we specifically picked out the subsequence $(m_n)$
so that in each step $m_n -1$ either options 1 or 2 is used to produce the
next leaf. We conclude that this cannot happen and hence $c$ is 
non singular in this setting.

We now have that $(g_n(l_{m_n-2}))$ converges to the full leaf 
$c$. 
Since $c$
makes a perfect fit with $d$ this would imply that $(g_n(l_{m_n-1}))$ converges
to the leaf $d$. This was disallowed just before the statement of Case A.
%It follows that  for $n$ big $g_n(l_j)$ intersects the line
%leaf of $l_0$ separating $g_n(C)$ from $g_n(x)$. But then $l_0$
%would not be in $g_n(\mathcal P)$ contradiction.

We conclude that $t \not = c$. But $t, c$ share an ideal point.
Hence there is a leaf $t'$ (possibly $t' = t$) so that 
$t'$ is non separated from $c$ and $c, t, t'$ have an
ideal point in common.  In addition
$t$ is the limit of leaves intersecting $H_n$, or leaves making a perfect fit
with $H_n$ or leaves non separated from $H_n$. Since $H_n$ intersects
$c$ this is impossible.

Therefore the sequence  $(g_n(l_{m_2-2}))$ escapes compact sets in $\oo$.
In particular $(g_n(l_{m_n-2}))$ has to converge to 
$w$. 
This proves Claim 6.

\vskip .1in
\noindent
{\bf {End of the analysis of Case A}}

Since the sequence $(g_n(l_{m_n-2}))$ converges to $w$, 
we claim that this implies that $H_n$ has an ideal point $q_n$
so that $(q_n)$ converges to $w$.
This is clear if $H_n = g_n(l_{m_n-1})$ is produced by either
option 1 or option 3 of Step 1. Suppose that 
$g_n(l_{m_n-1})$ is produced by
option 2, so there are two leaves $g_n(u_{m_n-1})$ and $g_n(l_{m_n-1})$
which are non separated from each other.
The first sequence has ideal points converging to $w$ as they
intersect $g_n(l_{m_n-2})$.
This implies that they cannot be eventually constant and so
one of them has to escape compact sets in $\oo$.
Since $H_n$ intersects $c$ then the sequence $(H_n = g_n(l_{m_n-1}))$
does not escape compact sets in $\oo$.

This implies that the sequence
$(g_n(u_{m_n-1}))$ escapes compact
sets in $\oo$, for otherwise we obtain a contradiction as in the
end of the proof of claim 6.
This now implies that $H_n$ has an ideal point which converges
to $w$ as $n \rightarrow \infty$.
This in turn implies that the sequence $(H_n \cap c)$
converges to $w$.

We have so far proved that
$(H_n)$ does not escape compact sets in $\oo$, and 
$(H_n \cap c)$ converges to $w$. Then the arguments in the proof
of Claim 5 imply the conical limit point property for
$p = \eta(x)$.

This finishes the analysis of Case A.

\vskip .2in
\noindent
{\bf {Case B}} $-$ The chain $\mathcal T_0$ from $w$ to $z$ has length $\geq 3$.

Why is this different from the situation in Case A? The crucial property
in Case A was that the sequence $(H_n \cap c)$ escapes in $c$ (recall that
$c = e_1$ in that case). This was obtained because $w$ is an ideal 
point of $e_1 = c$ and $(g_n(C))$ converges to $w$. In case B the point
$w$ is not an ideal point of $c = e_{k_1-1}$. A priori the fact
that $(g_n(C))$ converges to $w$ gives no information concerning 
the sequence $(H_n \cap c)$. Conceivably $(H_n \cap c)$ is not escaping in $c$ and
perhaps $(H_n)$ is even constant. Conceivably $H_n$ could be a fixed
singular leaf with one prong intersecting $e_{k_1-1}$ and a line leaf
separating $g_n(C)$ from $g_n(x)$. A priori this structure is certainly possible.
In the same way, conceivably $(g_n(l_{m_n-2}))$ could be a constant
sequence which is a fixed singular leaf with a line leaf
separating $g_n(C)$ from $g_n(x)$ and
intersecting $H_n$ and also with a prong intersecting
$g_n(e_{k_1-3})$; and so on.
The information we have in Case B is that $(g_n(C))$ converges to $w$ which is an
ideal point of $e_1$. Therefore we need to start with the leaf $e_1$
and proceed to $e_i, i \geq 1$.

%Let the first leaf in this chain be denoted by $e$. 
As in Case A, we denote by $c$ and $d$ the last two leaves of $\mathcal T_0$.
The path of perfect fits $\mathcal T_0$ is

$$\mathcal T_0 \ \ = \ \ \{ e_1, e_2, ..., e_{k_0} \}, \ \ {\rm where} \ \ e_i, e_{i+1} \ \ 
{\rm share \ an \ ideal \ point}.$$

\noindent
In addition we proved in Steps 4 and 4.a that
$e_i$ separates $e_{i-1}$ from $e_{i+1}$ and $e_i, e_{i+1}$ make  a perfect fit.
Also  $e_{k_0} = d, \  e_{k_0 - 1} = c$.
The chain $\mathcal T$ is chosen to be minimal so that $e_{i-1}, e_{i+1}$ do 
not share an ideal point. Finally  $e_1 \not = c, d$.

\vskip .1in 
\noindent
{\bf {Claim 7}} $-$ The leaf $e_1$ is not part of the path $g_n(\mathcal P)$.

Supose on the contrary that $e_1$ is a leaf in the standard  path $g_n(\mathcal P)$ from $g_n(C)$ to $g_n(x)$. 
% Each individual leaf in the chain $\mathcal E$ makes a perfect fit with its neighbor  leaves 
%in $\mathcal E$ and
%they satisfy the separation property: $e_i$ separates $e_{i-1}$ from $e_{i+1}$ and
%$e_{i-1}, e_{i+1}$ do not share an ideal point.
Then  the following happens: if $e_1$ is part of the path $g_n(\mathcal P)$ then the 
the standard  path $g_n(\mathcal P)$
from $g_n(C)$ to $g_n(x)$ will have to follow the chain $\mathcal T$ at least
until the leaf $c$.  This is because the properties above imply that once $e_i$ is in
$g_n(\mathcal P)$, then $e_{i+1}$ is the next leaf chosen in $g_n(\mathcal P)$
under option 3 of Step 1. This works until at least $c = e_{k_0-1}$. We do not know
if $d$ is chosen because it may not separate $g_n(x)$ from the leaf
$c$ (or from $g_n(C)$).

\vskip .1in
\noindent
{\bf {Subclaim}} $-$ 
The setting with $e_1$ in the path $g_n(\mathcal P)$
 implies that $l_0$ cannot be part of the path $g_n(\mathcal P)$ which is 
contrary to the set up that $l_0$ is always a leaf in $g_n(\mathcal P)$.

To prove that $l_0$ is not part of $g_n(\mathcal P)$, suppose first that $d$ is not the last leaf of 
$\mathcal T$ (so $\mathcal T \not = \mathcal T_0$). Then $(g_n(x))$ converges
to $z_0$ which is not $z$ and $d$ separates $g_n(x)$ from $g_n(C)$. It follows
that $d$ is the next leaf in the path $g_n(\mathcal P)$ $-$ the one after $c$, again
using option 3 in Step 1.
For simplicity of notation 
let the next leaf in $\mathcal T$ after $d$  be denoted by 
$e$, that is, $e = e_{k_1+1}$. If $e$ separates $g_n(x)$ from $g_n(C)$,
then this leaf $e$  is also in $g_n(\mathcal P)$ 
obtained by option 3 of Step 1.
In addition $e$ separates $l_0$ from $g_n(x)$, and then from then on
the leaf $l_0$ could not be a leaf in $g_n(\mathcal P)$, contradiction.
If on the other
hand $e$ does not separate $g_n(x)$ from $l_0$, it follows that
$e$  is the last leaf in the chain $\mathcal T$ and $z_0$ is an ideal
point of $e$.  Since $d$ is in the path $g_n(\mathcal P)$, and $l_0$ in $g_n(\mathcal P)$
intersects $d$ then
the next leaf in the path $g_n(\mathcal P)$ is $l_0$. The ideal
poins of $l_0$ are not equivalent to $z_0$ because
$z_0 \sim z \sim \partial d$ and $l_0 \cap d \not = \emptyset$.
It follows that $(g_n(x))$ cannot converge to $z_0$, contradiction.
% is a leaf $d_n$ which 
%either intersects $d$ in a point $r_n$ converging to $z$ as $n \rightarrow \infty$
%or $d_n$ makes a perfect fit with $d$ at $z$. In either case this leaf $d_n$ cannot be $l_0$ 
%for $n$ big enough
%and $d_n$ separates $g_n(x)$ from $l_0$. Then again $l_0$ cannot be
%n $g_n(\mathcal P)$.

The other possibility is that $d$ is the last leaf in the path $\mathcal T$,
that is, $\mathcal T = \mathcal T_0$. Then the sequence $(g_n(x))$ converges
to $z$ (that is, $z_0 = z$). 
Let $u_n$ be the next leaf in the path $g_n(\mathcal P)$ after $c$. 
Suppose first that up to subsequence $u_n = d$ for every $n$.
If $d$ is a leaf in the path $g_n(\mathcal P)$, then $d$ has a line
leaf $d^*$ separating $g_n(x)$ from $g_n(C)$. 
%This line leaf also 
%separates the corresponding line leaf $l^*_0$ of $l_0$ from the leaf $c$.
Since $l^*_0$ does not share an ideal point with $d$, it follows that
$g_n(x)$ cannot converge to $z_0$, if $l_0$ is the
next leaf in $g_n(\mathcal P)$, contradiction.
%Suppose also that $d$ separates $g_n(x)$ from $c$. Then  $l_0$ cannot be in $g_n(\mathcal P)$
%as above.
%If on the other hand  $d$ does not separate $g_n(x)$ from $c$ then $d$ is
%non separated from a leaf $d'$ of $\oou$ so that 
%$d = g_n(l'_t)$, \ $d' = g_n(l_t)$ for some $t \in {\bf N}$. But then
%$d'$ separates $g_n(x)$ from $l_0$ and $l_0$ cannot be a leaf 
%in $g_n(\mathcal P)$.

Finally suppose then that $u_n \not = d$ for all $n$. 
This implies that
$d$ does not separate $g_n(x)$ from $c$.
In addition $u_n$ cannot be produced  according to Option 3 of Step 1. This also 
implies that there is not a leaf $\tau$ non separated from $d$ and
with $\tau$ separating $d$ from $g_n(x)$. Otherwise $d$ would be the next leaf
in $g_n(\mathcal P)$.
Also since 
$(g_n(x))$ converges to $z = z_0$, then the sequence $(u_n)$ converges
to $d$ $-$ whether it is produced by Option 1 or Option 2 in Step 1.
%So in particular $u_n$ intersects $l_0$ for $n$ big.
In addition, if $u_n$ does not intersect $l_0$ for $n$ sufficiently big,
then $l_0$ cannot be in $g_n(\mathcal P)$ so this implies that
$u_n$ intersects $l_0$ for all $n$ big.
Then the next leaf in the path $g_n(\mathcal P)$ will separate $l_0$ from 
$g_n(x)$.
%in $\oo \cup \partial \oo$. 
This implies that $l_0$ cannot be the next leaf
in the path $g_n(\mathcal P)$.

In any of the cases we obtain  $l_0$ is not a  leaf in $g_n(\mathcal P)$,
which is a contradiction to the setup.
This finishes the proof of the Subclaim and hence proves Claim 7.

%\vskip .1in
%Do we need the following????
%Now there are two options:
%S%uppose first that the leaf $d$ separates $g_n(x)$ from $g_n(C)$. Then $d$ would be the 
%next  leaf in the
%standard  path $g_n(\mathcal P)$
%from $g_n(C)$ to $g_n(x)$. In particular this would imply that $l_0$ is not in the path $g_n(\mathcal P)$,
%at least for $n$ big ???????
% and $l_n$ (what is the name of this leaf?????) would 
%not be in this path. If the leaf $a$ does not separate $g_n(x)$ from $g_n(C)$ then perhaps $a$ is not
%the next leaf in the canonical path. However the next leaf in this path cannot intersect $l_n$ in the position
%that allows for $l_n$ to be the next leaf in the canonical path, see figure \ref{stuck1}.
%We conclude that the canonical path from $g_n(C)$ to $g_n(x)$ does not use the leaf $e$. In the same way it
%cannot contain a leaf with ideal point $w$ as it then would be forced to go through or produce a leaf $f$
%making a perfect fit with $e$ and $e$ not separating $f$ from the rest of the canonical path $\mathcal L$.
%As seen in the proof ???? this can be dealt with .
%We conclude that $w$ is not an ideal point in the canonical path from $g_n(C)$ to $g_n(x)$. Hence this 
%canonical path has to cross $e$ transversely ina leaf $F_n$.
%???? IN this case let $U$ be the component of $\oo - e_1$ containing $e_2$.

%\begin{figure}
%\centeredepsfbox{qg27.eps}
%\caption{The structure of the canonical path forces an impossible position on the next leaf $l_n$.}
%\label{stuck1}
%\end{figure}

\vskip .1in
In case B let  $U$ be the component of $\oo - e_1$ which accumulates in $z$.

\vskip .1in
\noindent
{\bf {Case B.1}} $-$ The leaf $e_1$ separates $g_n(C)$ from $g_n(x)$.

Since $e_1$ cannot be part of the path $g_n(\mathcal P)$ and $e_1$ separates
$g_n(C)$ from $g_n(x)$,
it follows that there is a leaf $F_n$ in the path $g_n(\mathcal P)$
which intersects $e_1$ transversely.
%In $w$ is not an ideal point of $U$ then no line leaf of $F_n$ can separate 
%$g_n(C)$ from $g_n(x)$ contradiction.
Assume first that $w$ is an ideal point of $U$.
Let $W$ be  the component  of $\oo - e_1$ which
  accumulates in $w$ and is not $U$.

Let $d_n$ be the previous leaf in the path $g_n(\mathcal P)$. 
We may assume up to subsequence that  one of the following holds for all $n$.

\begin{itemize}

\item Suppose that this leaf $F_n$ of $g_n(\mathcal P)$ is obtained 
from $d_n$ by
 Option 1 of Step 1.  The leaf $F_n$ has a line leaf which separates $g_n(x)$ from $g_n(C)$.
 Let $v_n$ be the singularity in $F_n$. 
%First we show that $u_n$ is in $W$.
 Suppose first that $v_n$ is not in $W$.
Then  since $F_n$ intersects $e_1$, the singularity
 $v_n$ is in $U \cup e_1$. Since
 Option 1 of Step 1 is used to choose $F_n$ then $F_n$ intersects $d_n$.

 The first possibility is  that $v_n$ is in $U$. Then all the other prongs of $F_n$ 
which do not intersect $e_1$  are contained in $U$.
As $d_n$ is not contained in $U$ then the prong of $F_n$ not contained in $U$ is the
one intersecting $d_n$. It follows that the line leaf of $F_n$ separating $g_n(C)$
from $g_n(x)$ is contained in $U$ and in fact would separate $g_n(x)$ from $e_2$
as well, a contradiction to the property of $e_2$. Therefore $v_n$ cannot be in $U$ for
$n$ big enough.

 The second possibility is that $v_n$ is in $e_1$. Then $(v_n)$  and hence $(H_n)$ are
 constant sequences. 
In particular in this case all the leaves $e_i$ are periodic and have periodic orbits
which will be denoted by $z_i$. In addition the $\{ z_i \}, 1 \leq i \leq k_0$ are
connected by a chain of lozenges $\{ C_i, 1 \leq i < k_0 \}$ so that $C_i$ has
corners in $z_i$ and $z_{i+1}$. For the sake or argument suppose that
$e_1$ is a stable leaf, hence $e_1 = \oos(z_1)$. 
Here $d_n$  intersects 
$F_n$ (and $F_n$ is independent of $n$), while also
 separating $g_n(C)$ from $g_n(x)$. Since the sequence $(g_n(C))$ 
converges to $w$, it follows
 that $d_n$ has an ideal point $t_n$ so that the 
sequence $t_n$ converges to $w$ when $n \rightarrow \infty$.

\vskip .1in
\noindent
{\bf {Special subcase}} $-$

Suppose first that  $(d_n)$ does not  escape compact sets. Then it  converges to a leaf $d^*$ which has
 an ideal point $w$ and either intersects $F_n$ or has an ideal point distinct from $w$ which
 is equivalent to the ideal points of $F_n$. This is only possible if $d^*$ is the leaf $e_1$,
and in that case $d^*$ intersects $F_n$ for $n$ big, see fig. \ref{qg15}.
Since $e_1$ is singular, this now implies that $\oou(z_1)$ has a line leaf which separates $g_n(C)$
from $g_n(x)$. As $(d_n)$ converges to $e_1$, it now follows that $\oou(z_1)$ is in $g_n(\mathcal P)$
for any $n$ big enough, that is, $F_n = \oou(z_1)$ for $n$ big.
Then the lozenge $C_1$ implies that the next leaf in $g_n(\mathcal P)$ is
$\oos(z_2)$. Notice here that $\oou(z_2) = e_2$ is in $\mathcal T$. 
In the same way the next leaf in $g_n(\mathcal P)$ is $\oou(z_3)$, whereas
$\oos(z_3) = e_3$  is in $\mathcal T$.
All of these leaves of $g_n(\mathcal P)$ are obtained using option 3 of Step 1.
Proceeding by induction on $i$ it follows that the leaf of $g_n(\mathcal P)$
that intersects $d$ is $\oos(z_{k_0})$ (because $d$ is an unstable leaf) and
$d = \oou(z_{k_0})$. But recall that this leaf of $g_n(\mathcal P)$ is $l_0$
and we proved here it is obtained using option 3 of Step 1. But this contradicts the
setup constructing $l_0$. 
%We conclude that this cannot happen.
 It follows that this is impossible.
 
\begin{figure}
\centeredepsfbox{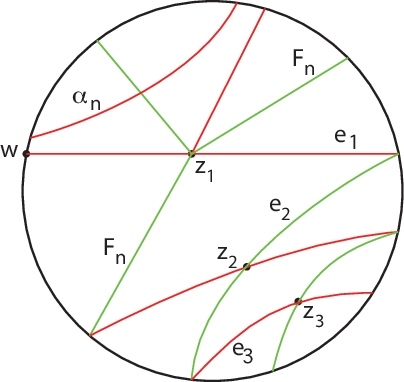}
\caption{This depicts the situation in the special subcase. Without loss of
generality assume that $e_1$ is stable. Then $e_2$ is unstable, $e_3$ is stable and
so on. Here $e_1 = \oos(z_1)$. Also $F_n = \oou(z_1)$ is in the path $g_n(\mathcal P)$
as are $\oos(z_2)$, $\oou(z_3)$ and so on. In this figure the stable leaves
are red and the unstable ones are green.}
\label{qg15}
\end{figure}

 We conclude that the only possibility here is that $v_n$ is in $W$.
As above consider the previous
 leaves $d_n$ in $g_n(\mathcal P)$.
Suppose first that the sequence $(d_n)$ does not escape compact sets
in $\oo$. As in the analysis above $(g_n(C))$ converges to $w$, which
implies that the leaves $d_n$ have ideal points converging to $w$, which implies
that $(d_n)$ converges to $e_1$. The difference here is that perhaps the leaf $e_1$
is not  singular, or more to the point  that $e_1$ is  not  periodic, so we do not a priori
have lozenges $\{ C_i \}$ as in the previous argument. 
But if $(d_n)$ converges to $e_1$ and $e_1$ is not singular, then the next leaves in
the path (the $F_n$) have to converge to $e_2$. So here the conclusion
is that $(F_n)$ converges to $e_2$.

At this point we know that either  $(d_n)$ escapes compact sets 
in $\oo$ or if not then $(F_n)$ converges to $e_2$.
 
We analyse further the possibility that $(d_n)$ escapes compact sets in $\oo$.
Since $F_n$ intersects $e_1$ for all $n$,
this shows that the sequence $(F_n \cap e_1)$ converges to $w$. 
Suppose first that  $(F_n \cap U)$ does not
escape in $\oo$ then it limits to a leaf $e_0$ making a perfect fit with $e_1$ and so that $e_1$ does not separate
$e_0$ from $e_2$. By proposition \ref{doublefit} this produces two adjacent lozenges $C_1, C_2$,
each with a 
side in a half leaf of $e_1$. 
The analysis of Case A.1 proves the conical limit point property for $p = \eta(x)$
because as in Case A.1 this setup implies that $w$ is equivalent to either
$\partial \oou(v_0)$ or $\partial \oos(v_0)$, where $v_0$ is the periodic orbit in $l_0$.
So if $(d_n)$ escapes compact sets in $\oo$ we can assume that $(F_n \cap U)$ escapes
compact sets in $\oo$. 

The final conclusion in this case is that either $(F_n \cap U)$ escapes compact
sets in $\oo$ or $(F_n)$ converges to $e_2$.

%Since $k_0 \geq 3$, there is another lozenge $C_3$ with
%a corner in $e_2$. The analysis in Cases A.1.1  shows that this situation
%shows that adjusting the sequence $(g_n)$ proves that $x$ is a conical limit point.
%For the analysis of Case A.1.1 to be applicable one needs that $e_1$ separates $g_n(x)$ from
%$g_n(C)$. The final piece of this argument is also applicable to the following two options.
%Hence in this case we can assume that $F_n \cap U$ escapes in $\oo$.

\item Suppose that  $F_n$ is obtained by option 2 of Step 1 (non separated leaves). Then 
similar arguments as in the previous case imply that
%a proof exactly as in Case A.1 for non separated leaves 
%also yields that either 
either the sequence $(F_n \cap U)$
converges to $w$ and therefore escapes compact sets in $\oo$ or $(F_n)$ converges to $e_2$.

\item Finally in the case that  $F_n$ is produced using perfect fits,  then as in the first case above 
we obtain that either 
the sequence $(F_n \cap U)$ escapes
in $\oo$ or $(F_n)$ converges to $e_2$.
\end{itemize}

\noindent
{\bf {Intermediate conclusion in Case B.1}} $-$ We can assume that either $(F_n \cap U)$ escapes in $\oo$
or that $(F_n)$ converges to $e_2$. 

\vskip .13in
\noindent
{\bf {Induction on the leaves $e_i$}}

Induction means we are going to analyse the subsequent   
leaves $\{ e_i \}$ in the path $g_n(\mathcal P)$. 
First we  show the conical limit point property for $p = \eta(x)$
 unless the sequence (in $n$) of subsequent leaves in $g_n(\mathcal P)$
converges to either $e_1$ or $e_3$. Then we iterate this process.
%The analysis is similar to the one in case A.1, but more complex because the
%chain $\mathcal T_0$ has more than two leaves.
%Here we know that $e_2$ is not the last leaf of $\mathcal T_0$.

\vskip .1in
\noindent
{\bf {Subcase 1}} $-$
We suppose first that we are in the case that $(F_n \cap U)$ escapes in $\oo$.

The next leaf
in the standard path $g_n(\mathcal P)$
will denoted by $v^1_n$. It has to intersect $e_2$ since it separates $g_n(C)$ from $g_n(x)$. 
It is also contained in $U$. 
Here  $v^1_n$ either intersects $F_n \cap U$  or $v^1_n$ is non separated from a leaf
intersecting $(F_n \cap U)$ or $v^1_n$ is non separated from $(F_n \cap U)$.
In any case $(F_n \cap U)$ converges to $w$ and it follows that
the sequence $(v^1_n)$  limits to $e_1$.

%\begin{figure}
%\centeredepsfbox{qg28.eps}
%\caption{a. The leaves $F_n$ and $F'_n$, b. The case of perfect fits.}
%\label{stuck2}
%\end{figure}

The concern is  that this sequence 
$(v^1_n)$ also limits also to another leaf $e^*_1$ making a perfect fit with $e_2$ and so that 
$e_2$ separates $e^*_1$ from $e_1$ and $e_1, e^*_1$ share an ideal point. 
Then $e_1, e^*$ are in the boundary of two adjacent lozenges $D_1, D_2$ which have a common
side in a  half leaf of $e_2$. 
%The proof here is similar to part of case A.1.
In addition since $e_2$ makes a perfect fit with $e_3$, there is also a lozenge
$D_3$ with sides in $e_2$ and $e_3$. 
Notice that $D_1$ and $D_2$ are adjacent and intersecting a common stable leaf.
Also $D_2$ and $D_3$ are adjacent and intersecting a common unstable leaf.
Here $D_1$ and $D_3$ do not intersect a common leaf.

%The stabilizer of this pair of lozenges is denoted by $f$. Using powers of $f$ we utilize
%a new sequence of elements of $\pi_1(M)$: $(f^{i_n} g_n)$ so that
%$(f^{i_n} g_n(C))$ converges to $w'$ which is not equivalent to $w$. 
%Let then $x' = \lim f^{i_n} g_n (x)$ and suppose that $x' \sim w'$. Let $\mathcal E'$ be the chain
%from $w'$ to $x'$. This chain cannot have a leaf that intersects both $D_1$ and $D_2$ for otherwise
%it will be trapped and not able to reach $x'$. Hence the chain has to cross $D_1$ and not
%$D_2$. The methods in the proof of Cases A.1 show that $(f^{i_n} g_n(F_n))$ converges
%to a leaf $H$ with ideal point $w'$ and intersecting $e_1$. This leaf makes a perfect
%fit with the first leaf $T_1$ of $\mathcal E'$. Then $H, T_1$ and the second leaf of $\mathcal E'$
%(call it $T'_2$) are in the boundary of two adjacent lozenges. As seen in case A.1 this 
%can be dealt with by the barrier method.

As explained previously 
we can use a combination of the push off method and the barrier
of adjacent lozenges method to show the conical limit point
property for $p = \eta(x)$.

We conclude that we can assume that $(v^1_n)$ limits only to
$e_1$.

\vskip .1in
\noindent
{\bf {Intermediate conclusion}} $-$ In Case B.1 with $(F_n \cap U)$ escaping
in $\oo$, we can now assume that the
sequence $(v^1_n)$ limits only to $e_1$.

Now we iterate the process. 
Let $b_i \in \partial \oo$ be the common ideal point of $e_i$ and $e_{i+1}$.
By induction we will show that
if the leaf $v^k_n$ is the $k$-th leaf after $F_n$ in $g_n(\mathcal P)$, then
$(v^k_n)$ converges to $e_k$. 
We also assume that $(v^k_n)$ has ideal points $y_n$ so 
that $(y_n)$ converges to $b_k$. Here $v^k_n$ intersects $e_{k+1}$.
% and $e_{i+1}$ separates $(v^{i+1}_n)$
%from $e_i$. 
This has been proved for $k = 1$,
so suppose it is true for $k -1$ where 
$k < k_0$.
We claim that 
the next leaf $v^k_n$ has to intersect $e_{k+1}$. 
Otherwise $e_{k+1}$ separates $g_n(x)$ from $v^k_n$
and from further leaves in $g_n(\mathcal P)$. 
This is a contradiction because $(g_n(x))$ converges to $z_0$
and $k+1 \leq k_0$.
Suppose that  $e_k$ 
is non separated from another  leaf $e^*_k$ so that $e_k, e_{k+1}$ and
 $e^*_k$ share the ideal point $b_i$ and $e_{k+1}$ separates $e_k$ from
$e^*_k$. 
Then as before $e_k, e^*_k$ are sides in adjacent lozenges
$C'_1, C'_2$. 
%Let $f$ be a generator of the stabilizer of 
%this pair of lozenges. 
Then the push off method and the barrier method of adjacent lozenges
can be used to show the conical limit point property for $p = \eta(x)$.
This shows that we can assume that $(v^k_n)$ converges only to $e_k$.

The induction works that for all $k \leq k_1 -1$. For $k = k_1 -1$ this means
that $(v^{k_1-1}_n)$ converges to $e_{k_1-1} = c$ and
$v^{k_1-1}_n$ intersects $e_{k_1} = d$. Since there is only one leaf
in $g_n(\mathcal P)$ intersecting $d$ transversely, it follows
that for all $n$ big $v^{k_1-1}_n = l_0$. This contradicts
the fact that $(v^{k_1-1}_n)$ converges to $e_{k_1-1}$. 

This finishes the proof of Subcase 1. 

\vskip .1in
\noindent
{\bf {Subcase 2}} $-$ Suppose now that the sequence
$(F_n)$ converges to $e_2$.

Here we let $v^2_n = F_n$ and as in Subcase 1 we let the
subsequent leaves in $g_n(\mathcal P)$ be denoted by
$v^i_n$ where $i \geq 3$.
Suppose first that no subsequence of $(v^3_n)$ is constant
and equal to $e_3$.  
It follows that the leaves $v^3_n$ have to intersect $e_2$.
Suppose that $(v^3_n \cap e_2)$ does not escape in $e_2$ and
converge to a point $y$. This contradicts the fact that
$(g_n(x))$ converges to $z_0$ which is either an ideal
point of $e_3$ or $e_3$ separates it from $\oos(y)$ and
$\oou(y)$.
The fact that $v^3_n$ intersects $e_2$ also implies that
$e_3$ does not separate $g_n(x)$ from $g_n(C)$.
It follows that $z = z_0$ is an ideal
point of $e_3$ and $d = e_3$. Recall that $l_0$ intersects
$d$ transversely.
% Since $(v^2_n)$ converges to $e_2$, the
%setup implies that $(v^3_n)$ converges to $e_3 = d$. 
It follows
that $v^3_n$ intersects $l_0$ for $n$ big enough.
As seen in the proof of the Subclaim of Claim 7 this
leads to a contradiction that $(g_n(x))$ cannot converge to $z_0$.

The other possibility is that up to subsequence $(v^3_n) = e_3$.
Then proof proceeds exactly as in Subcase 1.

\vskip .2in
The other option here is that $w$ is not an ideal point of
the region $U$. This forces $e_1$ to be singular with singular
point denoted by $z_1$. 
It follows that the sequence
$(F_n)$ converges to $\oou(z_1)$. 
In addition it follows that all $e_i$ are periodic leaves and let
$z_i$ be the periodic point in $e_i$. These points are connected
by lozenges $C_i$ from $z_i$ to $z_{i+1}$. This is depicted in
figure  \ref{qg15} associated with the proof of the Special subcase.
Notice however that unlike in the Special subcase, $F_n \not =
\oou(z_1)$, but we only have that the sequence $(F_n)$
converges to $\oou(z_1)$. But similar to the proof of the Special
subcase it now follows that the following leaves in $g_n(\mathcal P)$
(for each $n$) converge to $\oos(z_2)$ when $n$ converges to infinity.
The following leaves in $g_n(\mathcal P)$ converge to $\oou(z_3)$
and so on. Proceeding in this manner we eventually get that a sequence
with leaves in $g_n(\mathcal P)$ converges to the leaf $d$. So 
for $n$ big enough the next leaf in $g_n(\mathcal P)$ cannot be
$l_0$, and this contradicts the fact that $l_0$ is a leaf in
$g_n(\mathcal P)$ for every $n$. This finishes the proof
in this case.

\vskip .1in
This finishes the analysis of Subcase 2 and hence of Case B.1.

The final case to be analyzed is the following:

\vskip .1in
\noindent
{\bf {Case B.2}} $-$ The leaf $e_1$ does not separate $g_n(C)$ from $g_n(x)$.

Since $e_1$ makes a perfect fit with $e_2$ and $g_n(C)$ is in the same component $U$ 
of $(\oo - e_1)$ limiting
on $g_n(x)$, then the canonical path $g_n(C)$ to $g_n(x)$ need not intersect $e_1$. But it has to intersect
$e_2$.
Suppose that up to subsequence $e_2$ is a leaf of $g_n(\mathcal P)$. Then
$e_i$ is a leaf of $g_n(\mathcal P)$ for $2 \leq i < k_1$.
If $d = e_{k_1}$ separates $g_n(x)$ from $g_n(C)$ then $d$ is the next leaf
in $g_n(\mathcal P)$ and as in the proof of Case B.1, Subcase 1 this contradicts
the fact that $l_0$ is in $g_n(\mathcal P)$.
If on the other hand $d = e_{k_1}$ does not separate $g_n(x)$ from $g_n(C)$, then
$k_0 = k_1$ and $d$ is the last leaf of $\mathcal T$. 
Since $g_n(x)$ converges to an ideal point of $d$ and $l_0$ intersects
$d$ transversely, this leads to a contradiction as in the proof of
Case B.1, Subcase B.1.

We assume from now on that $e_2$ is not a leaf of $g_n(\mathcal P)$ for any $n$.
 Let $v^1_n$ be the 
 leaf in the path $g_n(\mathcal P)$ intersecting $e_2$.
%,  see figure \ref{stuck2}, a.
The leaves $\{ v^1_n \}$ intersect $e_2$ so they belong to the same foliation as $e_1$. 
%For simplicity assume that
%$e_1$ is an unsable leaf as in figure \ref{stuck2}, b.
The previous leaf in the canonical path $g_n(\mathcal P)$ are denoted by $v^0_n$.
The previous leaf to those in $g_n(\mathcal P)$ are denoted by
$v^{-1}_n$.

\vskip .1in
\noindent
{\bf {Claim 8}} $-$ The sequence $(v^{-1}_n)$ converges to $w$.

These leaves are in the same foliation as $e_1$. Since they separate $g_n(x)$
from $g_n(C)$ and $e_1$ does not separate $g_n(x)$ from $g_n(C)$,
then for $n$ big $v^{-1}_n$ is contained in $U$.
In addition since each has a line leaf separating $g_n(C)$ from $g_n(x)$ and
$(g_n(C))$ converges to $w$, then $v^{-1}_n$ has ideal points $y_n$ so
that $(y_n)$ converges to $w$.

Suppose that $(v^{-1}_n)$ does not converge to $w$. Then it converges to a leaf
$t$ which has an ideal point $w$. Notice that it cannot converge to $e_1$ as
$v^{-1}_n$ does not intersect $e_2$. Rather it is the leaves $v^1_n$ which
intersect $e_2$. Therefore $t, e_1$ are non separated and hence periodic.
There are two adjacent lozenges $C_1, C_2$ with sides contained in $t, e_1$
respectively. In addition since $e_1$ and $e_2$ make a perfect fit,
then there is a third lozenge $C_3$ with sides in $e_1$ and $e_2$. 
All 3 lozenges $C_1, C_2, C_3$ intersect a common (stable or unstable)
leaf.
We now employ the arguments in the beginning of the analysis of Case B.1,
the first subcase (option 1 in step 1, the special subcase) and arrive at a contradition
concerning how the leaf $l_0$ in $g_n(\mathcal P)$ is obtained.

We conclude that $(v^{-1}_n)$ converges to $w$.
This proves claim 8.

\begin{figure}
\centeredepsfbox{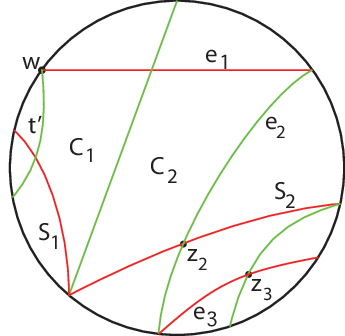}
\caption{Case B.2. Here we produce leaves $t', e_1, e_2$ forming 
a double perfect fit. After the first step the argument is very similar to the proof of
the Special case. Here again stable leaves are red and unstable leaves are green.}
\label{qg16}
\end{figure}

\vskip .1in
Now consider the sequence $(v^0_n)$. 
Since $(v^{-1}_n)$ converges to $w$ then no matter
which option in Step 1 is used to produce $v^0_n$ from $v^{-1}_n$, it 
follows that $v^0_n$ has an ideal point $r_n$ with $(r_n)$ converging to
$w$.
Suppose first that $(v^0_n)$ does not escape compact sets.
Then it converges to a leaf $t'$ with ideal point $w$ and we may assume
that $t'$ makes a perfect fit with $e_1$. Notice that $t'$ and
$e_1$ are in distinct foliations.
Then $e_1$ does not separate $e_2$ from $t'$ so $t', e_1$ and $e_2$ form
a double perfect fit. By Proposition \ref{doublefit} there are adjacent
lozenges $C_1, C_2$ so that $C_1$ has sides contained in $t'$ and $e_1$, 
whereas $C_2$ has sides contained in $e_1, e_2$.
Also there are two leaves $S_1, S_2$ which are non separated from 
each other and $S_1$ contains a side of $C_1$, $S_2$ contains a side of
$C_2$. Since $(v^0_n)$ converges to $t'$ and $t'$ intersects $S_1$, 
it follows that the next step in $g_n(\mathcal P)$ is given
by the leaves $S_1, S_2$. This means that there is $j_n$ with
$S_1 = g_n(u_{j_n})$,  $S_2 = g_n(l_{j_n})$ and $g_n(l_{j_n})$ is
obtained using Option 2 of Step 1. In this case the next leaf
$g_n(l_{j_n})$ is constant with $n$.
As in the proof of claim 8, using the Special subcase, this leads to a contradiction to
how the leaf $l_0$ of $g_n(\mathcal P)$ is obtained.
See also fig. \ref{qg16} where we use the periodic
orbits $z_i$ in $e_i$ as in the Special subcase.

We conclude that $(v^0_n)$ not escaping compact sets leads to
a contradiction.
Therefore $(v^0_n)$ escapes compact sets. Since it has ideal
points which converge to $w$, it follows that
$(v^0_n)$ converges to $w$. 

Recall that $v^1_n$ intersects $e_2$ transversely and they are
contained in $U$. Using that $(v^0_n)$ converges to $w$ it 
now follows that $(v^1_n)$ converges to $e_1$. 
This is the same as the Intermediate conclusion as in
Case B.1. From here on the proof is exactly as in Case B.1.

This finishes the analysis of Case B.2 and therefore of Case B. 

This finishes the proof of the uniform convergence theorem.
\end{proof}

This immediately implies part of theorem D:

\begin{corollary}{}{}
If $\Phi$ is a bounded pseudo-Anosov flow in $M$ then $\pi_1(M)$ is
Gromov hyperbolic and the flow ideal boundary $\rr$ is $\pi_1(M)$
equivariantly homeomorphic to the Gromov ideal boundary.
\label{gromhyp}
\end{corollary}

\begin{proof}{}
Theorem \ref{uniform} shows that $\pi_1(M)$ acts as a uniform convergence group
on $\rr$. The space $\rr$ is homeomorphic to a sphere, and hence metrisable
and perfect. Under these conditions Bowditch \cite{Bow1} proved that $\pi_1(M)$
is Gromov hyperbolic and $\si$ is $\pi_1(M)$ equivariantly homeomorphic
to $\rr$.
\end{proof}

\noindent
{\bf {Notation}} $-$
The $\pi_1(M)$ equivariant homeomorphism from $\si$ to $\rr$ is unique as 
its value is prescribed on fixed points of covering translations and they
are dense in $\si$ \cite{Gr,Gh-Ha}. This homeomorphism 
is denoted by $\tau: \si \rightarrow \rr$.

\section{Flow ideal compactification and equivalent models of compactification of $\mi$}

\noindent
{\bf {The flow ideal compactification $\mi \cup \rr$.}}

Once and for all fix a section $\nu_0: \oo \rightarrow \mi$.
Also fix a homemorphism $\nu_1$ from $\rrrr$ to $(-1,1)$
which is monotone increasing. We define a homeomorphism

$$\nu_2: \ \mi \ \rightarrow \ \oo \times (-1,1)$$

\noindent
as follows.
Given $x$ in $\mi$ let $y = \Theta(x)$ and let $t(x)$ be the
unique real number so that $x = \wwp_{t(x)} (\nu_0(y))$. Let
now $t_1(x) = \nu_1(t(x))$. Define 

$$\nu_2(x) \ \ = \ \ (\Theta(x),t_1(x))$$

\noindent
It is immediate that $\nu_2$ is a homeomorphism.
We have an induced action of $\pi_1(M)$ on $\theta \times (-1,1)$
given by conjugation of the action on $\mi$ by $\nu_2$.
Now consider the space

$$Z \ = \ \cd \times [-1,1] \ = \ (\oo \cup \partial \oo) \times [-1,1],$$

\noindent
with the product topology. We are using the topology in 
$\cd$ that was previously defined, making it into
a closed disk, and $[-1,1]$ has
the standard topology.
In particular

$$\partial Z \ = \ \cd \times \{ -1, 1 \} \ \ \cup \ \
 \partial \oo \times [-1,1].$$

\noindent
We previously defined the topological 
quotient of $\partial Z$ by the equivalence relation  $\simeq$ to be the space $\rr$. 
The quotient map is denoted by $\zeta: Z = \partial (\cd \times [-1,1])
\rightarrow \rr$. Recall  the other quotient map $\eta: \partial \oo \rightarrow \rr$.
Define
a quotient map

$$\psi : \ \cd \times [-1,1] \ \rightarrow \ \mi \cup \rr$$

\noindent
as follows:

if $x$ is in $\oo \times (-1,1)$ let $\psi(x) = \nu^{-1}_2(x)$,

if $x$ in $\partial Z$ let $\psi(x) = \zeta(x)$.

\vskip .08in
\noindent
{\bf {Topology in $\mi \cup \rr$}} $-$ 
The map $\psi$ is surjective and induces a quotient topology in $\mi \cup \rr$.

\vskip .1in
We define an action of $\pi_1(M)$ on $\mi \cup \rr$ by glueing the
actions of $\pi_1(M)$ on $\mi$ by covering translations and the action
on $\rr$. At this point we only know that the individual actions are 
continuous. Instead of proving continuity of the joint action
it can be immediately derived from Theorem \ref{equivar} to be proved later.

\vskip .1in
In the last section we proved that if $\Phi$ is a bounded
pseudo-Anosov flow then $\pi_1(M)$ is Gromov hyperbolic
and that $\rr$ is $\pi_1(M)$ equivariantly homeomorphic to the
Gromov ideal boundary $\partial \pi_1(M) = \si$ by a homeomorphism
$\tau: \si \rightarrow \rr$. We now define

$$f: \ \mi \cup \si \ \rightarrow \ \mi \cup \rr$$

\noindent
as follows:

$${\rm if} \ \ x \in \mi, \ \ {\rm let} \ \ f(x) = x$$

$${\rm if} \ \ x \in \si, \ \ {\rm let} \ \ f(x) = \tau(x).$$

\noindent
We will show that the map $f$ is a $\pi_1(M)$ equivariant homeomorphism.

\begin{lemma}{}{}
The space $\mi \cup \rr$ is Hausdorff.
\end{lemma}

\begin{proof}{}
Let $x, y$ distinct points in $\mi \cup \rr$. Suppose first that one of them,
say $x$ is in $\mi$. Let $W_0$ be an open neighborhood of $x$ in $\mi$ with 
$y$ not in $W_0$. Let $W_1$
be an open neighborhood of $x$ in $\mi$ with

$$x \ \in \ W_1 \ \subset \ \overline W_1 \ \subset \ W_0,$$

\noindent
where the closure is in $\mi$. The set $\mi \cup \rr - \overline W_1$ is
open in $\mi \cup \rr$ because
its inverse image in $\cd \times [-1,1]$ is
$(\cd \times [-1,1]) - \overline W_1$, which is open in
$\cd \times [-1,1]$. Clearly $y \in (\mi \cup \rr) - \overline W_1$
and $x \in W_1$ (notice $W_1$ is also open in $\mi \cup \rr$) hence
$x, y$ have disjoint neighborhoods.

From now on suppose that both $x, y$ are in $\rr$.
Put a metric $d$ in $\cd \times [-1,1]$ compatible with the topology.
Since $\rr$ is homeomorphic to a 2-sphere, which is Hausdorff,
there are open sets $V_0, V_1$ of $\rr$ with  disjoint closures
in $\rr$ and 
$x \in V_0,  \  y \in V_1$. Notice that the induced topology from 
$\mi \cup \rr$ in $\rr$ is the same topology we defined before in $\rr$.
Let $Z_i = \zeta^{-1}(V_i)$ which are open sets in 
$\partial (\cd \times [-1,1])$, and they have disjoint closures.
Therefore in the metric $d$ of $\cd \times [-1,1]$ there is 
$\epsilon > 0$ so that if $a \in Z_0$ and $b \in Z_1$ then
$d(a,b) > 2 \epsilon$. For each $t \in Z_i$ choose
a ball $B(t,r_t)$ in $\cd \times [-1,1]$ of radius $r_t > 0$
centered at $t$ so that also

$$B(t,r_t) \cap \partial (\cd \times [-1,1]) \ \ \subset
\ \ Z_i, \ \ r_t < \epsilon  \ \ \ \ \ \ \ (1)$$

\noindent
Let 

$$Y_i \ \ = \ \ \bigcup_{t \in Z_i} B(t,r_t)$$

\noindent
Clearly $Y_0, Y_1$ are open sets in $\cd \times [-1,1]$ and
they are disjoint. In addition they are saturated by the
quotient map, that is $Y_i = \psi^{-1}( \psi(Y_i))$.
This is because condition (1) implies that 

$$Y_i \ \cap \ \partial (\cd \times [-1,1]) \ \ = \ \ Z_i$$

\noindent
In this case let $W_i = \psi(Y_i) \subset \mi \cup \rr$. The last property
implies that $W_0, W_1$ are open in $\mi \cup \rr$.
Since they are disjoint and $x \in W_0, \ y \in W_1$, then
$x, y$ are separated in $\mi \cup \rr$. 

This finishes the proof that $\mi \cup \rr$ is Hausdorff.
\end{proof}

\begin{theorem}{}{}
The bijection $f: \mi \cup \si  \rightarrow \mi \cup \rr$ 
is a $\pi_1(M)$ equivariant homeomorphism.
\label{equivar}
\end{theorem}

\begin{proof}{}
From the definition of $f$ it is obvious that it is a bijection
and in addition that it is $\pi_1(M)$ equivariant.
In addition $\mi \cup \si$ is compact (Gromov compactification)
and $\mi \cup \rr$ is a Hausdorff space. By elementary
point set topology \cite{Mu}  it suffices to show that $f$ is continuous,
which will imply that $f$ is a homeomorphism.
Finally since $\mi \cup \si$ is a metric space, it is first
countable. Therefore in order to check that $f$ is continuous,
we only need to check that $f$ is sequentially continuous.

Before we prove that $f$ is continuous notice we have  the least amount of properties
of $\mi \cup \rr$ in order to obtain the result. In particular at this
point we do not know that $\mi \cup \rr$ is compact, or that the 
action of $\pi_1(M)$ on $\mi \cup \rr$ is continuous. Both
of these properties can be
proved directly, but  for the sake of brevity this is deduced immediately from
Theorem \ref{equivar}.

\vskip .05in
Let then $(x_n)$ be a sequence in $\mi \cup \si$ converging to 
$x$ in $\mi \cup \si$. If $x$ is in $\mi$ then $x_n$ is in $\mi$
for $n$ big, so $f(x_n) = x_n$ (in $\mi$) converges to
$x = f(x)$ in $\mi$ and hence in $\mi \cup \rr$.

Therefore suppose from now on that $x$ is in $\si$.
We will show that any subsequence of $(x_n)$ has a further
subsequence $(x_{n_k})$  so that $(f(x_{n_k}))$ converges
 to $f(x)$ in $\mi \cup \rr$. This proves continuity of $f$ 
at $x$. 
Consider first a sequence $(x_n)$  in $\si$ with $x_n \rightarrow x$ in $\mi \cup \si$.
Since the topology of $\mi \cup \si$ induces the Gromov topology in 
$\si$, it follows that $(x_n)$ converges to $x$ in $\si$. So 
$f(x_n)$ converges to $f(x)$ in $\rr$ and it follows that
$(f(x_n))$ converges to $f(x)$ in $\mi \cup \rr$.

Therefore we may assume that $x_n$ is in $\mi$ for all $n$.
We will take subsequences at will.
First, up to subsequence there are $g_n \in \pi_1(M)$ and
$y_n$ in $\mi$ with $g_n(y_n) = x_n$ and $y_n \rightarrow y$ in
$\mi$. This is because $M$ is compact. Up to another subsequence
assume that $(g_n)$ is a sequence of distinct elements.

Since $(g_n)$ is a sequence of distinct elements of $\pi_1(M)$
and $\pi_1(M)$ acts as a convergence group in $\mi \cup \si$ \cite{Gr,Gh-Ha},
there is a subsequence (still denoted by $(g_n)$ by abuse of notation)
so that $(g_n)$ has a source $a \in \si$ and a sink $b \in \si$
for the action on $\mi \cup \si$.
In fact the same is true for the action on $\mi \cup \rr$, but this 
takes quite a bit longer to prove. We will first prove that $f$ is
a homeomorphism which also implies this fact from convergence group
property of $\pi_1(M)$ on $\mi \cup \si$.
Notice however that we proved in Theorem \ref{conver1} that $\pi_1(M)$
acts as a convergence group on $\rr$.

Notice that by the convergence group properties on
$\mi \cup \si$, the sink for the sequence $(g_n)$ is
$x$, and so it follows that $b = x$. Therefore the sink
for the sequence $(g_n)$ acting on $\mathcal R$ is
$f(b) = f(x)$.

We previously explained that $f(a), f(b)$ are the source/sink pair for
the sequence $(g_n)$ acting on $\rr$.
Let 

$$A \ = \ \eta^{-1}(f(a)), \ \ \ \ B \ = \ \eta^{-1}(f(b))$$

\noindent
Then $A, B$ are the source/sink {\underline {sets}} for the sequence
$(g_n)$ acting on $\partial \oo$.

Let now 

$$v_n \ = \ \Theta(y_n), \ \ \ \ \ \ 
v \ = \ \Theta(y)$$

\noindent
These are points in $\oo$.
Let $z_0 = \ee(\partial \oos(v)), \ z_1 = \ee(\partial \oou(v))$. We will think of these as
subsets of $\partial \oo$ and also as points in $\rr$.

\vskip .1in
\noindent
{\underline {Case 1}} $-$ Suppose first that $z_0, z_1$ are both
not equal to $f(a)$.

Then in $\partial \oo$ the sequence
$(g_n(\partial \oos(v)))$ converges to a subset of $B$
and so by the convergence group property the sequence 
$(g_n(\partial \oos(v_n)))$ 
converges to a subset of $B$ \ (1). This is because for $n$
big the equivalence class of $\partial \oos(v_n)$ under $\sim$

$${\rm is \ in \ a \ fixed \ compact \ set }  \ \ 
C \ \   {\rm of} \ \  \ \partial \oo - \eta^{-1}(f(a)) \ = \ \partial \oo - A. \ \ \ \ \ \ (2)$$

\noindent
Similarly $(g_n(\partial \oou(v_n)))$ converges to a subset
of $B$ \ (2). We claim that the sequence $(g_n(v_n))$ (contained in 
$\oo$) cannot have a convergent
subsequence in $\oo$. Suppose that up to subsequence
$(g_n(v_n))$ converges to $w$ in $\oo$. Then the escape lemma
and  (2) above show that 
$\partial \oos(w)$ is related to $B$. In addition the
escape lemma and property \ (2) shows that 
$\partial \oou(w)$ is also related to $B$. 
This would imply that $\partial \oos(w)$ is related to $\partial \oou(w)$.
This is impossible
by Proposition \ref{noident}.
We conclude that this cannot happen.

Since $(g_n(v_n))$ escapes compact sets in $\oo$, then the escape lemma
and (2) show that in $\oo \cup \partial \oo$  this sequence can only converge to points in $B$.
Then in $\cd \times [-1,1]$ the sequence  $(g_n((\Theta^{-1}(v_n)))$ converges to vertical
stalks in $B \times [-1,1]$. Therefore in $\mi \cup \rr$ the sequence
$(g_n(y_n)) = (x_n) = (f(x_n))$ converges to $f(b) = f(x)$. This finishes the analysis in
this case.

\vskip .1in
\noindent
{\underline {Case 2}} $-$ Suppose that either $\eta(\partial \oos(v)) = f(a)$
\  or \  $\eta(\partial \oou(v)) = f(a)$.

Without loss of generality suppose that $\eta(\partial \oos(v)) = f(a)$.
 Therefore $\eta(\partial \oou(v))$ is not equal  to $f(a)$.
By the arguments in the analysis of case 1, it follows that
$(g_n(\partial \oou(v_n)))$ converges to a subset of 
$B = \eta^{-1}(f(b))$.

Suppose that a subsequence of $(g_n(v_n))$ escapes compact
sets in $\oo$ and converges in $\oo \cup \partial \oo$.
By the escape lemma, this sequence converges to a point
in $B$. Then $(f(x_n))$ converges to $f(b) = f(x)$ as in Case 1.

Therefore from now on we assume that 
the sequence $(g_n(v_n)) = w_n$ converges to $w$ which is a point
in $\oo$. Recall that $\Theta(x_n) = w_n$.

In addition $g_n(\partial \oou(v_n)) = \oou(w_n)$ converges to a subset of
$\partial \oou(w)$. Therefore $\partial \oou(w)) \subset B$ is contained
in the sink set for the sequence $(g_n)$ acting on $\partial \oo$.

Let $z_n \in \mi$ with $\Theta(z_n) = w_n$ and $(z_n)$ converging
to $z$, hence $\Theta(z) = w$.
Let $D$ be a small closed disk in $\mi$, transverse to $\wwp$,
where we may assume that $D$ 
contains all $\{ z_n \}$ and $z$ in its interior. Let $t_n \in 
\rrrr$ with 

$$x_n \ \ = \ \ \wwp_{t_n} (z_n)$$

\noindent
Since $(x_n)$ converges to $x$ in $\mi \cup \si$,
it follows that the sequence of absolute values $(|t_n|)$ converges
to infinity. Suppose by way of contradiction that there is a
subsequence, still denoted by $(t_n)$ so that $t_n \rightarrow +\infty$.
Assume furthermore that all $y_n$ are in a sector
of $y$.

Consider an arbitrary point $\rho$ in 
$\wlu(y)$ very near $y$ and in the half leaf of $\wlu(y)$  that intersects
$\wls(y_n)$ for some $n$ sufficiently big. 
For the time being fix
the point $\rho$.
For each such $\rho$ we will define a sequence $(c^{\rho}_n)$ as follows.
%This sequence in fact depends on the value of $c$, but 
%for the sake of notational simplicity they will all be denoted
%by the same notation $(c_n)$. 
The points $c^{\rho}_n$ are points
in $\wlu(y_n) \cap \wls(\rho)$ so that the sequence $(c^{\rho}_n)$ converges to 
a point very near $y$.
Then 

$$d(c^{\rho}_n,y_n) \ \ = \ \ d(g_n(c^{\rho}_n),g_n(y_n)) \ \ = \ \ d(g_n(c^{\rho}_n),x_n)$$

\noindent
is very small.  Recall that $(t_n)$ converges to $+\infty$ so $z_n$ is lot
flow backwards of $x_n$. Since  $g_n(c^{\rho}_n)$ is in $\wlu(x_n)$, then there is a unique point

$$s_n \ \ \ = \ \ \ \wwp_{\rrrr} (g_n(c^{\rho}_n)) \ \cap \ D$$

\noindent
By assumption $t_n \rightarrow +\infty$ so 
 $s_n$ is
a point which is extremely flow backwards of $g_n(c^{\rho}_n)$. It follows that
$(d(s_n,z_n)))$ converges to zero.
In particular

$$(g_n(\wls(c^{\rho}_n))) \ \ = \ \ (\wls(g_n(c^{\rho}_n))) \ \ = \ \ ((\wls(s_n)))$$

\noindent
converges to $\wls(z)$. Therefore $(g_n(\partial \oos(\rho)))$
converges to a subset of $\ee(\partial \oos(w))$. The set $\mathcal E(\partial \oos(w))$ is
not equivalent to $B$ under $\sim$ because $B$ contains $\partial \oou(w)$.

In addition this is true for any $\rho$ near enough $y$ in $\wlu(y)$.
Since only finitely many ideal points of stable leaves can be
equivalent to each other by the relation $\sim$, it now follows that 
$\ee(\partial \oos(w))$ is the sink for the sequence $(g_n)$ acting
on $\partial \oo$. This contradicts the fact that $\ee(\partial \oou(w))$
is inequivalent to $B$ and the last one
is the sink set of $(g_n)$ acting on $\partial \oo$.

\vskip .1in
This contradition shows that no subsequence of $(t_n)$ can
converge to $+\infty$. It follows that $(t_n)$ converges
to $-\infty$. Then in $\cd \times [-1,1]$, the sequence
$(x_n)$ converges to $z \times \{ -1 \}$. Therefore in $\mi \cup \rr$,
the sequence $(x_n)$ converges to $\eta(\ee(\partial \oou(w)))$ 
(recall that $\Theta(z) = w$). Since 

$$\partial \oou(w) \ \ \sim \ \ B$$

\noindent
it follows that $(x_n)$ converges to $f(b)$ in $\mi \cup \rr$. 
This finishes the proof that $f$ is continuous.

As explained before this implies that $f$ is a $\pi_1(M)$ 
equivariant homeomorphism. 
As a consequence $\pi_1(M)$ acts continuously on $\mi \cup \rr$ and $\pi_1(M)$
acts as a convergence group
on $\mi \cup \rr$.
\end{proof}

\section{Quasigeodesic pseudo-Anosov flows}

\begin{theorem}{}{}
Let $\Phi$ be a bounded pseudo-Anosov flows which is not topologically
equivalent to a suspension Anosov flow. Then 
$\pi_1(M)$ is Gromov hyperbolic and $\Phi$ is a quasigeodesic
pseudo-Anosov flow.
\end{theorem}

\begin{proof}{}
The first part was proved in Corollary \ref{gromhyp}.  We prove the second part.
We first show the following facts in $\mi \cup \rr$ using some very
easy properties of this compactification.

\vskip .1in
\noindent
1) For any $x$ in $\mi$, there is a limit
$\lim_{t \rightarrow +\infty} \wwp_t(x)$ in $\mi \cup \rr$.

Using the homeomorphism $\nu_2$ between $\mi$ and $\oo \times (-1,1)$ we
analyse this in $\cd \times [-1,1]$ and then in $\mi \cup \rr$.
Considering $x$ in $\oo \times (-1,1)$ the limit
$\lim_{t \rightarrow 1} (\Theta(x),t)$ obviously exists in $\cd \times [-1,1]$
$-$ it is just $(\Theta(x),1)$.
Since $\psi: \cd \times [-1,1] \rightarrow \mi \cup \rr$ is continuous
property 1) is true in $\mi \cup \rr$. 
We denote the limit above by  $x_+$ which is a point in  $\rr$.
Clearly this is independent of the point in $\gamma = \wwp_{\rrrr}(x)$
and is also denoted by $\gamma_+$. Similarly
$\lim_{t \rightarrow -\infty} \wwp_t(x)$ exists in $\mi \cup \rr$ for
any $x$ in $\mi$. This is denoted by $x_-$ or $\gamma_-$. 

\vskip .1in
\noindent
2) For any $x$ in $\mi$ then $x_+ \not = x_-$.

Let $\gamma = \wwp_{\rrrr} (x)$. If $\gamma_- = \gamma_+$
then $\ee(\partial \oos(x))$ is equivalent to $\ee(\partial \oou(x))$.
This was proved not to be true in Proposition \ref{noident}
 so $\gamma_+ \not = \gamma_-$.

\vskip .1in
\noindent
3) The map $P_+: \mi \rightarrow \rr$ given by
$P_+(x) = x_+$ is continuous and similarly for $P_-:\mi \rightarrow \rr$.

There is a map $PP_+: \oo \times (-1,1) \rightarrow \partial (\cd \times [-1,1])$
given by $PP_+(y) = (y,1)$. This map is obviously continuous. 
The map $P_+$ is equal to $\psi \circ (PP_+)$, hence $P_+$ is continuous.

\vskip .1in
Theorem \ref{equivar} shows that $\mi \cup \si$ is
homeomorphic to $\mi \cup \rr$ via the map $f$. Therefore the
same 3 properties 1), 2), 3) also hold in $\mi \cup \si$.
By a result of the author and Lee Mosher \cite{Fe-Mo}  this implies that
$\Phi$ is a quasigeodesic flow.

This finishes the proof of the theorem.
\end{proof}

\begin{corollary}{}{}
Let $\Phi$ be a pseudo-Anosov flow in $M$ with $\pi_1(M)$
Gromov hyperbolic. Then $\Phi$ is quasigeodesic if and only 
if $\Phi$ is a bounded pseudo-Anosov flow.
\end{corollary}

\section{Concluding remarks} 

\noindent
{\bf {Examples of quasigeodesic flows with freely homotopic orbits}}

Mosher \cite{Mo5} showed that a Reebless finite depth foliation in $M^3$
hyperbolic admist an almost transverse pseudo-Anosov flow $\Phi$.
The author and Mosher proved that these flows are quasigeodesic
\cite{Fe-Mo}. Mosher proved a large class of these flows have
freely homotopic orbits. In fact for each $n$ one can construct
examples with free homotopy classes of size at least $n$ which
are quasigeodesic. 
This implies that the Main theorem is optimal.

\vskip .1in
\noindent
{\bf {Unbounded pseudo-Anosov flows}}

Suppose $\Phi$ is a pseudo-Anosov flow which is not bounded 
and which is not topologically conjugate to a suspension Anosov flow.
In section 5 we saw how to construct chains of perfect fits of infinite
length. One open question is whether one can also construct
chains of free homotopies of infinite length. Perhaps more refined
perturbation methods will yield this result.

Suppose in addition that $M$ is atoroidal. There are examples,
for example $\rrrr$-covered Anosov flows in hyperbolic $3$-manifolds \cite{Fe2}.
One very important open question is the following: are these the
only examples? In other words if $\Phi$ is unbounded in $M$ atoroidal
does it imply that $\Phi$ is an $\rrrr$-covered Anosov flow?

In particular if the answer to this  very general question is true, it
 will imply that
there are many examples of quasigeodesic Anosov flows in hyperbolic
manifolds. This is because there are many examples of 
Anosov flows in hyperbolic $3$-manifolds which are not $\rrrr$-covered \cite{Fe4}.
 Up to now there are no examples of quasigeodesic Anosov flows
in hyperbolic $3$-manifolds. On the other hand it would be very
interesting also to construct counterexamples to the general question.

\vskip .1in
\noindent
{\bf {Applications to asymptotic properties of foliations}}

Suppose that $\fol$ is a Reebless foliation in $M^3$ hyperbolic \cite{No}. 
Reebless roughly means that leaves are $\pi_1$ injective. \cite{No}.
These foliations are very common in $3$-manifolds \cite{Ga1,Ga2,Ga3,Ro1,Ro2,Ro3}.
In addition Candel \cite{Ca} proved that there is a metric in $M$ so that
every leaf of $\fol$ is a hyperbolic surface. Therefore each leaf $L$ of the
lifted foliation $\fn$ to $\mi$ can be thought of as an embedding

$$i : L \cong {\bf H}^2 \ \rightarrow \ \mi \cong {\bf H}^3$$

\noindent
The {\em continuous extension question} asks whether for each $L$ the
map $i$ extends to a continuous map $i : L \cup \partial L \rightarrow \mi \cup \si$,
where $\partial L$ is the ideal circle of $L$. The continuos extension 
property has been proved in the following settings: \ 1) $\fol$ is a finite
depth foliation. See \cite{Ga1,Ga2,Ga3} for finite depth foliations and \cite{Fe8}
for the proof of the continuous extension property in this case; \ 2) 
$\rrrr$-covered foliations \cite{Fe9}. Here $\rrrr$-covered means that the
leaf space of $\fn$ is homeomorphic to the reals; \ 3) Foliations with one sided
branching \cite{Fe9}. One sided branching means that the leaf space
of $\fn$ has non Hausdorff behavior only in one direction $-$ positive or negative.

In the generic two sided branching case the conjecture is that $\fol$ admits
an almost transverse pseudo-Anosov flow \cite{Th5,Th6,Cal3}.
Calegari \cite{Cal3} has done substantial work in this direction $-$
he produced two essential laminations which are transverse to $\fol$. Perhaps
these can be used to produce a pseudo-Anosov flow $\Phi$ almost transverse to $\fol$.
If this flow $\Phi$ is quasigeodesic, this would imply the continuous extension property
for $\fol$ because of the following result. The author proved in \cite{Fe8} that for a general $\fol$ in
$M$ hyperbolic, if there is a {\em {quasigeodesic}} pseudo-Anosov flow almost
transverse to $\fol$, then $\fol$ has the continuous extension property.

{\footnotesize
{
\setlength{\baselineskip}{0.01cm}

\noindent
Florida State University

\noindent
Tallahassee, FL 32306-4510

and 

\noindent
Princeton University

\noindent
Princeton, NJ 08544-1000

}
}

\end{document}